\documentclass[reqno]{amsart}
\usepackage{amsfonts}
\usepackage{mathrsfs}
\usepackage{color}
\usepackage{cite}
\usepackage{amscd}
\usepackage{latexsym}
\usepackage{amsfonts}
\usepackage{graphicx}
\usepackage{CJK,indentfirst,amsmath,amsfonts,amssymb,amsthm,cite,cases, subeqnarray,setspace}
\allowdisplaybreaks[4]

\begin{document}
\title[NLS approximation for the Euler-Poisson equation]
{Justification of the NLS approximation\\ for the Euler-Poisson equation}
\author{Huimin Liu and Xueke Pu}

\address{Huimin Liu \newline
Faculty of Applied Mathematics, Shanxi University of Finance and Economics, Taiyuan 030006, P.R.China}
\email{hmliucqu@163.com}

\address{Xueke Pu \newline
School of Mathematics and Information Science, Guangzhou University, Guangzhou, 510006, P.R.China} \email{puxueke@gmail.com}

\thanks{This work is supported by NSFC (11871172).}

\begin{abstract}
The nonlinear Schr\"{o}dinger (NLS) equation can be derived as a formal approximation equation describing the envelopes of slowly modulated spatially and temporarily oscillating wave packet-like solutions to the ion Euler-Poisson equation. In this paper, we rigorously justify such approximation by giving error estimates in Sobolev norms between exact solutions of the ion Euler-Poisson system and the formal approximation obtained via the NLS equation.  The justification consists of several difficulties such as the resonances and loss of regularity,  due to the quasilinearity of the problem. These difficulties are overcome by introducing normal form transformation and cutoff functions and carefully constructed energy functional of the equation.
\end{abstract}

\maketitle \numberwithin{equation}{section}
\newtheorem{proposition}{Proposition}[section]
\newtheorem{theorem}{Theorem}[section]
\newtheorem{lemma}[theorem]{Lemma}
\newtheorem{remark}[theorem]{Remark}
\newtheorem{hypothesis}[theorem]{Hypothesis}
\newtheorem{definition}{Definition}[section]
\newtheorem{corollary}{Corollary}[section]
\newtheorem{assumption}{Assumption}[section]
\section{\textbf{Introduction}}
\setcounter{section}{1}\setcounter{equation}{0}

The nonlinear Schr\"{o}dinger (NLS) equation plays important roles in describing approximately slow modulations in time and space of an underlying spatially and temporarily oscillating wave packet in more complicated systems, such as the equations describing surface water waves \cite{Z} or the Euler-Poisson system describing the motion of plasma composed of ions and isothermal electrons \cite{SI}. Indeed, even at the linearized level, there are electron waves, ion acoustic waves in the Euler-Poisson system. In the current paper, we consider the NLS approximation for the amplitude of the ion oscillation in the Euler-Poisson system
\begin{subequations}\label{equation1}
\begin{numcases}{}
\partial_{t}n+\partial_{x}(nv)=0,\\
\partial_{t}v+v\partial_{x} v=-\partial_{x}{P(n)}-\partial_{x}\phi,\\
\partial_{x}^{2}\phi=e^{\phi}-n,
\end{numcases}
\end{subequations}
where $(x, t)\in\mathbb{R}\times\Bbb R^+$, $n$ is the ion density, $v$ is the ion velocity, the self-consistent field $-\partial_{x}\phi$ satisfies the Poisson equation, {the ion pressure $P$ satisfies $P(n)=\ln n$.} Taking $\rho=n-1$ and
\begin{equation*}
\begin{split}
\begin{pmatrix} \rho \\ v \end{pmatrix}=\epsilon\Psi_{NLS}+\mathcal{O}(\epsilon^{2}),
\end{split}
\end{equation*}
with
\begin{equation}
\begin{split}\label{LS}
\epsilon\Psi_{NLS}=\epsilon A\big(\epsilon(x-c_{g}t),\epsilon^{2}t\big)e^{i(k_{0}x-\omega_{0}t)}\varphi(k_{0})+c.c.,
\end{split}
\end{equation}
the nonlinear Schr\"odinger equation (NLS) can be derived for the complex amplitude $A$,
\begin{equation}
\begin{split}\label{A}
\partial_{T}A=i\nu_{1}\partial_{X}^{2}A+i\nu_{2}A|A|^{2},
\end{split}
\end{equation}
where $T=\epsilon^{2}t\in\mathbb{R}$ is the slow time scale and $X=\epsilon(x-c_{g}t)\in\mathbb{R}$ is the slow spatial scale and coefficients $\nu_{j}=\nu_{j}(k_{0})\in \mathbb{R}$ with $j\in\{1,2\}$. In the above modulation approximation, $0<\epsilon\ll1$ is a small perturbation parameter, $\omega_{0}>0$ is the basic temporal wave number associated to the basic spatial wave number $k_{0}>0$ of the underlying temporally and spatially oscillating wave train $e^{i(k_{0}x-\omega_{0}t)}$, $c_{g}$ is the group velocity and `c.c.' denotes the complex conjugate. The NLS equation is derived in order to describe the slow modulations in time and in space of the wave train $e^{i(k_{0}x-\omega_{0}t)}$ and the time and space scales of the modulations are $\mathcal{O}(1/\epsilon^{2})$ and $\mathcal{O}(1/\epsilon)$, respectively.  For the Euler-Poisson equation \eqref{equation1}, the basic spatial wave number $k=k_{0}$ and the basic temporal wave number $\omega=\omega_{0}$ are related via the following linear dispersion relation
\begin{equation}
\begin{split}\label{equation3}
\omega(k)=k\sqrt{\frac{2+k^{2}}{1+k^{2}}}=k\widehat{q}(k), \ \widehat{q}(k)=\sqrt{\frac{2+k^{2}}{1+k^{2}}}.
\end{split}
\end{equation}
From the dispersion relation, the group velocity $c_{g}=\frac{\partial w}{\partial k}(k_{0})$ of the wave packet can be found, and  $\varphi(k_{0})=\big(1,-\widehat{q}(k_{0})\big)^{T}$ is an eigenvector for the linearized equation. Our ansatz leads to waves moving to the right. To obtain waves moving to the left, $-\omega_{0}$ and $c_{g}$ have to be replaced with $\omega_{0}$ and $-c_{g}$, respectively.  The NLS equation is a completely integrable Hamiltonian system, which can be solved explicitly with the help of inverse scattering method (see \cite{A}  for example). We also note that in the year of 1968, V.E. Zakharov \cite{Z} derived the nonlinear Schr\"{o}dinger equation from the equations of hydrodynamics for an ideal fluid with a free surface of a deep fluid.

In this paper, we will consider such a NLS approximation for the ion Euler-Poisson equation. But before we state the main result concerning such an approximation in this paper, we first make some literature review on some important aspects of the Euler-Poisson equation and the modulational approximation.

In recent years, many efforts have been made to study the global existence and approximation of solutions to the Euler-Poisson equation for ions as well as the Euler-Poisson equation for electrons. We did not mention the Euler-Poisson for electrons above, but it is worth mentioning some recent results on this equation. The Euler-Poisson equations for ions as well as electrons are both important PDE models arising in plasma physics, share some basic difficulties as other important fluid models,  and are far from being well understood.  Guo firstly obtained global irrotational solutions with small velocity for the 3D electron fluid, based on the Klein-Gordon effect \cite{Guo98}. For the two dimensional electron fluid in Euler-Poisson system, Ionescu and Pausader obtained the global stability of the constant equilibrium solution \cite{IP13}. Jang considered the global solution with spherical symmetry initial data \cite{Jang12}. Furthermore, Jang, Li and Zhang obtained the smooth global solutions \cite{JLZ14}. Finally, Li and Wu solved the Cauchy problem for the two dimensional electron Euler-Poisson system \cite{LW14}. Recently, Guo, Han and Zhang \cite{GHZ15} finally completely settled this problem of global existence and proved that no shocks form for the 1D Euler-Poisson system for electrons. For the Euler-Poisson equation for ions, Guo and Pausader \cite{GP11} constructed global smooth irrotational solutions with small amplitude for ion dynamics in $\Bbb R^3$. For the long wave approximation, Guo and Pu \cite{GP14} established rigorously the KdV limit for the ion Euler-Poisson system in 1D for both cold and hot plasma cases, where the electron density satisfies the classical Maxwell-Boltzmann law. This result was generalized to the higher dimensional cases, and the 2D Kadomtsev-Petviashvili-II (KP-II) equation and the 3D Zakharov-Kuznetsov (ZK) equation are derived under different scalings \cite{P}. Almost at the same time, \cite{LLS} also established the Zakharov-Kuznetsov equation in 3D from the Euler-Poisson system. Recently, the authors in the present paper \cite{LP,LP17} obtained rigorously the quantum KdV limit in 1D and the KP-I and KP-II equations in 2D for the Euler-Poisson system for cold as well as hot plasma taking quantum effect into account, where the electron equilibrium is given by a Fermi-Dirac distribution. Han-Kwan \cite{HK} also introduced a long wave scaling for the Vlasov-Poisson equation and derived the KdV equation in 1D, the KP-II equation in 2D and the Zakharov-Kuznetsov equation in 3D using the modulated energy method. These long wavelength limit results are interesting, since the oscillatory solution of the KdV equation is nothing but a solution of the NLS equation in the small wave number region with frequency induced from the dispersive relation. In this sense, derivation of the NLS equation in the current paper is a more interesting problem in the context of Euler-Poisson system \eqref{equation1}.

Showing error estimates for the NLS approximation of dispersive wave systems with quadratic terms is not a trivial task, while in the absence of quadratic terms a simple application of Gronwall's inequality yields the desired result \cite{KS}. If semilinear quadratic terms are present in the original system they can be removed by the so called normal form transformation, a near identity change of variables, if the eigenvalues of the linearized problem satisfy a non-resonance condition \cite{K}. Note that the normal form transformation was first introduced to eliminate semilinear quadratic terms in quadratic nonlinear Klein-Gordon systems to study global solutions in \cite{S85}. The normal form method has been widely used and extended in decay estimate and to extend the existence time for small solutions in some quasilinear systems in recent years. The interested readers may refer to \cite{W} for almost global existence for the 2D infinite depth full water wave equation , Germain and Masmoudi for the global existence of the 3D Euler-Maxwell system \cite{GM}, Alazard, Delort and Szeftel for the 2D gravity water waves and for the nonlinear Klein-Gordon equation \cite{AD,De,D09}, Hunter, Ifrim and Tataru for the water waves and the Burgers-Hilbert equation \cite{H,H16,IT}, Ionescu and Pusateri for the water wave system in 2D \cite{IP15,IP16,IP18}. Besides, the normal form transformation was also developed to study the long wavelength approximation and the modulational approximation of nonlinear systems.  For the long wavelength limit, one can refer to D\"ull \cite{D12} and Schneider and Wayner \cite{SW,SW02} for the water wave problem. For the modulational approximation, one can refer to D\"ull \cite{D1} for the quasilinear Klein-Gordon eqution and Schneider \cite{S1998, S98} for the hyperbolic systems and the KdV equation.

The quasilinearity causes two basic problems in common in justifying the NLS approximation. One is loss of derivatives, which finally makes it not easy to close energy estimates. The other one is  resonances (violating the non-resonance condition \eqref{equation1,1}), and due to the continuous spectrum of the linearized problem of many physics systems such as the ion-acoustic wave problem studied in this paper and the water wave problems, the eigenvalues are continuous functions and it is hard to avoid resonances. Thus some attempts were made to overcome these two main difficulties as well as many other difficulties induced by structures of a particular quasilinear system, to justify the NLS approximation. As far as the authors know, only a few special examples are known to work up to date. One is the quasilinear dispersive wave system in which the right-hand side only loses half a derivative, in which case the elimination of the quadratic terms is still possible with the help of a normal form transformation. In this case, the transformed system loses only one derivative in total and can be handled with the Cauchy-Kowalevskaya theorem \cite{D,S,SW}. It was realized that it is not easy to close the energy estimates once the order of total loss of derivatives is greater than one. Recently, justification of the NLS approximation for a quasilinear Klein-Gordon equation has been obtained in \cite{D1} by using the normal form transformation to define a new energy to simplify the error estimates, in which the nonlinearities lose one derivative causing that the transformed system to lose two derivatives in total. But luckily there is no resonance in such a quasilinear Klein-Gordon model \cite{D1}, i.e. the non-resonance condition is satisfied. Besides, the NLS equation has been justified for the 2D water wave problem in the special case of zero surface tension and infinite depth \cite{T} by finding a transformation adapted to this problem which allows to eliminate all quadratic terms. See also the two dimensional hyperbolic NLS approximation for the 3D water wave problem \cite{T15}. Another example is in the context of the Korteweg-de Vries equation where the result can be obtained by applying a Miura transformation \cite{S1} to eliminate the dangerous quadratic terms.


Due to the quasilinearity of the system and the dispersive relation \eqref{equation3}, we face all the following principal difficulties in justifying the NLS approximation, namely,
\begin{enumerate}
\item the quasilinearity,
\item a quadratic nonlinearity,
\item trivial resonance at the wave number $k=0$ as well as nontrivial resonance at $k=k_{0}$,
\item loss of one derivative in the quadratic nonlinear term that finally causes the transformed system to lose two derivatives.
\end{enumerate}

It is worth highlighting that the water waves problem treated by Totz-Wu \cite{T} present similar difficulties, where the authors found some special transformation adapted to their problem whereas we used the normal form transformation and modified energy functional in this paper. Precisely, to handle the quadratic nonlinearity, we may need to introduce the normal form transformation, but then the quasilinearity causes the quadratic nonlinearity to lose one derivative and eventually causes the loss of two derivatives in total in the transformed system. Therefore it is not easy to close the energy estimates. On the other hand, both trivial and non-trivial resonances occur, i.e. the non-resonance condition is not satisfied.  {Hence, the validity of the NLS approximation to the ion Euler-Poisson equation is far from a trivial problem.} In this paper, we will consider such an interesting problem and justify such a modulation approximation to the Euler-Poisson equation for ions \eqref{equation1}.

The main result of this paper is the following
\begin{theorem}\label{Thm1}
Fix $s_{A}\geq6$. Then for all $k_{0}\neq0$ and for all $C_{1}, \ T_{0}>0$, there exist $C_{2}>0, \  \epsilon_{0}>0$ such that for all solutions $A\in C([0,T_{0}],H^{s_{A}}(\mathbb{R},\mathbb{C}))$ of the NLS equation \eqref{A} with
\begin{equation*}
\begin{split}
\sup_{T\in[0,T_{0}]}\big\|A(\cdot,T)\big\|_{H^{s_{A}}}(\mathbb{R},\mathbb{C})\leq C_{1},
\end{split}
\end{equation*}
the following holds. For all $\epsilon\in(0,\epsilon_{0})$, there are solutions
\begin{equation*}
\begin{split}
\begin{pmatrix} n-1 \\ v \end{pmatrix}\in \Big(C\big([0,T_{0}/\epsilon^{2}],H^{s_{A}}(\mathbb{R},\mathbb{R})\big)\Big)^{2},
\end{split}
\end{equation*}
of the ion Euler-Poisson equation \eqref{equation1} that satisfy
\begin{equation*}
\begin{split}
\sup_{t\in[0,T_{0}/\epsilon^{2}]}\Big\|\begin{pmatrix} n-1 \\ v \end{pmatrix}-\epsilon\Psi_{NLS}(\cdot,t)\Big\|_{H^{s_{A}}(\mathbb{R},\mathbb{R})^{2}}
\leq C_{2}\epsilon^{3/2},
\end{split}
\end{equation*}
\end{theorem}
\noindent where $\varphi(k_{0})=\big(1,-\widehat{q}(k_{0})\big)^{T}$ is given in the definition of $\varepsilon\Psi_{NLS}$, $\widehat{q}$ comes from the equation \eqref{equation3}.

\begin{remark}
First, compared with the solution $(n-1,v)$ and the approximation $\epsilon\Psi_{NLS}$, which are both of order $\mathcal{O}(\epsilon)$ in $L^{\infty}$, the error of order $\mathcal{O}(\epsilon^{3/2})$ is small enough such that the dynamics of the NLS equation can be found in the ion Euler-Poisson system \eqref{equation1}. Secondly, we note that the Fourier transform of $\epsilon\Psi_{NLS}$ is sufficiently strongly concentrated around the wave numbers $\pm k_{0}$, hence by using a modified approximation that has compact support in Fourier space but differs only slightly from $\epsilon\Psi_{NLS}$, the smoothness of the error bound can be made equal to the assumed smoothness of the amplitude. Finally, in the following proof of Theorem \ref{Thm1}, we always assume that $s_A$ is an integer to simplify the proof, although it can be generalized to all real numbers $s_A\geq6$.
\end{remark}

\begin{remark}
In the momentum equation for ion-Euler-Poisson equation \eqref{equation1}, we choose the ion pressure $P=\ln n$. Indeed, the result in this paper can be generalized to the general $\gamma$-law of the ion pressure $P$, i.e., when $P(n)=n^{\gamma}$ for $\gamma\geq1$.
\end{remark}
We would like to end the introduction by giving the structure of this paper here. {In Section 2 we derive the NLS equation formally, and then we estimate the terms that remain after inserting the approximation into \eqref{equation6}. In Section 3 we outline the basic ideas for the diagonalized system \eqref{equation7} of the Euler-Poisson equation for ions when justifying the NLS approximation.} In Section 4 we perform the normal form transformation that is invertible and does not lose any derivatives for $|k|\leq \delta$, and then we present the transformed error equations for $(\mathcal{R}^{0},R^{1})$. In Section 5 we construct our energy and perform the error estimates to prove Theorem \ref{Thm1}.

\textbf{Notation}. We denote the Fourier transform of a function $u\in L^{2}(\mathbb{R},\mathbb{K})$, with $\mathbb{K}=\mathbb{R}$ or $\mathbb{K}=\mathbb{C}$ by
\begin{equation*}
\begin{split}
\widehat{u}(k)=\frac{1}{2\pi}\int_{\mathbb{R}}u(x)e^{-ikx}dx.
\end{split}
\end{equation*}
Let $H^{s}(\mathbb{R},\mathbb{K})$ be the space of functions mapping from $\mathbb{R}$ into $\mathbb{K}$ for which the norm
\begin{equation*}
\begin{split}
\|u\|_{H^{s}(\mathbb{R},\mathbb{K})}=\Big(\int_{\mathbb{R}}|\widehat{u}(k)|^{2}(1+|k|^{2})^{s}dk\Big)^{1/2}
\end{split}
\end{equation*}
is finite. Usually we write $L^{2}$ and $H^{s}$ instead of $L^{2}(\mathbb{R},\mathbb{R})$ and $H^{s}(\mathbb{R},\mathbb{R})$. We use the space $L^{p}(m)(\mathbb{R},\mathbb{K})$ defined by $u\in L^{p}(m)(\mathbb{R},\mathbb{K})$ such that $\sigma^{m}u\in L^{p}(\mathbb{R},\mathbb{K})$, where $\sigma(x)=(1+x^{2})^{1/2}$. Finally, we write $A\lesssim B$, if $A\leq CB$ for a constant $C>0$, and $A=\mathcal{O}(B)$, if $|A|\lesssim B$.

\section{\textbf{Derivation of the NLS approximation and estimates for the residual}}
NLS type equation has been given formally for the one dimensional motion of plasma composed of cold ions and isothermal electrons in \cite{SI}. In the following we will obatin NLS equation for the Euler-Poisson equation \eqref{equation1}.
In order to isolate \eqref{equation1} into linear, quadratic and higher order terms, we can rewrite the normalized Euler-Poisson system \eqref{equation1} as follows,
\begin{subequations}\label{equation4}
\begin{numcases}{}
\partial_{t}\rho+\partial_{x}v+\partial_{x}(\rho v)=0,\\
\partial_{t}v+\partial_{x}\rho+\partial_{x}\phi+v\partial_{x} v-\partial_{x}\frac{\rho^{2}}{2}=-\partial_{x}\big[\ln(1+\rho)-\rho+\frac{\rho^{2}}{2}\big],\\
\rho=(1-\partial_{x}^{2})\phi+\frac{\phi^{2}}{2}+\big[e^{\phi}-1-\phi-\frac{\phi^{2}}{2}\big].
\end{numcases}
\end{subequations}
For small $\rho$, the last line defines an inverse operator $\rho\mapsto\phi(\rho)$. We further expand this inverse operator up to third order as
\begin{equation}
\begin{split}\label{equation5}
\phi(\rho)=(1-\partial_{x}^{2})^{-1}\rho-\frac{1}{2}(1-\partial_{x}^{2})^{-1}\big[(1-\partial_{x}^{2})^{-1}\rho\big]^{2}+\mathcal{M}(\rho),
\end{split}
\end{equation}
where $\mathcal{M}$ satisfies some good properties. We note that $v\partial_{x}v=\partial_{x}({v^{2}}/{2})$ and we can rewrite the above system \eqref{equation4} as
\begin{equation}
\begin{split}\label{equation6}
\partial_{t}&\begin{pmatrix}\rho\\ v\end{pmatrix}+\begin{pmatrix}0&\partial_{x}\\ \partial_{x}(1-\partial_{x}^{2})^{-1}+\partial_{x}&0\end{pmatrix}\begin{pmatrix}\rho\\ v\end{pmatrix}\\
=&\begin{pmatrix}-\partial_{x}(\rho v)\\-\partial_{x}\frac{v^{2}}{2}+\partial_{x}\frac{\rho^{2}}{2}
+\frac{1}{2}\partial_{x}(1-\partial_{x}^{2})^{-1}\big[(1-\partial_{x}^{2})^{-1}\rho\big]^{2}-\partial_{x}\mathcal{M}(\rho)
-\partial_{x}\big(\ln(1+\rho)-\rho+\frac{\rho^{2}}{2}\big)\end{pmatrix}.
\end{split}
\end{equation}
Let $S=\begin{pmatrix}1& 1\\-q(|\partial_{x}|)& q(|\partial_{x}|)\end{pmatrix}$ and $\begin{pmatrix}\rho\\ v\end{pmatrix}=S\begin{pmatrix}U_{1}\\ U_{-1}\end{pmatrix}$.
{We define the operator $\Omega$ by $\widehat{\Omega u}(k)=i\omega(k)\widehat{u}(k)$.} We can diagonalize the linear part of the equation \eqref{equation6} as
\begin{equation}
\begin{split}\label{equation7}
\partial_{t}U_{j}=j\Omega U_{j}+Q_{j}(U,U)+N_{j},
\end{split}
\end{equation}
where $j\in\{1,-1\}$ and the quadratic term $Q_{j}$ and the high order term $N_{j}$ take the form
\begin{equation}
\begin{split}\label{equation8}
Q_{j}&=-\frac{\partial_{x}(\rho v)}{2}-j\frac{\partial_{x}}{4q(|\partial_{x}|)}\big[-|v|^{2}+\rho^{2}+(1-\partial_{x}^{2})^{-1}[(1-\partial_{x}^{2})^{-1}\rho]^{2}\big],\\
N_{j}&=j\frac{\partial_{x}}{2q(|\partial_{x}|)}\big[\ln(1+\rho)-\rho+\frac{\rho^{2}}{2}+\mathcal{M}(\rho)\big].
\end{split}
\end{equation}
Plugging $\rho=U_{1}+U_{-1}$ and $v=-q(|\partial_{x}|)(U_{1}-U_{-1})$ into $Q_{j}$ and $N_{j}$, we now compute the Fourier transform of $U_{j}$ as
\begin{equation}
\begin{split}\label{equation9}
\partial_{t}\widehat{U}_{j}
=&ij\omega(k)\widehat{U}_{j}
+\frac{ik}{2}\int_{\mathbb{R}}\widehat{q}(m)\big(\widehat{U}_{1}(k-m)
+\widehat{U}_{-1}(k-m)\big)\big(\widehat{U}_{1}(m)-\widehat{U}_{-1}(m)\big)dm\\
&+\frac{ijk}{4\widehat{q}(k)}\int\widehat{q}(k-m)\widehat{q}(m)\big(\widehat{U}_{1}(k-m)
-\widehat{U}_{-1}(k-m)\big)\big(\widehat{U}_{1}(m)-\widehat{U}_{-1}(m)\big)dm\\
&-\frac{ijk}{4\widehat{q}(k)}\int\big(1+\frac{1}{\langle k\rangle^{2}}\frac{1}{\langle k-m\rangle^{2}}\frac{1}{\langle m\rangle^{2}}\big)\big(\widehat{U}_{1}(k-m)+\widehat{U}_{-1}(k-m)\big)\big(\widehat{U}_{1}(m)+\widehat{U}_{-1}(m)\big)dm\\
&+\frac{ijk}{2\widehat{q}(k)}\sum_{n\geq3}(-1)^{n+1}\frac{1}{n}(\widehat{U}_{1}+\widehat{U}_{-1})^{\ast n}
+\frac{ijk}{2\widehat{q}(k)}\mathcal{M}(\widehat{U}_{1}+\widehat{U}_{-1})(k),
\end{split}
\end{equation}
with $j\in\{1,-1\}$ and $\langle k\rangle=\sqrt{(1+k^{2})}$. In order to derive the NLS equation as an approximation equation for system \eqref{equation9}, we make the ansatz
\begin{equation}
\begin{split}\label{equation10}
\begin{pmatrix} U_{1} \\ U_{-1} \end{pmatrix}=\epsilon\widetilde{\Psi}
=\epsilon\widetilde{\Psi}_{1}+\epsilon\widetilde{\Psi}_{-1}+\epsilon^{2}\widetilde{\Psi}_{0}
+\epsilon^{2}\widetilde{\Psi}_{2}+\epsilon^{2}\widetilde{\Psi}_{-2},
\end{split}
\end{equation}
with
\begin{equation*}
\begin{split}
&\epsilon\widetilde{\Psi}_{\pm1}=\epsilon\widetilde{A}_{\pm1}\big(\epsilon(x-c_{g}t),\epsilon^{2}t\big)E^{\pm1}\begin{pmatrix} 1 \\ 0 \end{pmatrix},\\
&\epsilon^{2}\widetilde{\Psi}_{0}=\begin{pmatrix}\epsilon^{2}\widetilde{A}_{01}\big(\epsilon(x-c_{g}t),\epsilon^{2}t\big) \\ \epsilon^{2}\widetilde{A}_{02}\big(\epsilon(x-c_{g}t),\epsilon^{2}t\big)\end{pmatrix},\\
&\epsilon^{2}\widetilde{\Psi}_{\pm2}=\begin{pmatrix}\epsilon^{2}\widetilde{A}_{(\pm2)1}\big(\epsilon(x-c_{g}t),\epsilon^{2}t\big)E^{\pm2} \\ \epsilon^{2}\widetilde{A}_{(\pm2)2}\big(\epsilon(x-c_{g}t),\epsilon^{2}t\big)E^{\pm2}\end{pmatrix},
\end{split}
\end{equation*}
where $0<\epsilon\ll1$, $E^{\pm j}=e^{\pm ij(k_{0}x-\omega_{0}t)}$, $\omega_{0}=\omega(k_{0})$, $\widetilde{A}_{-j}=\overline{\widetilde{A}}_{j}$ and $\widetilde{A}_{-j\ell}=\overline{\widetilde{A}}_{j\ell}$. The ansatz leads to waves that move to the right. If one replaces in the above ansatz the vector $(1,0)^{T}$ by $ (0,1)^{T}$, $-\omega_{0}$ by $\omega_{0}$ and $c_{g}$ by $-c_{g}$, it leads to waves that move to the left.

We insert our ansatz \eqref{equation10} into the system \eqref{equation9} and then replace the dispersion relation $\omega=\omega(k)$ in all terms of the form $\omega\widetilde{A}_{j}E^{j}$ or $\omega\widetilde{A}_{j\ell}E^{j}$ by their Taylor expansions around $k=jk_{0}$. By equating the coefficients of the $\epsilon^{m}E^{j}$ to zero, we find that the coefficients of $\epsilon , \ \epsilon^{2} E^{0}$ and $\epsilon^{2} E^{1}$ vanish identically due to the definition of $\omega$ and $c_{g}$.

For $\epsilon^{2} E^{2}$ we obtain
\begin{equation*}
\begin{split}
\big(-2\omega_{0}+\omega(2k_{0})\big)\widetilde{A}_{21}=\gamma_{21}\big(\widetilde{A}_{1}\big)^{2},\\
\big(-2\omega_{0}-\omega(2k_{0})\big)\widetilde{A}_{22}=\gamma_{22}\big(\widetilde{A}_{1}\big)^{2},
\end{split}
\end{equation*}
where the coefficients $\gamma_{2\ell}\in\mathbb{R}$ for $\ell=1,2$. From the explicit form of $\omega(k)$, we see $-2\omega_{0}-\omega(2k_{0})\neq0$, and hence $\widetilde{A}_{2\ell}$ are well-defined in terms of $(\widetilde{A}_{1})^{2}$.

For $\epsilon^{3}E^{0}$ we obtain
\begin{equation*}
\begin{split}
\big(c_{g}-(\partial_{k}\omega)(0)\big)\partial_{X}\widetilde{A}_{01}=\gamma_{31}\partial_{X}(\widetilde{A}_{1}\widetilde{A}_{-1}),\\
\big(c_{g}+(\partial_{k}\omega)(0)\big)\partial_{X}\widetilde{A}_{02}=\gamma_{32}\partial_{X}(\widetilde{A}_{1}\widetilde{A}_{-1}),
\end{split}
\end{equation*}
where the coefficients $\gamma_{3\ell}\in\mathbb{R}$. Recall $c_{g}=(\partial_{k}\omega)(k_{0})$, then $c_{g}\pm(\partial_{k}\omega)(0)\neq0$, and the $\widetilde{A}_{0\ell}$ are well-defined in terms of $\widetilde{A}_{1}\widetilde{A}_{-1}$.

For $\epsilon^{3}E^{1}$ we obtain
\begin{equation*}
\begin{split}
\partial_{T}\widetilde{A}_{1}=\frac{i}{2}\partial_{k}^{2}\omega(k_{0})\partial_{X}^{2}\widetilde{A}_{1}+g_{1},
\end{split}
\end{equation*}
where $g_{1}$ is a sum of multiples of $\widetilde{A}_{1}\widetilde{A}_{0\ell}$ and $\widetilde{A}_{-1}\widetilde{A}_{2\ell}$. From the above steps we obtain algebraic relations such that $\widetilde{A}_{0\ell}$ and $\widetilde{A}_{2\ell}$ can be expressed in terms of $\widetilde{A}_{1}\widetilde{A}_{-1}$ and $(\widetilde{A}_{1})^{2}$, respectively. Eliminating $\widetilde{A}_{0\ell}$ and $\widetilde{A}_{2\ell}$, then gives the NLS equation
\begin{equation}
\begin{split}\label{NLS}
\partial_{T}\widetilde{A}_{1}=i\frac{\partial_{k}^{2}\omega(k_{0})}{2}\partial_{X}^{2}\widetilde{A}_{1}
+i\nu_{2}(k_{0})\widetilde{A}_{1}\big|\widetilde{A}_{1}\big|^{2},
\end{split}
\end{equation}
with some $\nu_{2}(k_{0})\in\mathbb{R}$.

Considering the residual
\begin{equation}
\begin{split}\label{equat11}
Res_{U}(\epsilon\widetilde{\Psi})=\begin{pmatrix} Res_{U_{1}}(\epsilon\widetilde{\Psi}) \\ Res_{U_{-1}}(\epsilon\widetilde{\Psi}) \end{pmatrix},
\end{split}
\end{equation}
which contains all terms that do not cancel after inserting ansatz \eqref{equation10} into system \eqref{equation9}. To prove the approximation property of the NLS equation \eqref{NLS}, the resdual term $Res_{U}(\epsilon\widetilde{\Psi})$ needs to be small enough such that the error term $R$ is of order $\mathcal{O}(1)$ for the time scaling $\mathcal{O}(1/\epsilon^{2})$. We take the following standard way to modify $\epsilon\widetilde{\Psi}$ as $\epsilon\Psi$ such that the resdual term $Res_{U}(\epsilon\Psi)$ small enough. Firstly, extend the above approximation $\epsilon\widetilde{\Psi}$ to its higher order terms. Secondly, restrict the modified approximation in Fourier space to small neighborhoods of  a finite number of integer multiples of the basic wave number $k_0>0$, by some cutoff function. By such a modification, the approximation will not change too much, but will lead to a simpler control of the error and make the approximation an analytic function.

Since $\pm\omega(mk_{0})\neq\pm m\omega(k_{0})$ for all integers $m\geq2$, we can proceed analogously as in \cite{D} to replace $\epsilon\widetilde{\Psi}$
 by a new approximation $\epsilon\Psi$ of the form
 \begin{equation}
\begin{split}\label{equation122}
\epsilon\Psi=\epsilon\Psi_{1}+\epsilon\Psi_{-1}+\epsilon^{2}\Psi_{p},
\end{split}
\end{equation}
where
\begin{align*}
\epsilon\Psi_{\pm1}=&\epsilon\psi_{\pm1}\begin{pmatrix} 1 \\ 0 \end{pmatrix} =\epsilon A_{\pm1}\big(\epsilon(x-c_{g}t),\epsilon^{2}t\big)E^{\pm1}\begin{pmatrix} 1 \\ 0 \end{pmatrix},\\
\epsilon^{2}\Psi_{p}=&\begin{pmatrix} \epsilon^{2}\psi_{p_{1}} \\ \epsilon^{2}\psi_{p_{-1}} \end{pmatrix}
=\epsilon^{2}\Psi_{0}+\epsilon^{2}\Psi_{2}+\epsilon^{2}\Psi_{-2}+\epsilon^{2}\Psi_{h},\\
\epsilon^{2}\Psi_{0}=&\begin{pmatrix} \epsilon^{2}\psi_{01} \\ \epsilon^{2}\psi_{02} \end{pmatrix}
=\begin{pmatrix} \epsilon^{2}A_{01}(\epsilon(x-c_{g}t),\epsilon^{2}t) \\ \epsilon^{2}A_{02}\big(\epsilon(x-c_{g}t),\epsilon^{2}t\big) \end{pmatrix} ,\\
\epsilon^{2}\Psi_{\pm2}=&\begin{pmatrix} \epsilon^{2}\psi_{(\pm2)1} \\ \epsilon^{2}\psi_{(\pm2)2} \end{pmatrix}
=\begin{pmatrix} \epsilon^{2}A_{(\pm2)1}\big(\epsilon(x-c_{g}t),\epsilon^{2}t\big)E^{\pm2}
    \\ \epsilon^{2}A_{(\pm2)2}\big(\epsilon(x-c_{g}t),\epsilon^{2}t\big) E^{\pm2} \end{pmatrix} ,\\
\epsilon^{2}\Psi_{h}=&\sum_{j=\pm1, n=1,2,3}\begin{pmatrix} \epsilon^{1+n}A_{j1}^{n}\big(\epsilon(x-c_{g}t),\epsilon^{2}t\big)E^{j}
    \\ \epsilon^{1+n}A_{j2}^{n}\big(\epsilon(x-c_{g}t),\epsilon^{2}t\big) E^{j} \end{pmatrix} \\
    &+\sum_{j=\pm2, n=1,2}\begin{pmatrix} \epsilon^{2+n}A_{j1}^{n}\big(\epsilon(x-c_{g}t),\epsilon^{2}t\big)E^{j}
    \\ \epsilon^{2+n}A_{j2}^{n}\big(\epsilon(x-c_{g}t),\epsilon^{2}t\big) E^{j} \end{pmatrix} \\
    &+\sum_{n=1,2}\begin{pmatrix} \epsilon^{2+n}A_{01}^{n}\big(\epsilon(x-c_{g}t),\epsilon^{2}t\big)
    \\ \epsilon^{2+n}A_{02}^{n}\big(\epsilon(x-c_{g}t),\epsilon^{2}t\big)  \end{pmatrix} \\
    &+\sum_{j=\pm3, n=0,1}\begin{pmatrix} \epsilon^{3+n}A_{j1}^{n}\big(\epsilon(x-c_{g}t),\epsilon^{2}t\big)E^{j}
    \\ \epsilon^{3+n}A_{j2}^{n}\big(\epsilon(x-c_{g}t),\epsilon^{2}t\big) E^{j} \end{pmatrix} \\
    &+\sum_{j=\pm4}\begin{pmatrix} \epsilon^{4}A_{j1}\big(\epsilon(x-c_{g}t),\epsilon^{2}t\big)E^{j}
    \\ \epsilon^{4}A_{j2}\big(\epsilon(x-c_{g}t),\epsilon^{2}t\big) E^{j} \end{pmatrix} ,
\end{align*}
where $A_{-j}=\overline{A}_{j}$ and $A_{-j\ell}=\overline{A}_{j\ell}$ have compact support in Fourier space for all $0<\epsilon\ll1$. Then, exactly as in Section 2 of \cite{D}, the following estimates hold for the modified residual.

\begin{lemma}\label{L4}
Let $s_{A}\geq6$ and $\widetilde{A}_{1}\in C\big([0,T_{0}],H^{s_{A}}(\mathbb{R},\mathbb{C})\big)$ be a solution of the NLS equation \eqref{NLS} with
\begin{equation*}
\begin{split}
\sup_{T\in[0,T_{0}]}\big\|\widetilde{A}_{1}\big\|_{H^{s_{A}}}\leq C_{A},
\end{split}
\end{equation*}
then for all $s\geq0$, there exist $C_{Res},C_{\Psi},\varepsilon_{0}>0$ depending on $C_{A}$ such that the following holds for all $\varepsilon\in(0,\varepsilon_{0})$. The approximation $\epsilon\Psi$ defined in \eqref{equation122} exists for all $t\in[0,T_{0}/\epsilon^{2}]$ and satisfies
\begin{subequations}\label{equation12}
\begin{numcases}{}
\sup_{t\in[0,T_{0}/\epsilon^{2}]}\big\|Res_{U}(\epsilon\Psi)\big\|_{H^{s}}\leq C_{Res}\epsilon^{9/2},\label{equation12-1}\\
\sup_{t\in[0,T_{0}/\epsilon^{2}]}\big\|S(\epsilon\Psi)-\epsilon\Psi_{NLS}\big\|_{H^{s_{A}}}\leq C_{\Psi}\varepsilon^{3/2},\label{equation12-2}\\
\sup_{t\in[0,T_{0}/\epsilon^{2}]}\Big(\big\|\widehat{\Psi}_{\pm1}\big\|_{L^{1}(s+1)(\mathbb{R},\mathbb{C})}
+\big\|\widehat{\Psi}_{p}\big\|_{L^{1}(s+1)(\mathbb{R},\mathbb{C})}\Big)\leq C_{\Psi}.\label{equation12-3}
\end{numcases}
\end{subequations}
\end{lemma}
The proof of Lemma \ref{L4} is analogous to that of Lemma 2.6 in \cite{D} (see also \cite{D1}). In fact, the first and the third estimates are valid for appropriate constants $C_{Res}$ and $C_{\Psi}$ for all $s>0$, this is a consequence of the fact that our approximation $\epsilon\Psi$ has compact support in Fourier space. Besides, the approximation $\epsilon\Psi$ differs so slightly from the actual NLS approximation $\epsilon(\widetilde{\Psi}_{1}+\widetilde{\Psi}_{-1})$ and higher order asymptotic expansions of the exact solution, which are needed to make the residual $Res_{U}(\epsilon\Psi)$ sufficiently small, that the bounds in \eqref{equation12-1} and \eqref{equation12-2} hold if $s_{A}\geq6$. This is shown by using the estimate
\begin{equation*}
\begin{split}
\big\|(\chi_{[-\delta,\delta]-1})\epsilon^{-1}\widehat{f}(\epsilon^{-1}\cdot)\big\|_{L^{2}(m)}\leq C\epsilon^{m+M-1/2}\|f\|_{H^{m+M}},
\end{split}
\end{equation*}
for all $m,M\geq0$, where $\chi_{[-\delta,\delta]}$ is the characteristic function on $[-\delta,\delta]$.
The bound \eqref{equation12-3} will be used to estimate
\begin{equation*}
\begin{split}
\|\psi_{j}f\|_{H^{s}}\leq C\|\psi_{j}\|_{C_{b}^{s}}\|f\|_{H^{s}}\leq C\big\|\widehat{\psi}_{j}\big\|_{L^{1}(s)(\mathbb{R},\mathbb{C})}\|f\|_{H^{s}},
\end{split}
\end{equation*}
without loss of powers in $\epsilon$ as it would be the case with $\|\psi_{j}\|_{L^{2}(s)(\mathbb{R},\mathbb{C})}$. Moreover, by an analogous argument as in the proof of Lemma 3.3 in \cite{D}, we have
\begin{lemma}\label{L5}
For all $s\geq0$ there exists a constant $C_{\psi}>0$ such that
\begin{equation}
\begin{split}\label{equation13}
\big\|\partial_{t}\widehat{\psi}_{\pm1}+i\omega\widehat{\psi}_{\pm1}\big\|_{L^{1}(s)}\leq C_{\psi}\epsilon^{2}.
\end{split}
\end{equation}
\end{lemma}

\section{\textbf{The basic ideas}}
In order to explain our method to prove Theorem \ref{Thm1}, we use the diagonalized system \eqref{equation7} of the {normalized Euler-Poisson equation \eqref{equation1},}
\begin{equation}
\begin{split}\label{abstract}
\partial_{t}U=\Lambda U+Q(U,U)+N(U),
\end{split}
\end{equation}
where $Q(U,U)$ and $N(U)$ are given by \eqref{equation8}. Recalling that $U=U(x,t)\in\mathbb{R}^{2}$, $x\in\Bbb R$, $t\in\mathbb{R}^+$, $\Lambda$ being a operator whose symbol is a diagonal matrix of the form
\begin{equation*}
\begin{split}
\widehat{\Lambda}(k)=diag\{i\omega(k),-i\omega(k)\}.
\end{split}
\end{equation*}
We write the $j_{1}$-th component of $Q(U,U)$ as
\begin{equation*}
\begin{split}
\widehat{Q}_{j_{1}}(U,U)
=\sum_{j_{2},j_{3}=\pm1}\int \widehat{\eta}_{j_{2},j_{3}}^{j_{1}}(k,k-m,m)\widehat{U}_{j_{2}}(k-m)\widehat{U}_{j_{3}}(m)dm,
\end{split}
\end{equation*}
where $j_{1}\in\{\pm1\}$ and $\widehat{\eta}_{j_{2},j_{3}}^{j_{1}}(k,k-m,m)$ is the kernel function of $\widehat{Q}_{j_{1}}(U,U)$. Let $\widehat{\widetilde{\eta}}_{j_{2},j_{3}}^{j_{1}}(k,k-m,m)$ be the kernel function of $\widehat{N}_{j_{1}}(U)$, according to the equation \eqref{equation9} in Fourier form, we have
\begin{equation}
\begin{split}\label{quu}
\big|\widehat{\eta}_{j_{2},j_{3}}^{j_{1}}(k,k-m,m)\big|, \ \big|\widehat{\widetilde{\eta}}_{j_{2},j_{3}}^{j_{1}}(k,k-m,m)\big|\leq C |k|,
\end{split}
\end{equation}
for any $k,m\in\mathbb{R}$.

We know that $U$ is formally approximated by $\epsilon\widetilde{\Psi}$ according to the Section 2, i.e. the residual
\begin{equation}
\begin{split}\label{resU}
Res(U)=-\partial_{t}U+\Lambda U+Q(U,U)+N(U),
\end{split}
\end{equation}
is small for $U=\epsilon\widetilde{\Psi}$. And the residual can be made arbitrarily small by modifying the formal approximation $\epsilon\widetilde{\Psi}$ (refer to the equation \eqref{equation122} and the Lemma \ref{L4}), i.e. for all $\gamma>0$ there exists a formal approximation $\epsilon\Psi$ that has compact support set comparing to the approximation $\epsilon\widetilde{\Psi}$,
such that
\begin{equation*}
\begin{split}
Res_{U}(\epsilon\Psi)=\mathcal{O}(\epsilon^{\gamma}).
\end{split}
\end{equation*}
{As in Section 2, the modified approximation $\epsilon\Psi$ satisfies}
\begin{equation}
\begin{split}\label{Phic}
&\epsilon\Psi=\epsilon\Psi_{c}+\epsilon^{2}\Psi_{p},\\
&\Psi_{c}:=\Psi_{1}+\Psi_{-1}=(\psi_{1}+\psi_{-1},0)^{\top}=:(\phi_{c},0)^{\top},
\Psi_{p}=(\psi_{p_{1}},\psi_{p_{-1}})^{\top},\\
&\text{\text{supp}}\widehat{\psi}_{\pm1}=\big\{k\mid|k\pm k_{0}|\leq\delta\big\},\\
&\text{supp}\widehat{\psi}_{p_{\pm1}}=\big\{k\mid|k\pm jk_{0}|\leq\delta, \ j=0,\pm2,\pm3,\pm4\big\},
\end{split}
\end{equation}
and
 \begin{equation*}
\begin{split}
\epsilon\Psi-\epsilon\widetilde{\Psi}=\mathcal{O}(\epsilon^{2}),
\end{split}
\end{equation*}
where $\delta>0$ sufficiently small, but independent of $0<\epsilon\ll1$.
In order to prove Theorem \ref{Thm1} we have to estimate the error
 \begin{equation}
\begin{split}\label{before}
\epsilon^{\beta}R=U-\epsilon\Psi,
\end{split}
\end{equation}
to be of order $\mathcal{O}(\epsilon^{\beta})$ for some $\beta>1$ on a time scale $t\in[0,T_{0}/\epsilon^{2}]$, i.e. we have to prove that $R$ is of order $\mathcal{O}(1)$ for all $t\in[0,T_{0}/\epsilon^{2}]$. Inserting \eqref{before} into \eqref{abstract} we find that the error $R$ satisfies
\begin{equation}
\begin{split}\label{error}
\partial_{t}R=&\Lambda R+2\epsilon Q(\Psi_{c},R)+2\epsilon^{2} Q(\Psi_{p},R)+\epsilon^{\beta} Q(R,R)+\epsilon^{2}N(R)+\epsilon^{-\beta}Res_{U}(\epsilon\Psi).
\end{split}
\end{equation}
For our equation, the linear operator $\Lambda$ generates a uniformly bounded semigroup. If $\beta>2$, which we assume henceforth, the terms $\epsilon^{2} Q(\Psi_{p},R)$, $\epsilon^{\beta} Q(R,R)$ and $\epsilon^{2}N(R)$ can be controlled over the relevant time interval. Also, the residual term $\epsilon^{-\beta}Res_{U}(\epsilon\Psi)$ can be made of $\mathcal{O}(\epsilon^{2})$ by choosing the approximation $\epsilon\Psi$ appropriately. However, the remaining linear term $\epsilon Q(\Psi_{c},R)$ can perturb the linear evolution such  that the solutions begin to grow on time scales $\mathcal{O}(\epsilon^{-1})$ and hence we would lose all control over the size of $R$ on the desired time scale $\mathcal{O}(\epsilon^{-2})$.
To show that the error remains small over the desired time intervals $\mathcal{O}(\epsilon^{-2})$, we need to eliminate the quadratic term $2\epsilon Q(\Psi_{c},R)$ from \eqref{error} via a normal form transformation. That is to say, we make a change of dependent variable of the form
\begin{equation}
\begin{split}\label{equation45}
\widetilde{R}_{j_{1}}:=R_{j_{1}}+\epsilon \sum_{j_{2}\in\{\pm1\}}B_{j_{1},j_{2}}(\Psi_{c},R_{j_{2}}), \ j_{1}\in\{\pm1\},
\end{split}
\end{equation}
where
\begin{equation}
\begin{split}\label{equation46}
\widehat{B}_{j_{1}j_{2}}(\Psi_{c},R_{j_{2}})=\int \widehat{b}_{j_{1},j_{2}}(k,k-\ell,\ell)\widehat{\phi}_{c}(k-\ell)\widehat{R}_{j_{2}}(\ell)d\ell,
\end{split}
\end{equation}
where we have used the equation \eqref{Phic} and $\widehat{R}_{j_{2}}$ refers to the $j_{2}$ component of $\widehat{R}$. Careful calculations show that the kernel function $\widehat{b}_{j_{1},j_{2}}$ of the normal form transformation can be written as a quotient whose denominator is
\begin{equation}
\begin{split}\label{15'}
-j_{1}\omega(k)-\omega(k-\ell)+j_{2}\omega(\ell).
\end{split}
\end{equation}
As long as the denominator remains away from zero, such a normal form transformation is well defined. That is to say, a non-resonance condition has to be satisfied:
\begin{equation}
\begin{split}\label{equation1,1}
|j_{1}\omega(k)+\omega(k_{0})-j_{2}\omega(k-k_{0})|>0,
\end{split}
\end{equation}
for $j_{1},j_{2}\in\{\pm1\}$ and for all $k\in \mathbb{R}$ uniformly. It is easy to see that $\omega(k)=k\widehat{q}(k)$ in \eqref{equation3} for the ion Euler-Poisson equation does not satisfy \eqref{equation1,1}. In particular, there is a resonance at the wave number $k=0$ (whenever $j_2=-1$), which is trivial and eventually causes no problems for the definition of the normal form transformation  because the nonlinear term also vanishes linearly at $k=0$ and hence $\widehat{B}_{j_{1}j_{2}}$  of \eqref{equation46} can be well defined for all $|k|\leq\delta$. However, there is always another resonance for the wave number $k=k_{0}$ (for $j_1=-1$), which turns out to be nontrivial. Therefore, we can not take the normal form method of \cite{Z} directly. For this, we introduce a suitable rescaling of the error function $R$ dependent on the wave number, and then to use a number of special normal form transformations to treat such a nontrivial resonance, as did in \cite{S,S98}.

More precisely, we scale the variable $R$ to reflect the fact that the nonlinearity vanishes at $k=0$. For some $\delta>0$ above sufficiently small, but independent of $0<\epsilon\ll1$, define a weight function $\vartheta$ by its Fourier transform
\begin{equation}
\begin{split}\label{v}
\widehat{\vartheta}(k)=\Big\{\begin{matrix} 1\ \ \ \ \ \ \ \ \ \ \ \ \ \ \ \ \ \ \ \ \ \ \ \ \ \ \ \ \text{for} \ \ |k|>\delta, \\ \epsilon+(1-\epsilon)| k|/\delta \ \ \ \ \ \ \ \ \ \ \text{for} \ \ |k|\leq\delta.\end{matrix}
\end{split}
\end{equation}
This makes $\widehat{\vartheta}(k)\widehat{R}(k)$ small at the wave numbers close to zero to reflect the fact that the nonlinearity vanishes at $k=0$.
Rewrite the solution $U$ of \eqref{before} as a sum of the approximation and error, i.e.,
\begin{equation}
\begin{split}\label{equation48}
U=\epsilon\Psi+\epsilon^{\beta}\vartheta R,
\end{split}
\end{equation}
with a $\beta>2$. Here and hereafter, we define $\vartheta R$ by $\widehat{\vartheta R}=\widehat{\vartheta} \widehat{R}$ to avoid writing the convolution $\vartheta\ast R$.

Inserting \eqref{equation48} into \eqref{abstract}, we find that the error $R$ satisfies
\begin{equation}
\begin{split}\label{equation49}
\partial_{t}R=& \Lambda R+2\epsilon \vartheta^{-1}Q(\Psi_{c},\vartheta R)+2\epsilon^{2}\vartheta^{-1}Q(\Psi_{p},\vartheta R)+\epsilon^{\beta} \vartheta^{-1}Q(\vartheta R, \vartheta R)\\
&+\epsilon^{2}\vartheta^{-1}N(\vartheta R)+\epsilon^{-\beta}\vartheta^{-1}Res_{U}(\epsilon\Psi).
\end{split}
\end{equation}
Due to the inequality \eqref{quu} and $\vartheta^{-1}$ is at most of order $\mathcal{O}(1/\epsilon)$, we find that the terms $\epsilon^{2}\vartheta^{-1}Q(\Psi_{p},\vartheta R)$ and $\epsilon^{2}\vartheta^{-1}N(\vartheta R)$ are at least of order $\mathcal{O}(\epsilon^{2})$. Therefore, all terms on the RHS of \eqref{equation49} are at least of $\mathcal{O}(\epsilon^{2})$ except for the linear term $2\epsilon \vartheta^{-1}Q(\Psi_{c},\vartheta R)$. We also note that the term $\Lambda R$ can be explicitly computed and causes no growth in $R$.
Besides, note that whether $k$ is close to or far away from zero not will directly influences the size of the nonlinear term in \eqref{equation49} in Fourier space due to the derivative acting on it. Hence to separate the behavior in these two regions,  we define projection operators $P^{0}$ and $P^{1}$ by the Fourier multiplier
\begin{equation}
\begin{split}\label{equation58}
\widehat{P}^{0}(k)=\chi_{\mid k\mid\leq\delta}(k)\ \ \ \text{and} \ \ \ \widehat{P}^{1}(k)=\mathbf{1}-\widehat{P}^{0}(k),
\end{split}
\end{equation}
for a $\delta>0$ sufficiently small (the same constant $\delta$ in the definition of $\vartheta$), but independent of $0<\epsilon\ll1$. When necessary we will write $R=R^{0}+R^{1}$ with $R^{j}=P^{j}R$, for $j=0,1$.
Acting the projection operators $P^{0}$ and $P^{1}$ on the equation \eqref{equation49}, we will obtain the following evolutionary equations for $R^{0}$ and $R^{1}$,
\begin{equation}
\begin{split}\label{R0}
\partial_{t}R^{0}=\Lambda R^{0}+2\epsilon P^{0}\vartheta^{-1} Q(\Psi_{c},\vartheta R^{0})+2\epsilon P^{0}\vartheta^{-1} Q(\Psi_{c},\vartheta R^{1})+\mathcal{O}(\epsilon^{2}),
\end{split}
\end{equation}
and
\begin{equation}
\begin{split}\label{R1,1}
\partial_{t}R^{1}=\Lambda R^{1}+2\epsilon P^{1}\vartheta^{-1} Q(\Psi_{c},\vartheta_{0} R^{0})+2\epsilon P^{1}\vartheta^{-1} Q(\Psi_{c},\vartheta R^{1})+\mathcal{O}(\epsilon^{2}),
\end{split}
\end{equation}
where $\vartheta_{0}=\vartheta-\epsilon$. Noting that since $\widehat{\Psi}_{c}(k-m)=0$ unless $|(k-m)\pm k_{0}|<\delta$ and $\widehat{R}^{0}(m)=0$ for $|m|>\delta$, we see that $P^{0}\vartheta^{-1} Q(\Psi_{c},\vartheta R^{0})=0$. We want to estimate the error $(R^{0},R^{1})$ on a time scale $\mathcal{O}(1/\epsilon^{2})$, thus we need to eliminate the $\mathcal{O}(\epsilon)$ terms from \eqref{R0} and \eqref{R1,1} by using the normal form transformation
\begin{equation}
\begin{split}\label{B01}
\widetilde{R}_{j_{1}}^{0}:=R_{j_{1}}^{0}+\epsilon \sum_{j_{2}\in\{\pm1\}}B^{0,1}_{j_{1},j_{2}}(\Psi_{c},R_{j_{2}}^{1}),
\end{split}
\end{equation}
\begin{equation}
\begin{split}\label{B1a}
\widetilde{R}_{j_{1}}^{1}:=R_{j_{1}}^{1}+\epsilon \sum_{j_{2}\in\{\pm1\}}\big(B_{j_{1},j_{2}}^{1,0}(\Psi_{c},R_{j_{2}}^{0})+ B_{j_{1},j_{2}}^{1,1}(\Psi_{c},R_{j_{2}}^{1})\big),
\end{split}
\end{equation}
where
\begin{equation*}
\begin{split}
B^{0,1}_{j_{1},j_{2}}(\Psi_{c},R_{j_{2}}^{1})=\int_{\mathbb{R}}\widehat{b}^{0,1}_{j_{1},j_{2}}(k,k-m,m)\widehat{\phi}_{c}(k-m)\widehat{R}_{j_{2}}^{1}(m)dm.
\end{split}
\end{equation*}
Here $j_{1},j_{2}\in\{\pm1\}$ and we have used the equation \eqref{Phic}. Similarly for $B^{1,0}(\Psi_{c},R^{0})$ and $B^{1,1}(\Psi_{c},R^{1})$.
Inserting \eqref{B01} and \eqref{B1a} into \eqref{R0} and \eqref{R1,1} respectively, and then letting the $\mathcal{O}(\epsilon)$ terms equal to zero formally, we will obtain
\begin{equation}
\begin{split}\label{equation111}
&\partial_{t}\widetilde{R}^{0}=\Lambda\widetilde{R}^{0}+\epsilon^{2}f(\Psi,\widetilde{R})+\epsilon^{-\beta}Res_{U^{0}}(\epsilon\Psi),\\
&\partial_{t}\widetilde{R}^{1}=\Lambda\widetilde{R}^{1}+\epsilon^{2}g(\Psi,\widetilde{R})+\epsilon^{-\beta}Res_{U^{1}}(\epsilon\Psi),
\end{split}
\end{equation}
provided these normal form transformations are all invertible and well-defined.
Fortunately, by careful analysis we find that these normal form transformations are all invertible and well-defined in the Euler-Poisson system considered in the present paper. Summarizing, by introducing the cutoff function $\vartheta$ and in particular $\vartheta_{0}$, both the trivial resonance $k=0$ and the non-trivial resonances $k=\pm k_{0}$ do not cause problems for the definition of the normal form transformations.
Besides, the normal form transformation $B^{1,0}(\Psi_{c},\widetilde{R}^{0})$ satisfies
\begin{equation*}
\begin{split}
\|\epsilon B^{1,0}_{j_{1},j_{2}}(\Psi_{c},\widetilde{R}_{j_{2}}^{0})\|_{H^{s'}}\lesssim\epsilon\|\widetilde{R}^{0}\|_{H^{s}},
\end{split}
\end{equation*}
for any $s,s'\geq6$.

However, there still exist some difficulties in order to obtain uniform estimates for the remainder $\widetilde{R}$ in \eqref{equation111}. First, though the trivial resonance $k=0$ associated to $B^{0,1}(\Psi_{c},R^{1})$ does not cause difficulties, $B^{0,1}(\Psi_{c},R^{1})$ will lose one $\epsilon$ because $\vartheta^{-1}=\mathcal{O}(1/\epsilon)$ for $|k|<\delta$, i.e. for arbitrary $s,s'\geq6$, we have
\begin{equation*}
\begin{split}
\|\vartheta B^{0,1}_{j_{1},j_{2}}(\Psi_{c},\widetilde{R}_{j_{2}}^{1})\|_{H^{s'}}\lesssim\|\widetilde{R}^{1}\|_{H^{s}}.
\end{split}
\end{equation*}
This means that $\epsilon^{2}f(\Psi,\widetilde{R})$ is indeed of order $\mathcal{O}(\epsilon)$, although it looks like $\mathcal{O}(\epsilon^2)$. Thus we still need to eliminate such an $\mathcal{O}(\epsilon)$ term in the first equation of \eqref{equation111} by a second normal form transformation but applied only to $\widetilde{R}^0$ to obtain a new error function $\mathcal{R}^{0}$, which together with $\widetilde{R}^{1}$ satisfies
\begin{subequations}\label{R01}
\begin{numcases}{}
\partial_{t}\mathcal{R}^{0}=\Lambda\mathcal{R}^{0}+\epsilon^{2}\widetilde{f}(\Psi,\mathcal{R}^{0},\widetilde{R}^{1})+\epsilon^{-\beta}Res_{U^{0}}(\epsilon\Psi),\label{R01-1}\\
\partial_{t}\widetilde{R}^{1}=\Lambda\widetilde{R}^{1}+\epsilon^{2}\widetilde{g}(\Psi,\mathcal{R}^{0},\widetilde{R}^{1})+\epsilon^{-\beta}Res_{U^{1}}(\epsilon\Psi),
\label{R01-2}
\end{numcases}
\end{subequations}
where $\widetilde{f}(\Psi,\mathcal{R}^{0},\widetilde{R}^{1})$ is of order $\mathcal{O}(1)$ and does not lose derivative. Traditionally, one would then consider the energy estimates starting from \eqref{R01}, since all the terms on the right are of order $\mathcal{O}(\epsilon^2)$. However this will cause the following second problem.

Secondly, we can not use the variation of constants formula and Gronwall's inequality to bound $(\mathcal{R}^{0},\widetilde{R}^{1})$ for the equation \eqref{R01}, not because of missing powers of $\epsilon$ but due to regularity problems. Note that the quadratic term $\epsilon\vartheta^{-1}P^{1}Q(\Psi_{c},\vartheta R^{1})$ of \eqref{R1,1} is quasilinear and loses one derivative, i.e.,  $R^{1}\mapsto\epsilon\vartheta^{-1}P^{1}Q(\Psi_{c},\vartheta R^{1})$ maps $H^{m+1}(\mathbb{R},\mathbb{C})$ into $H^{m}(\mathbb{R},\mathbb{C})$ or $C^{m+1}(\mathbb{R},\mathbb{C})$ into $C^{m}(\mathbb{R},\mathbb{C})$. This causes the term $B^{1,1}(\Psi_{c},R^{1})$ to lose one derivative, which implies that the term $\widetilde{g}(\Psi,\mathcal{R}^{0},\widetilde{R}^{1})$ loses even two derivatives in total.

In order to overcome the above two difficulties at the same time, we still make the normal form transformation twice on the error function $R^{0}$ but not on $R^{1}$. By a similar procedure, we obtain the evolutionary equations for $(\mathcal{R}^{0}, R^{1})$,
\begin{subequations}\label{R01,10}
\begin{numcases}{}
\partial_{t}\mathcal{R}^{0}=\Lambda\mathcal{R}^{0}+\epsilon^{2}\overline{f}(\Psi,\mathcal{R}^{0},R^{1})
+\epsilon^{-\beta}Res_{U^{0}}(\epsilon\Psi),\label{R01,10-1}\\
\partial_{t}R^{1}=\Lambda R^{1}+\epsilon\overline{g}(\Psi,\mathcal{R}^{0},R^{1})+\epsilon^{-\beta}Res_{U^{1}}(\epsilon\Psi),\label{R01,10-2}
\label{R01-2}
\end{numcases}
\end{subequations}
where $\overline{f}$ and $\overline{g}$ do not lose $\epsilon$ but $\overline{g}$ may only lose one derivative. Then we use the transformed remainder $\mathcal{R}^{0}$ of \eqref{R01,10-1} for low frequency $|k|\leq\delta$, and the non-transformed remainder $R^{1}$ of \eqref{R01,10-2} for high frequency $|k|\geq\delta$ combined with $\epsilon B^{1,0}(\Psi,\mathcal{R}^{0})$ and $\epsilon B^{1,1}(\Psi,R^{1})$ to define the energy
\begin{equation}
\begin{split}\label{Es}
\mathcal{E}_{s}=&\sum_{\ell=0}^{s}\Big[\frac{1}{2}\big(\int_{\mathbb{R}}(\partial_{x}^{\ell}\mathcal{R}^{0})^{2}dx+
\int_{\mathbb{R}}(\partial_{x}^{\ell}R^{1})^{2}dx\big)\\
&+\epsilon\big(
\int_{\mathbb{R}}\partial_{x}^{\ell}R^{1}\partial_{x}^{\ell}B^{1,0}(\Psi_{c},\mathcal{R}^{0})dx
+\int_{\mathbb{R}}\partial_{x}^{\ell}R^{1}\partial_{x}^{\ell}B^{1,1}(\Psi_{c},R^{1})dx\big)\Big],
\end{split}
\end{equation}
where $s=s_{A}\geq6$. Although there is an $\mathcal O(\epsilon)$ term $\epsilon\bar g(\cdot,\cdot,\cdot)$ in \eqref{R01,10-2}, we can eliminate the $\mathcal O(\epsilon)$ terms and keep only the $\mathcal O(\epsilon^2)$ {terms in the evolutionary equation of $\mathcal E_s$} by the carefully constructed energy functional $\mathcal E_s$ in \eqref{Es}. {The strategy of definition of the energy using the normal form transformation was already used in previous papers \cite{D1,D12,H,H16,IT,SW,SW02}.} But here we would like to remark that there are some basic differences in this paper.


Let us explain why we use $(\mathcal R^0,R^1)$ to construct the energy functional $\mathcal E_s$. On one hand, for the high frequency component $R^1$, if we start form \eqref{R01}, although both the evolution equation of $\mathcal{E}_{s}$ and $\|(\mathcal{R}^{0},\widetilde{R}^{1})\|_{H^{s}}^{2}$ in \eqref{R01} share the property that their right-hand side terms are all of order $\mathcal{O}(\epsilon^{2})$, $\tilde g(\cdot,\cdot,\cdot)$ loses two derivatives in total, which will leads to the difficulty for closing energy estimate of $\widetilde{R}^1$. On the other hand, the method is not unique to deal with the low frequency component $R^0$. In this paper, we do twice normal-form translations for $R^0$ to obtain the equation for $\mathcal R^0$. However, since the normal-form leading from $R^0$ to $\mathcal R^0$ only involves bounded frequencies, then we can also put $R^0$ with the related normal-form transforms in the energy argument $\mathcal E_s$ as done for $R^1$.

We can show $\|B^{1,0}(\Psi_{c},\mathcal{R}^{0})\|_{H^{s}}\lesssim\|\mathcal{R}^{0}\|_{H^{s}}$ in Lemma \ref{L10} as well as the equivalence between $\|(\mathcal{R}^{0},R^{1})\|_{H^{s}}^{2}$ and $\|(R^{0},R^{1})\|_{H^{s}}^{2}$ by the form of the normal form transformations for $|k|\leq\delta$. See \eqref{equation92} for details.  Besides, $B^{1,1}(\Psi_{c},R^{1})$ can be split into a term of the form $diag\big\{h_{1}(\Psi),h_{2}(\Psi)\big\}\partial_{x}R^{1}$ and terms that do not lose regularity in Lemma \ref{L8}, which is very important to obtain closed energy estimates. Then by integration by parts and using the inequality $\|B^{1,0}(\Psi_{c},\mathcal{R}^{0})\|_{H^{s}}\lesssim\|\mathcal{R}^{0}\|_{H^{s}}$, we can obtain the equivalence between ${\mathcal{E}_{s}}$ and $\|(\mathcal{R}^{0},R^{1})\|_{H^{s}}^{2}$ for sufficiently small $\epsilon$, which finally yields the equivalence between the energy $\mathcal E_s$ and the original remainder $\|(R^{0},R^{1})\|_{H^{s}}^{2}$ in \eqref{before}.

Finally, the structure of $\Lambda$ and the properties of $\omega$ allow us to construct a modified energy
\begin{equation}
\begin{split}\label{Ess}
\widetilde{\mathcal{E}}_{s}=\mathcal{E}_{s}+\epsilon^{2}h,
\end{split}
\end{equation}
where $h=\mathcal{O}\big(\|(\mathcal{R}^{0},{R}^{1})\|_{H^{s}}^{2}\big)$ as long as $\|(\mathcal{R}^{0},{R}^{1})\|_{H^{s}}$ is $\mathcal{O}(1)$. Consequently, we obtain
\begin{equation*}
\begin{split}
\partial_{t}\widetilde{\mathcal{E}}_{s}\leq C\epsilon^{2}(\widetilde{\mathcal{E}}_{s}+1),
\end{split}
\end{equation*}
as long as $\|(\mathcal{R}^{0},{R}^{1})\|_{H^{s}}$ is $\mathcal{O}(1)$ such that Gronwall's inequality yields the $\mathcal{O}(1)$ boundedness of $\widetilde{\mathcal{E}}_{s}$ and hence of $R$ for all $t\in[0,T_{0}/\epsilon^{2}]$.

\section{\textbf{The normal form transformation}}
As mentioned in Section 3, we need to eliminate the $\mathcal{O}(\epsilon)$ terms in the error equation for $R^{0}$ (i.e. low frequency $|k|<\delta$) by normal form transformations. We define a weight function $\vartheta$ to reflect the fact that the nonlinearity vanishes at $k=0$,
\begin{equation}
\begin{split}\label{va}
\widehat{\vartheta}(k)=\Big\{\begin{matrix} 1\ \ \ \ \ \ \ \ \ \ \ \ \ \ \ \ \ \ \ \ \ \ \ \ \ \ \ \ \text{for} \ \ |k|>\delta, \\ \epsilon+(1-\epsilon)| k|/\delta \ \ \ \ \ \ \ \ \ \ \text{for} \ \ |k|\leq\delta,\end{matrix}
\end{split}
\end{equation}
for some $\delta>0$ sufficiently small, but independent of $0<\epsilon\ll1$,
write the solution $U$ of \eqref{equation9} as a sum of the approximation and the error, i.e.,
\begin{equation}
\begin{split}\label{equat48}
U=\epsilon\Psi+\epsilon^{5/2}\vartheta R,
\end{split}
\end{equation}
where to avoid writing the convolution $\vartheta\ast R$, $\vartheta R$ is defined by $\widehat{\vartheta R}=\widehat{\vartheta} \widehat{R}$ in a slight abuse of notation. Note that $\widehat{\vartheta}(k) \widehat{R}(k)$ is small at the wave numbers close to zero, since the nonlinearity vanishes at $k=0$.

Then we define two projection operators $P^{0}$ and $P^{1}$ by the Fourier multiplier
\begin{equation}
\begin{split}\label{equation58}
\widehat{P}^{0}(k)=\chi_{\mid k\mid\leq\delta}(k)\ \ \ and \ \ \ \widehat{P}^{1}(k)=\mathbf{1}-\widehat{P}^{0}(k),
\end{split}
\end{equation}
for $\delta>0$ sufficiently small (the same constant $\delta$ in the definition of $\vartheta$), but independent of $0<\epsilon\ll1$. When necessary we will write $R=R^{0}+R^{1}$ with $R^{j}=P^{j}R$, for $j=0,1$. In the following, the superscripts $0,1$ always denote the spectrum projections, and should not confused with the subscripts that denote the component of $R$.

Recall the form of $\Psi$ in the equation \eqref{equation122}. For simplicity, let
\begin{equation*}
\begin{split}
\phi_{c}:=\psi_{1}+\psi_{-1}, \ \phi_{p_{1}}:=\psi_{p1}+\psi_{p-1}, \ \phi_{p_{2}}:=\psi_{p1}-\psi_{p-1}.
\end{split}
\end{equation*}
Then we have $\text{supp}\widehat{\phi}_{c}=\big\{k\mid|k\pm k_{0}|<\delta\big\}$ and $\text{supp}\widehat{\phi}_{p_{1,2}}=\big\{k\mid|k\pm jk_{0}|<\delta,j=0,\pm2,\pm3,\pm4\big\}$. Besides,
noting that since $\widehat{\phi}_{c}(k-m)=0$ unless $|(k-m)\pm k_{0}|<\delta$ and since $\widehat{R}_{j}^{0}(m)=0$ for $|m|>\delta$, we see that $P^{0}(\phi_{c}R_{j}^{0})=0$ for $j=\pm1$.

Inserting \eqref{equat48} into \eqref{equation9} and by the projection operators $P^{0}$ and $P^{1}$, we obtain evolutionary equations for $R^{0}$ and $R^{1}$ in Fourier transformation with $j_{1}\in\{\pm1\}$:
\begin{align}\label{equ1}
\partial_{t}\widehat{R}_{j_{1}}^{0}=&ij_{1}\omega(k)\widehat{R}_{j_{1}}^{0}
+\epsilon\frac{\widehat{P}^{0}(k)ik}{2\widehat{\vartheta}(k)}\widehat{\phi}_{c}\ast \big(\widehat{q}\widehat{\vartheta}(\widehat{R}_{1}^{1}-\widehat{R}_{-1}^{1})\big)
+\epsilon\frac{\widehat{P}^{0}(k)ik}{2\widehat{\vartheta}(k)}(\widehat{q}\widehat{\phi}_{c})\ast \big(\widehat{\vartheta}(\widehat{R}_{1}^{1}+\widehat{R}_{-1}^{1})\big)\nonumber\\
&+j_{1}\epsilon\frac{\widehat{P}^{0}(k)ik}{2\widehat{\vartheta}(k)\widehat{q}(k)}(\widehat{q}\widehat{\phi}_{c})\ast \big(\widehat{q}\widehat{\vartheta}(\widehat{R}_{1}^{1}-\widehat{R}_{-1}^{1})\big)
-j_{1}\epsilon\frac{\widehat{P}^{0}(k)ik}{2\widehat{\vartheta}(k)\widehat{q}(k)}(\widehat{\phi}_{c})
\ast\big(\widehat{\vartheta}(\widehat{R}_{1}^{1}+\widehat{R}_{-1}^{1})\big)\nonumber\\
&-j_{1}\epsilon\frac{\widehat{P}^{0}(k)ik}{2\widehat{\vartheta}(k)\langle k\rangle^{2}\widehat{q}(k)}\big(\langle\widehat{\partial_{x}}\rangle^{-2}\widehat{\phi}_{c}\big)\ast \big(\langle\widehat{\partial_{x}}\rangle^{-2}\widehat{\vartheta}(\widehat{R}_{1}^{1}+\widehat{R}_{-1}^{1})\big)\nonumber\\
&+\epsilon^{2}\frac{\widehat{P}^{0}(k)ik}{2\widehat{\vartheta}(k)}\widehat{\phi}_{p_{1}}\ast\big(\widehat{q}\widehat{\vartheta}
(\widehat{R}_{1}^{0}-\widehat{R}_{-1}^{0}+\widehat{R}_{1}^{1}-\widehat{R}_{-1}^{1})\big)\nonumber\\
&+j_{1}\epsilon^{2}\frac{\widehat{P}^{0}(k)ik}{2\widehat{\vartheta}(k)\widehat{q}(k)}\widehat{q}\widehat{\phi}_{p_{2}}\ast\big(\widehat{q}\widehat{\vartheta}
(\widehat{R}_{1}^{0}-\widehat{R}_{-1}^{0}+\widehat{R}_{1}^{1}-\widehat{R}_{-1}^{1})\big)\nonumber\\
&+\epsilon^{2}\frac{\widehat{P}^{0}(k)ik}{2\widehat{\vartheta}(k)}\widehat{q}\widehat{\phi}_{p_{2}}\ast\big(\widehat{\vartheta}
(\widehat{R}_{1}^{0}+\widehat{R}_{-1}^{0}+\widehat{R}_{1}^{1}+\widehat{R}_{-1}^{1})\big)\nonumber\\
&-j_{1}\epsilon^{2}\frac{\widehat{P}^{0}(k)ik}{2\widehat{\vartheta}(k)\widehat{q}(k)}(\widehat{\phi}_{p_{1}})\ast\big(\widehat{\vartheta}
(\widehat{R}_{1}^{0}+\widehat{R}_{-1}^{0}+\widehat{R}_{1}^{1}+\widehat{R}_{-1}^{1})\big)\nonumber\\
&-j_{1}\epsilon^{2}\frac{\widehat{P}^{0}(k)ik}{2\widehat{\vartheta}(k)\langle k\rangle^{2}\widehat{q}(k)}\big(\langle\widehat{\partial_{x}}\rangle^{-2}\widehat{\phi}_{p_{1}}\big)\ast \big(\langle\widehat{\partial_{x}}\rangle^{-2}\widehat{\vartheta}(\widehat{R}_{1}^{0}+\widehat{R}_{-1}^{0}+\widehat{R}_{1}^{1}+\widehat{R}_{-1}^{1})\big)\nonumber\\
&+\epsilon^{5/2}\frac{\widehat{P}^{0}(k)ik}{2\widehat{\vartheta}(k)}
\big(\widehat{\vartheta}(\widehat{R}_{1}^{0}+\widehat{R}_{-1}^{0}+\widehat{R}_{1}^{1}+\widehat{R}_{-1}^{1})\big)\ast
\widehat{q}\widehat{\vartheta}\big(\widehat{R}_{1}^{0}-\widehat{R}_{-1}^{0}+\widehat{R}_{1}^{1}-\widehat{R}_{-1}^{1})\big)\nonumber\\
&-j_{1}\epsilon^{5/2}\frac{\widehat{P}^{0}(k)ik}{4\widehat{\vartheta}(k)\widehat{q}(k)}\big(\widehat{\vartheta}
(\widehat{R}_{1}^{0}+\widehat{R}_{-1}^{0}+\widehat{R}_{1}^{1}+\widehat{R}_{-1}^{1})\big)^{\ast^{2}}\nonumber\\
&-j_{1}\epsilon^{5/2}\frac{\widehat{P}^{0}(k)ik}{4\widehat{\vartheta}(k)\langle k\rangle^{2}\widehat{q}(k)}\big(\langle\widehat{\partial_{x}}\rangle^{-2}\widehat{\vartheta}
(\widehat{R}_{1}^{0}+\widehat{R}_{-1}^{0}+\widehat{R}_{1}^{1}+\widehat{R}_{-1}^{1})\big)^{\ast^{2}}\nonumber\\
&+j_{1}\epsilon^{5/2}\frac{\widehat{P}^{0}(k)ik}{4\widehat{\vartheta}(k)\widehat{q}(k)}\big(\widehat{q}\widehat{\vartheta}
(\widehat{R}_{1}^{0}+\widehat{R}_{-1}^{0}+\widehat{R}_{1}^{1}+\widehat{R}_{-1}^{1})\big)^{\ast^{2}}\nonumber\\
&+j_{1}\epsilon^{2}\frac{\widehat{P}^{0}(k)ik}{2\widehat{\vartheta}(k)\widehat{q}(k)}\widehat{\mathcal{G}}
\widehat{\vartheta}\big(\widehat{R}_{1}^{0}+\widehat{R}_{-1}^{0}+\widehat{R}_{1}^{1}+\widehat{R}_{-1}^{1}\big)
+j_{1}\epsilon^{2}\frac{\widehat{P}^{0}(k)ik}{2\widehat{\vartheta}(k)\widehat{q}(k)}\widehat{\mathcal{M}}
\widehat{\vartheta}\big(\widehat{R}_{1}^{0}+\widehat{R}_{-1}^{0}+\widehat{R}_{1}^{1}+\widehat{R}_{-1}^{1}\big)\nonumber\\
&+\epsilon^{-5/2}\widehat{Res}_{U_{j_{1}}^{0}(\epsilon\Psi)}\nonumber\\
=&:ij_{1}\omega(k)\widehat{R}_{j_{1}}^{0}+\epsilon\frac{\widehat{P}^{0}(k)ik}{2\widehat{\vartheta}(k)}
\sum_{j_{2}\in\{\pm1\}}\sum_{n=1}^{5}\widehat{\alpha}_{j_{1},j_{2}}^{n}(k,k-m,m)\widehat{\phi}_{c}\ast \widehat{\vartheta}\widehat{R}_{j_{2}}^{1}
+\epsilon^{2}\mathcal{F}^{1},
\end{align}
and
\begin{align}\label{equ1,-1}
\partial_{t}\widehat{R}_{j_{1}}^{1}=&ij_{1}\omega(k)\widehat{R}_{j_{1}}^{1}
+\epsilon\frac{\widehat{P}^{1}(k)ik}{2\widehat{\vartheta}(k)}\widehat{\phi}_{c}\ast \big(\widehat{q}\widehat{\vartheta}(\widehat{R}_{1}^{0}-\widehat{R}_{-1}^{0})\big)
+\epsilon\frac{\widehat{P}^{1}(k)ik}{2\widehat{\vartheta}(k)}(\widehat{q}\widehat{\phi}_{c})\ast \big(\widehat{\vartheta}(\widehat{R}_{1}^{0}+\widehat{R}_{-1}^{0})\big)\nonumber\\
&+j_{1}\epsilon\frac{\widehat{P}^{1}(k)ik}{2\widehat{\vartheta}(k)\widehat{q}(k)}(\widehat{q}\widehat{\phi}_{c})\ast \big(\widehat{q}\widehat{\vartheta}(\widehat{R}_{1}^{0}-\widehat{R}_{-1}^{0})\big)
-j_{1}\epsilon\frac{\widehat{P}^{1}(k)ik}{2\widehat{\vartheta}(k)\widehat{q}(k)}(\widehat{\phi}_{c})\ast
\big(\widehat{\vartheta}(\widehat{R}_{1}^{0}+\widehat{R}_{-1}^{0})\big)\nonumber\\
&-j_{1}\epsilon\frac{\widehat{P}^{1}(k)ik}{2\widehat{\vartheta}(k)\langle k\rangle^{2}\widehat{q}(k)}\big(\langle\widehat{\partial_{x}}\rangle^{-2}\widehat{\phi}_{c}\big)\ast \big(\langle\widehat{\partial_{x}}\rangle^{-2}\widehat{\vartheta}(\widehat{R}_{1}^{0}+\widehat{R}_{-1}^{0})\big)\nonumber\\
&+\epsilon\frac{\widehat{P}^{1}(k)ik}{2\widehat{\vartheta}(k)}\widehat{\phi}_{c}\ast \big(\widehat{q}\widehat{\vartheta}(\widehat{R}_{1}^{1}-\widehat{R}_{-1}^{1})\big)
+\epsilon\frac{\widehat{P}^{1}(k)ik}{2\widehat{\vartheta}(k)}(\widehat{q}\widehat{\phi}_{c})\ast \big(\widehat{\vartheta}(\widehat{R}_{1}^{1}+\widehat{R}_{-1}^{1})\big)\nonumber\\
&+j_{1}\epsilon\frac{\widehat{P}^{1}(k)ik}{2\widehat{\vartheta}(k)\widehat{q}(k)}(\widehat{q}\widehat{\phi}_{c})\ast \big(\widehat{q}\widehat{\vartheta}(\widehat{R}_{1}^{1}-\widehat{R}_{-1}^{1})\big)
-j_{1}\epsilon\frac{\widehat{P}^{1}(k)ik}{2\widehat{\vartheta}(k)\widehat{q}(k)}
(\widehat{\phi}_{c})\ast\big(\widehat{\vartheta}(\widehat{R}_{1}^{1}+\widehat{R}_{-1}^{1})\big)\nonumber\\
&-j_{1}\epsilon\frac{\widehat{P}^{1}(k)ik}{2\widehat{\vartheta}(k)\langle k\rangle^{2}\widehat{q}(k)}\big(\langle\widehat{\partial_{x}}\rangle^{-2}\widehat{\phi}_{c}\big)\ast \big(\langle\widehat{\partial_{x}}\rangle^{-2}\widehat{\vartheta}(\widehat{R}_{1}^{1}+\widehat{R}_{-1}^{1})\big)\nonumber\\
&+\epsilon^{2}\frac{\widehat{P}^{1}(k)ik}{2\widehat{\vartheta}(k)}\widehat{\phi}_{p_{1}}\ast\big(\widehat{q}\widehat{\vartheta}
(\widehat{R}_{1}^{0}-\widehat{R}_{-1}^{0}+\widehat{R}_{1}^{1}-\widehat{R}_{-1}^{1})\big)\nonumber\\
&+\epsilon^{2}\frac{\widehat{P}^{1}(k)ik}{2\widehat{\vartheta}(k)}\widehat{q}\widehat{\phi}_{p_{2}}\ast\big(\widehat{\vartheta}
(\widehat{R}_{1}^{0}+\widehat{R}_{-1}^{0}+\widehat{R}_{1}^{1}+\widehat{R}_{-1}^{1})\big)\nonumber\\
&+j_{1}\epsilon^{2}\frac{\widehat{P}^{1}(k)ik}{2\widehat{\vartheta}(k)\widehat{q}(k)}\widehat{q}\widehat{\phi}_{p_{2}}\ast\big(\widehat{q}\widehat{\vartheta}
(\widehat{R}_{1}^{0}-\widehat{R}_{-1}^{0}+\widehat{R}_{1}^{1}-\widehat{R}_{-1}^{1})\big)\nonumber\\
&-j_{1}\epsilon^{2}\frac{\widehat{P}^{1}(k)ik}{2\widehat{\vartheta}(k)\widehat{q}(k)}(\widehat{\phi}_{p_{1}})\ast\big(\widehat{\vartheta}
(\widehat{R}_{1}^{0}+\widehat{R}_{-1}^{0}+\widehat{R}_{1}^{1}+\widehat{R}_{-1}^{1})\big)\nonumber\\
&-j_{1}\epsilon^{2}\frac{\widehat{P}^{1}(k)ik}{2\widehat{\vartheta}(k)\langle k\rangle^{2}\widehat{q}(k)}\big(\langle\widehat{\partial_{x}}\rangle^{-2}\widehat{\phi}_{p_{1}}\big)\ast \big(\langle\widehat{\partial_{x}}\rangle^{-2}\widehat{\vartheta}
(\widehat{R}_{1}^{0}+\widehat{R}_{-1}^{0}+\widehat{R}_{1}^{1}+\widehat{R}_{-1}^{1})\big)\nonumber\\
&+\epsilon^{5/2}\frac{\widehat{P}^{1}(k)ik}{2\widehat{\vartheta}(k)}
\big(\widehat{\vartheta}(\widehat{R}_{1}^{0}+\widehat{R}_{-1}^{0}+\widehat{R}_{1}^{1}+\widehat{R}_{-1}^{1})\big)\ast
\widehat{q}\widehat{\vartheta}\big(\widehat{R}_{1}^{0}-\widehat{R}_{-1}^{0}+\widehat{R}_{1}^{1}-\widehat{R}_{-1}^{1})\big)\nonumber\\
&+j_{1}\epsilon^{5/2}\frac{\widehat{P}^{1}(k)ik}{4\widehat{\vartheta}(k)\widehat{q}(k)}\big(\widehat{q}\widehat{\vartheta}
(\widehat{R}_{1}^{0}-\widehat{R}_{-1}^{0}+\widehat{R}_{1}^{1}-\widehat{R}_{-1}^{1})\big)^{\ast^{2}}\nonumber\\
&-j_{1}\epsilon^{5/2}\frac{\widehat{P}^{1}(k)ik}{4\widehat{\vartheta}(k)\widehat{q}(k)}
\big(\widehat{\vartheta}(\widehat{R}_{1}^{0}+\widehat{R}_{-1}^{0}+\widehat{R}_{1}^{1}+\widehat{R}_{-1}^{1})\big)^{\ast^{2}}\nonumber\\
&-j_{1}\epsilon^{5/2}\frac{\widehat{P}^{1}(k)ik}{4\widehat{\vartheta}(k)\langle k\rangle^{2}\widehat{q}(k)}\big(\langle\widehat{\partial_{x}}\rangle^{-2}\widehat{\vartheta}
(\widehat{R}_{1}^{0}+\widehat{R}_{-1}^{0}+\widehat{R}_{1}^{1}+\widehat{R}_{-1}^{1})\big)^{\ast^{2}}\nonumber\\
&+j_{1}\epsilon^{2}\frac{\widehat{P}^{1}(k)ik}{2\widehat{\vartheta}(k)\widehat{q}(k)}\widehat{\mathcal{G}}
\widehat{\vartheta}\big(\widehat{R}_{1}^{0}+\widehat{R}_{-1}^{0}+\widehat{R}_{1}^{1}+\widehat{R}_{-1}^{1}\big)
+j_{1}\epsilon^{2}\frac{\widehat{P}^{1}(k)ik}{2\widehat{\vartheta}(k)\widehat{q}(k)}\widehat{\mathcal{M}}
\widehat{\vartheta}\big(\widehat{R}_{1}^{0}+\widehat{R}_{-1}^{0}+\widehat{R}_{1}^{1}+\widehat{R}_{-1}^{1}\big)\nonumber\\
&+\epsilon^{-5/2}\widehat{Res}_{U_{j_{1}}^{1}}(\epsilon\Psi)\nonumber\\
=:&ij_{1}\omega(k)\widehat{R}_{j_{1}}^{1}+\epsilon\frac{\widehat{P}^{1}(k)ik}{2\widehat{\vartheta}(k)}
\sum_{j_{2}\in\{\pm1\}}\sum_{n=1}^{5}\widehat{\alpha}_{j_{1},j_{2}}^{n}(k,k-m,m)(\widehat{\phi}_{c}\ast \widehat{\vartheta}\widehat{R}_{j_{2}}^{0})\nonumber\\
& \ \ +\epsilon\frac{\widehat{P}^{1}(k)ik}{2\widehat{\vartheta}(k)}\sum_{j_{2}\in\{\pm1\}}
\sum_{n=1}^{5}\widehat{\alpha}_{j_{1},j_{2}}^{n}(k,k-m,m)(\widehat{\phi}_{c}\ast \widehat{\vartheta}\widehat{R}_{j_{2}}^{1})
+\epsilon^{2}\mathcal{F}^{2},
\end{align}
where the notation $\mathcal{G}$ may depend on the error $(R^{0},R^{1})$, but do not lose regularity. The notation $\epsilon^{2}\mathcal{F}^{1}$ means the $H^{s'}$ norm of this term can be bounded by $C\epsilon^{2}$ if $R$ is in some bounded neighborhood of the origin in $H^{s}$ for any $s,s'>0$, since the coefficients of $\epsilon^{2}\mathcal{F}^{1}$ are equal to $C\big|\epsilon^{2}\frac{k}{\vartheta}\big|\leq C\epsilon^{2}$ for $|k|<\delta$. Similarly, $\epsilon^{2}\mathcal{F}^{2}$ means the $H^{s-1}$ norm of this term can be bounded by $C\epsilon^{2}$ if $R$ is in some bounded neighborhood of the origin in $H^{s}$.

Note that $|\widehat{\alpha}_{j_{1},j_{2}}^{n}(k,k-m,m)|\leq C$ for all $k,m\in\mathbb{R}$.
In the following we will attempt to construct normal form transformations to eliminate the $\mathcal{O}(\epsilon)$ terms from \eqref{equ1} and examine their effect on the full equation \eqref{equ1}. For this purpose, we consider the first normal form transformation of the form
\begin{equation}
\begin{split}\label{equ3}
\widetilde{R}_{j_{1}}^{0}=R_{j_{1}}^{0}+\epsilon \sum_{j_{2}\in\{\pm1\}}\sum_{n=1}^{5}B_{j_{1},j_{2}}^{0,1,n}(\phi_{c},R_{j_{2}}^{1}),
\end{split}
\end{equation}
where
\begin{equation}
\begin{split}\label{equ4}
\widehat{B}_{j_{1},j_{2}}^{0,1,n}(\phi_{c},R_{j_{2}}^{1})=\int \widehat{b}_{j_{1},j_{2}}^{0,1,n}(k,k-m,m)\widehat{\phi}_{c}(k-m)\widehat{R}_{j_{2}}(m)dm.
\end{split}
\end{equation}

\emph{Construction of $B_{j_{1},j_{2}}^{0,1,n}$}. Differentiating the $\widetilde{R}_{j_{1}}^{0}$ in \eqref{equ3} w.r.t. $t$, we obtain
\begin{equation}
\begin{split}\label{equ5}
\partial_{t}\widetilde{R}_{j_{1}}^{0}=\partial_{t}R_{j_{1}}^{0}+\epsilon\sum_{j_{2}\in\{\pm1\}}\sum_{n=1}^{5}B_{j_{1},j_{2}}^{0,1,n}(\partial_{t}\phi_{c},R_{j_{2}}^{1})
+\epsilon\sum_{j_{2}\in\{\pm1\}}\sum_{n=1}^{5} B_{j_{1},j_{2}}^{0,1,n}(\phi_{c},\partial_{t}R_{j_{2}}^{1}).
\end{split}
\end{equation}
Recall that $\widehat{\Omega u}(k)=i\omega(k)\widehat{u}(k)$ and $\|\partial_{t}\widehat{\phi}_{c}+i\omega\widehat{\phi}_{c}\|_{L^{1}(s)}\leq C\epsilon^{2}$ in Lemma \ref{L5}. Then provided the transformation $B_{j_{1},j_{2}}^{0,1,n}$ is well-defined and bounded, we have
\begin{equation}
\begin{split}\label{equ6}
\partial_{t}\widetilde{R}_{j_{1}}^{0}=&j_{1}\Omega\widetilde{R}_{j_{1}}^{0}
+\epsilon \sum_{j_{2}\in\{\pm1\}}\sum_{n=1}^{5}\Big[-j_{1}\Omega B_{j_{1},j_{2}}^{0,1,n}(\phi_{c},R_{j_{2}}^{1})
+\frac{P^{0}\partial_{x}}{2\vartheta}\alpha_{j_{1},j_{2}}^{n}(\phi_{c} \vartheta R_{j_{2}}^{1})\\
&\ \ \ \ \ \ \ \ \ \ \ \ \ \ \ \ \ \ \ \ \ \ \ \ \ \ \ \ \ \ \ \ \ \ \ \ \ -B_{j_{1},j_{2}}^{0,1,n}(\Omega\phi_{c},R_{j_{2}}^{1})
+ j_{2}B_{j_{1},j_{2}}^{0,1,n}(\phi_{c},\Omega R_{j_{2}}^{1})\Big]\\
&+\epsilon^{2}\sum_{j_{2},j_{3}\in\{\pm1\}}\sum_{n,\tilde{n}=1}^{5}B_{j_{1},j_{2}}^{0,1,n}
\big(\phi_{c},\frac{P^{1}\partial_{x}}{2\vartheta}\alpha_{j_{2},j_{3}}^{\tilde{n}}(\phi_{c}\vartheta R_{j_{3}}^{1})\big)
+\epsilon^{2}\mathcal{F}^{3},
\end{split}
\end{equation}
where $\epsilon^{2}\mathcal{F}^{3}$ have the same property with $\epsilon^{2}\mathcal{F}^{1}$.
To eliminate all terms which are formally $\mathcal{O}(\epsilon)$ of \eqref{equ6}, we choose $B^{0,1,n}$ so that
\begin{equation*}
\begin{split}
-j_{1}\Omega B_{j_{1},j_{2}}^{0,1,n}(\phi_{c},R_{j_{2}}^{1})
+\frac{P^{0}\partial_{x}}{2\vartheta}\alpha_{j_{1},j_{2}}^{n}(\phi_{c} \vartheta R_{j_{2}}^{1})
-B_{j_{1},j_{2}}^{0,1,n}(\Omega\phi_{c},R_{j_{2}}^{1})
+ j_{2}B_{j_{1},j_{2}}^{0,1,n}(\phi_{c},\Omega R_{j_{2}}^{1})=0.
\end{split}
\end{equation*}
It is equivalent to require that the kernel of $B^{0,1,n}$ be of the form
\begin{equation}
\begin{split}\label{equ7}
\widehat{b}_{j_{1},j_{2}}^{0,1,n}(k,k-m,m)=\frac{ -k\widehat{P}^{0}(k)\widehat{\alpha}_{j_{1},j_{2}}^{n}(k,k-m,m)}
{-j_{1}\omega(k)-\omega(k-m)+j_{2}\omega(m)}\frac{\widehat{\vartheta}(m)}{2\widehat{\vartheta}(k)}.
\end{split}
\end{equation}
Note that the kernel $\widehat{b}_{j_{1},j_{2}}^{0,1,n}$ has to be analyzed  only for $|(k-m)\pm k_{0}|<\delta$, \ $|k|\leq\delta$ and $|m|>\delta$,  since $\widehat{P}^{0}$ and $\widehat{\phi}_{c}$ are localized near $k=0$ and $k-m=\pm k_{0}$ respectively. Thus the resonance at $k=0$ will play a role for $B_{j_{1},j_{2}}^{0,1,n}$. Note also that if we consider the denominator of this kernel $\widehat{b}_{j_{1},j_{2}}^{0,1,n}$ near $k=0$, then
\begin{equation*}
\begin{split}
-j_{1}\omega(k)-\omega(k-m)+j_{2}\omega(m)
=-j_{1}\omega'(0)k-\big(\omega(-m)+\omega'(-m)k\big)+j_{2}\omega(m)+\mathcal{O}(k^{2}).
\end{split}
\end{equation*}
If $j_{2}=1$, this equality is bounded by some $\mathcal{O}(1)$ constant for all $|k|<\delta$. On the other hand, if $j_{2}=-1$, there exists a positive constant $C$ such that
\begin{equation}
\begin{split}\label{equ8}
\big|-j_{1}\omega(k)-\omega(k-m)+j_{2}\omega(m)\big|\geq C|k|.
\end{split}
\end{equation}
Thus there exists some $C\geq 0$ such that
\begin{equation}
\begin{split}\label{equ9}
\big|\widehat{\vartheta}(k)\widehat{b}_{j_{1},j_{2}}^{0,1,n}(k,k-m,m)\big|\leq C,
\end{split}
\end{equation}
for all $|k|\leq \delta$ and $n=1,2,3,4,5$. Since the factor of $\widehat{P}^{0}(k)$ makes $\widehat{b}_{j_{1},j_{2}}^{0,1,n}(k,k-m,m)=0$ if $|k|>\delta$, for any $s'>0$, there exists $C_{s'}$ such that
\begin{equation}
\begin{split}\label{equ1,1}
\big\|\varepsilon B_{j_{1},j_{2}}^{0,1,n}(\phi_{c},R_{j_{2}}^{1})\big\|_{H^{s'}}\leq C_{s'}\|R_{j_{2}}^{1}\|_{H^{s}},
\end{split}
\end{equation}
given $R_{j_{2}}^{1}\in H^{s}$ for some $s\geq6$. In particular, it holds when $s'=s$. However, we cannot assume that $C_{s'}\sim\mathcal{O}(\varepsilon)$ since $\widehat{\vartheta}^{-1}(k)\sim\varepsilon^{-1}$ for $k\approx0$, in spite of the factor of $\varepsilon$ in front of $B_{j_{1},j_{2}}^{0,1,n}$.

It is worth noting that although the terms $\epsilon^{2}B_{j_{1},j_{2}}^{0,1,n}\big(\phi_{c},\frac{P^{1}\partial_{x}}{2\vartheta}\alpha_{j_{2},j_{3}}^{\tilde{n}}(\phi_{c}\vartheta R_{j_{3}}^{1})\big)$ appearing in \eqref{equ6} are formally to be $\mathcal{O}(\epsilon^{2})$, the kernel of the transformation $B^{0,1,n}_{j_{1},j_{2}}$ is indeed $\mathcal{O}(\epsilon^{-1})$ for certain wave numbers so that this term is in fact only $\mathcal{O}(\epsilon)$ for those waves numbers and must therefore be retained. Therefore a second normal form transformation is needed.
Before constructing the second normal form transformation, we prove the following Lemma which will simplify the discussion in the sequel and will allow us to extract the real dangerous terms from $\epsilon^{2}B_{j_{1},j_{2}}^{0,1,n}
(\phi_{c},\frac{P^{1}\partial_{x}}{2\vartheta}\alpha_{j_{2},j_{3}}^{\tilde{n}}(\phi_{c}\vartheta R_{j_{3}}^{1}))$. This lemma take advantage of the strong localization of $\phi_{c}$ near the wave numbers $\pm k_{0}$ in Fourier space \cite{S}.
\begin{lemma}\label{L6}
Fix $p\in\mathbb{R}$ and assume that $\kappa=\kappa(k,k-m,m)\in C(\mathbb{R},\mathbb{C})$. Assume further that $\psi$ has a finitely supported Fourier transform and that $R\in H^{s}$. Then, \\
(i)\  if $\kappa$ is Lipschitz with respect to its second argument in some neighborhood of $p\in\mathbb{R}$,  there exists $C_{\psi,\kappa,p}>0$ such that
\begin{equation}
\begin{split}\label{equation56}
\Big\|\int&\kappa(\cdot,\cdot-m,m)\widehat{\psi}\big(\frac{\cdot-m-p}{\epsilon}\big)\widehat{R}(m)dm
-\int\kappa(\cdot,p,m)\widehat{\psi}\big(\frac{\cdot-m-p}{\epsilon}\big)\widehat{R}(m)dm\Big\|_{H^{s}}\\
&\leq C_{\psi,\kappa,p}\epsilon\|R\|_{H^{s}},
\end{split}
\end{equation}
(ii) if $\kappa$ is globally Lipschitz with respect to its third argument, there exists $D_{\psi,\kappa}>0$ such that
\begin{equation}
\begin{split}\label{equation57}
\Big\|\int&\kappa(\cdot,\cdot-m,m)\widehat{\psi}\big(\frac{\cdot-m-p}{\epsilon}\big)\widehat{R}(m)dm
-\int\kappa(\cdot,\cdot-m,\cdot-p)\widehat{\psi}\big(\frac{\cdot-m-p}{\epsilon}\big)\widehat{R}(m)dm\Big\|_{H^{s}}\\
&\leq D_{\psi,\kappa}\epsilon\|R\|_{H^{s}}.
\end{split}
\end{equation}
\end{lemma}
\begin{proof}
For (i), we have
 \begin{equation*}
\begin{split}
\Big\|\int&\kappa(\cdot,\cdot-m,m)\widehat{\psi}\big(\frac{\cdot-m-p}{\epsilon}\big)\widehat{R}(m)dm
-\int\kappa(\cdot,\cdot-m,\cdot-p)\widehat{\psi}\big(\frac{\cdot-m-p}{\epsilon}\big)\widehat{R}(m)dm\Big\|^{2}_{H^{s}}\\
&=\int\Big(\int(\kappa(k,k-m,m)-\kappa(k,p,m))\widehat{\psi}(\frac{k-m-p}{\epsilon})\widehat{R}(m)dm\Big)^{2}(1+k^{2})^{s}dk\\
&\leq\int\Big(C_{\kappa}\int|k-m-p|\widehat{\psi}(\frac{k-m-p}{\epsilon})\widehat{R}(m)dm\Big)^{2}(1+k^{2})^{s}dk\\
&\leq C_{\kappa}^{2}\Big(\int(1+\ell)^{s/2}|\ell||\widehat{\psi}(\frac{\ell}{\epsilon})|d\ell\Big)^{2}\|R\|^{2}_{H^{s}}\\
&\leq C_{\psi,\kappa,p}\epsilon^{2}\|R\|^{2}_{H^{s}},
\end{split}
\end{equation*}
thanks to the Young's inequality and the fact that $\widehat{\psi}$ has compact support. The second one (ii) is similar.
\end{proof}

We now construct the second normal form transformation to remove the remaining term $\epsilon^{2}B_{j_{1},j_{2}}^{0,1,n}\big(\phi_{c},\frac{P^{1}\partial_{x}}{2\vartheta}\alpha_{j_{2},j_{3}}^{\tilde{n}}(\phi_{c}\vartheta R_{j_{3}}^{1})\big)$ with $j_{1},j_{2},j_{3}\in\{\pm1\}$ from \eqref{equ6}. First recall that $\psi_{\pm1}$ is supported in a neighborhood of size $\delta$ of $\pm k_{0}$ in Fourier space and $\phi_{c}=\psi_{1}+\psi_{-1}$. Thus
\begin{equation}
\begin{split}\label{equ12}
&\epsilon^{2} B_{j_{1},j_{2}}^{0,1,n}\big(\phi_{c}, \frac{P^{1}\partial_{x}}{2\vartheta}\alpha_{j_{2},j_{3}}^{\tilde{n}}(\phi_{c}\vartheta R_{j_{3}}^{1})\big)\\
=&\epsilon^{2} B_{j_{1},j_{2}}^{0,1,n}\big(\psi_{1},\frac{P^{1}\partial_{x}}{2\vartheta}\alpha_{j_{2},j_{3}}^{\tilde{n}}(\psi_{1}\vartheta R_{j_{3}}^{1})\big)
+\epsilon^{2} B_{j_{1},j_{2}}^{0,1,n}\big(\psi_{-1},\frac{P^{1}\partial_{x}}{2\vartheta}\alpha_{j_{2},j_{3}}^{\tilde{n}}(\psi_{-1}\vartheta R_{j_{3}}^{1})\big)\\
&+\epsilon^{2} B_{j_{1},j_{2}}^{0,1,n}\big(\psi_{1},\frac{P^{1}\partial_{x}}{2\vartheta}\alpha_{j_{2},j_{3}}^{\tilde{n}}(\psi_{-1}\vartheta R_{j_{3}}^{1})\big)
+\epsilon^{2} B_{j_{1},j_{2}}^{0,1,n}\big(\psi_{-1},\frac{P^{1}\partial_{x}}{2\vartheta}\alpha_{j_{2},j_{3}}^{\tilde{n}}(\psi_{1}\vartheta R_{j_{3}}^{1})\big).
\end{split}
\end{equation}
Recalling from \eqref{equ7} that
\begin{equation*}
\begin{split}
\widehat{b}_{j_{1},j_{2}}^{0,1,n}(k,k-\ell,\ell)
=\frac{ -k\widehat{P}^{0}(k)\widehat{\alpha}_{j_{1},j_{2}}^{n}(k,k-\ell,\ell)}{-j_{1}\omega(k)
-\omega(k-\ell)+j_{2}\omega(\ell)}\frac{\widehat{\vartheta}(\ell)}{2\widehat{\vartheta}(k)},
\end{split}
\end{equation*}
each of the four terms on the RHS of \eqref{equ12} can be rewritten as
\begin{equation}
\begin{split}\label{equ13}
\epsilon^{2} B_{j_{1},j_{2}}^{0,1,n}&(\psi_{l},\frac{P^{1}\partial_{x}}{2\vartheta}\alpha_{j_{2},j_{3}}^{\tilde{n}}(\psi_{\nu}\vartheta R_{j_{3}}^{1}))\\
=&\frac{\epsilon^{2}}{2}\sum_{j_{1},j_{2}\in\{1,-1\}}\int\widehat{b}_{j_{1},j_{2}}^{0,1,n}(k,k-\ell,\ell)\widehat{\psi}_{l}(k-\ell)\\
&\times\widehat{\vartheta}^{-1}(\ell)\widehat{P}^{1}(\ell)i\ell\Big(
\int\widehat{\alpha}_{j_{2},j_{3}}^{\tilde{n}}(\ell,\ell-m,m)\widehat{\psi}_{\nu}(\ell-m)
\widehat{\vartheta}(m)\widehat{R}_{j_{3}}^{1}(m)dm\Big)d\ell,
\end{split}
\end{equation}
where $l,\nu\in\{+,-\}$. Applying Lemma \ref{L6} to obtain
\begin{equation}
\begin{split}\label{equation85}
\frac{\epsilon^{2}}{2} \widehat{B}_{j_{1},j_{2}}^{0,1,n}&\Big(\psi_{l}, \  P^{1}(\cdot-lk_{0})i(\cdot-lk_{0})\vartheta^{-1}(\cdot-lk_{0})\alpha^{\tilde{n}}_{j_{2},j_{3}}(\cdot,\cdot-lk_{0},\cdot-nk_{0})
\psi_{\nu}\vartheta(\cdot-\nu k_{0})R_{j_{3}}^{1}\Big)(k)\\
=&\frac{\epsilon^{2}}{2}\int\widehat{b}_{j_{1},j_{2}}^{0,1,n,l,\nu}(k)\widehat{\psi}_{l}(k-\ell)
\vartheta^{-1}(k-lk_{0})\widehat{P}^{1}(k-lk_{0})i(k-lk_{0})\\
&\times\Big(
\int\widehat{\alpha}_{j_{2},j_{3}}^{\tilde{n}}\big(k-lk_{0},lk_{0},k-(l+\nu)k_{0}\big)\widehat{\psi}_{\nu}(\ell-m)
\widehat{\vartheta}\big(k-(l+\nu)k_{0}\big)\widehat{R}_{j_{3}}^{1}(m)dm\Big)d\ell\\
&+\epsilon^{2}\mathcal{F}^{4},
\end{split}
\end{equation}
where $\epsilon^{2}\mathcal{F}^{4}$ has the same property with $\epsilon^{2}\mathcal{F}^{1}$, and $l+\nu$ is interpreted as if $l$ and $\nu$ were $+1$ and $-1$. We also use the abbreviation
\begin{equation*}
\begin{split}
\widehat{b}_{j_{1},j_{2}}^{0,1,n,l,\nu}(k)
=\frac{ -k\widehat{P}^{0}(k)\widehat{\alpha}_{j_{1},j_{2}}^{n}(k,lk_{0},k-lk_{0})}{-j_{1}\omega(k)
-\omega(lk_{0})+j_{2}\omega(k-lk_{0})}\frac{\widehat{\vartheta}\big(k-(l+\nu)k_{0}\big)}{2\widehat{\vartheta}(k)}.
\end{split}
\end{equation*}
\begin{lemma}\label{L7}
There exists $C>0$ such that
\begin{equation*}
\begin{split}
&\Big\|\epsilon^{2} B_{j_{1},j_{2}}^{0,1,n}\big(\psi_{1}, P^{1}\partial_{x}\vartheta^{-1}\alpha_{j_{2},j_{3}}^{\tilde{n}}(\psi_{-1}\vartheta R_{j_{3}}^{1})\big)\Big\|_{H^{s}}
\leq C\epsilon^{2}\|R_{j_{3}}^{1}\|_{H^{s}},\\
&\Big\|\epsilon^{2} B_{j_{1},j_{2}}^{0,1,n}(\psi_{-1}, P^{1}\partial_{x}\vartheta^{-1}\alpha_{j_{2},j_{3}}^{\tilde{n}}(\psi_{1}\vartheta R_{j_{3}}^{1})\big)\Big\|_{H^{s}}
\leq C\epsilon^{2}\|R_{j_{3}}^{1}\|_{H^{s}},
\end{split}
\end{equation*}
with $j_{1},j_{2},j_{3}\in\{\pm1\}$.
\end{lemma}
\begin{proof}
Since $B_{j_{1},j_{2}}^{0,1,n}$ contains the factor $\widehat{P}^{0}(k)$, the integral over $k$ occurring in the $H^{s}$ norm runs only over $|k|<\delta$. Thus we can bound the $H^{s}$ norm by bounding the maximum of the kernel. The first term in Lemma \ref{L7} has the modified kernel
\begin{equation}
\begin{split}\label{equation86}
\epsilon^{2}\widehat{b}_{j_{1},j_{2}}^{0,1,1,+,-}(k)
\widehat{\vartheta}^{-1}(k-k_{0})\widehat{P}^{1}(k-k_{0})i(k-k_{0})\widehat{\alpha}_{j_{2},j_{3}}^{\tilde{n}}(k-k_{0},k_{0},k)\widehat{\vartheta}(k).
\end{split}
\end{equation}
Since $\widehat{\vartheta}(k)\widehat{b}_{j_{1},j_{2}}^{0,1,n,+,-}(k)$ is $\mathcal{O}(1)$ bounded, and all other terms in \eqref{equation86} are $\mathcal{O}(1)$ bounded for $|k|<\delta$, we have an $\mathcal{O}(\epsilon^{2})$ bound for the kernel \eqref{equation86}. The second term in Lemma \ref{L7} can be estimated similarly.
\end{proof}

Thanks to Lemma \ref{L7}, we do not need to eliminate the third and fourth terms in \eqref{equ12} by the normal form transformation. Therefore we turn to the first term in \eqref{equ12}, whose modified kernel has the form 
\begin{equation}
\begin{split}\label{equation87}
\frac{\epsilon^{2}}{2}\widehat{b}_{j_{1},j_{2}}^{0,1,n,+,+}(k)
\widehat{\vartheta}^{-1}(k-k_{0})\widehat{P}^{1}(k-k_{0})i(k-k_{0})\widehat{\alpha}_{j_{2},j_{3}}^{\tilde{n}}(k-k_{0},k_{0},k-2k_{0})\widehat{\vartheta}(k-2k_{0}),
\end{split}
\end{equation}
plus some $\mathcal{O}(\epsilon^{2})$ error terms. A similar expression for the kernel of the second term in \eqref{equ12} can be obtained. In contrast to the terms considered in Lemma \ref{L7}, this expression does not contain a factor of $\widehat{\vartheta}(k)$ to offset the $\widehat{\vartheta}(k)$ in the denominator of $\widehat{b}_{j_{1},j_{2}}^{0,1,n,+,+}(k)$ and hence they must be eliminated by a second normal form transformation. We look for a transformation of the form
\begin{equation}
\begin{split}\label{equation88}
\mathcal{R}_{j_{1}}^{0}:=\widetilde{R}_{j_{1}}^{0}+\epsilon\sum_{j_{2},j_{3}\in\{\pm1\}}\sum_{n,\tilde{n}=1}^{5} \big(D_{j_{1},j_{2},j_{3}}^{n,\tilde{n},+}(\psi_{1},\psi_{1},R_{j_{3}}^{1})+ D_{j_{1},j_{2},j_{3}}^{n,\tilde{n},-}(\psi_{-1},\psi_{-1},R_{j_{3}}^{1})\big).
\end{split}
\end{equation}
Differentiating the expression for $\mathcal{R}_{j_{1}}^{0}$, we find that the terms of $\mathcal{\mathcal{O}(\epsilon)}$ in \eqref{equ6} will be eliminated if $D_{j_{1},j_{2},j_{3}}^{n,\tilde{n},+}(\psi_{1},\psi_{1},R_{j_{3}}^{1})$ satisfies
\begin{equation}
\begin{split}\label{equation89}
-j_{1}\Omega& D_{j_{1},j_{2},j_{3}}^{n,\tilde{n},+}(\psi_{1},\psi_{1},R_{j_{3}}^{1})-D_{j_{1},j_{2},j_{3}}^{n,\tilde{n},+}(\Omega\psi_{1},\psi_{1},R_{j_{3}}^{1})
-D_{j_{1},j_{2},j_{3}}^{n,\tilde{n},+}(\psi_{1},\Omega\psi_{1},R_{j_{3}}^{1})\\
&+D_{j_{1},j_{2},j_{3}}^{n,\tilde{n},+}(\psi_{1},\psi_{1},j_{3}\Omega R_{j_{3}}^{1})
=-\frac{\epsilon}{2}B_{j_{1},j_{2}}^{0,1,n}(\psi_{1},\vartheta^{-1}P^{1}\partial_{x}\alpha_{j_{2},j_{3}}^{\tilde{n}}\big(\psi_{1}\vartheta R_{j_{3}}^{1})\big),
\end{split}
\end{equation}
with similar expression for $D_{j_{1},j_{2},j_{3}}^{n,\tilde{n},-}$. We find that we have to set
\begin{equation}
\begin{split}\label{equation90}
\epsilon &\widehat{D}_{j_{1},j_{2},j_{3}}^{n,\tilde{n},+}(\psi_{1},\psi_{1},R_{j_{3}}^{1})\\
=&\frac{\epsilon^{2}}{2}\int_{\mathbb{R}}\widehat{b}_{j_{1},j_{2}}^{0,1,n,+,+}(k)\widehat{\psi}_{1}(k-\ell)
\vartheta^{-1}(k-k_{0})\widehat{P}^{1}(k-k_{0})(k-k_{0})\\
&\times\Big(
\int_{\mathbb{R}}\frac{-\widehat{\alpha}_{j_{2},j_{3}}^{\tilde{n}}(k-k_{0},k_{0},k-2k_{0})\widehat{\psi}_{1}(\ell-m)
\widehat{\vartheta}(k-2k_{0})\widehat{R}_{j_{3}}^{1}(m)}{-j_{1}\omega(k)-\omega(k_{0})
-\omega(k_{0})+j_{3}\omega(k-2k_{0})}dm\Big)d\ell,
\end{split}
\end{equation}
where we have used in the kernel that $k-\ell\approx\ell-m\approx k_{0}$ due to the localization of $\psi_{1}$ so we have $m\approx-2k_{0}$, which is made rigorous with Lemma \ref{L6}. We have to estimate the kernel w.r.t. the sup norm, thanks to the Young's inequality. Since the numerator in this expression is already $\mathcal{O}(\varepsilon)$, the denominator satisfies
\begin{equation*}
\begin{split}
-j_{1}\omega(k)-\omega(k_{0})-\omega(k_{0})+j_{3}\omega(k-2k_{0})\approx-2\omega(k_{0})-j_{3}\omega(2k_{0})\neq0.
\end{split}
\end{equation*}
Therefore the mapping $\epsilon D_{j_{1},j_{2},j_{3}}^{n,\tilde{n},+}$ is well-defined and $\mathcal{O}(\epsilon)$-bounded. The expression for $\epsilon D_{j_{1},j_{2},j_{3}}^{n,\tilde{n},-}$ can be constructed and estimated in a very similar fashion and hence, the normal form is well defined and invertible. We also have for any $s\geq6$, there exists $C>0$ such that
\begin{equation}
\begin{split}\label{equ1,19}
\epsilon\big\|D_{j_{1},j_{2},j_{3}}^{n,\tilde{n},\pm}(\psi_{\pm1},\psi_{\pm1},R_{j_{3}}^{1})\big\|_{H^{s}}
\leq C\epsilon\big\|R_{j_{3}}^{1}\big\|_{H^{s}}.
\end{split}
\end{equation}

Finally, we consider the composition of the two normal form transformations \eqref{equ3} and \eqref{equation88}, i.e.,
\begin{equation}
\begin{split}\label{equation92}
\mathcal{R}_{j_{1}}^{0}&=\widetilde{R}_{j_{1}}^{0}+\epsilon \sum_{j_{2},j_{3}\in\{\pm1\}}\sum_{n,\tilde{n}=1}^{5}\big(D_{j_{1},j_{2},j_{3}}^{n,\tilde{n},+}(\psi_{1},\psi_{1},R_{j_{3}}^{1})+\epsilon D_{j_{1},j_{2},j_{3}}^{n,\tilde{n},-}(\psi_{-1},\psi_{-1},R_{j_{3}}^{1})\big)\\
&=R_{j_{1}}^{0}+\epsilon \sum_{j_{2}\in\{\pm1\}}\sum_{n=1}^{5}B_{j_{1},j_{2}}^{0,1,n}(\phi_{c},R_{j_{2}}^{1})\\
& \ \ \ \ \ \ \ \ \ \  +\epsilon \sum_{j_{2},j_{3}\in\{\pm1\}}\sum_{n,\tilde{n}=1}^{5}\big(D_{j_{1},j_{2},j_{3}}^{n,\tilde{n},+}(\psi_{1},\psi_{1},R_{j_{3}}^{1})+\epsilon D_{j_{1},j_{2},j_{3}}^{n,\tilde{n},-}(\psi_{-1},\psi_{-1},R_{j_{3}}^{1})\big)\\
&= R_{j_{1}}^{0}+\epsilon F_{j_{1}}(R^{1}),
\end{split}
\end{equation}
where $\|\epsilon \vartheta F(R^{1})\|_{H^{s'}}\lesssim\|\epsilon R^{1}\|_{H^{s}}$ for any $s', s\geq6$ due to \eqref{equ9}, \eqref{equ1,1} and \eqref{equ1,19}.
Then we have
\begin{equation}
\begin{split}\label{equation91}
\partial_{t}\mathcal{R}^{0}=\Lambda\mathcal{R}^{0}+\epsilon^{2}\mathcal{F}^{5},
\end{split}
\end{equation}
where $\epsilon^{2}\mathcal{F}^{5}$ has the same property with $\epsilon^{2}\mathcal{F}^{1}$.
Let
\begin{equation}
\begin{split}\label{phi3}
\phi_{3}:=&\phi_{p_{1}}+\epsilon^{1/2}\big(\vartheta (R_{1}^{0}+R_{-1}^{0})+\frac{1}{2}(R_{1}^{1}+R_{-1}^{1})\big),\\
\phi_{4}:=&\phi_{p_{2}}+\epsilon^{1/2}\big(\vartheta (R_{1}^{0}-R_{-1}^{0})+\frac{1}{2}(R_{1}^{1}-R_{-1}^{1})\big).\\
\end{split}
\end{equation}
Recall $\vartheta_{0}=\vartheta-\epsilon$ and rewrite the equation \eqref{equ1,-1} as
\begin{align}\label{equ11-1}
\partial_{t}\widehat{R}_{j_{1}}^{1}
=&ij_{1}\omega(k)\widehat{R}_{j_{1}}^{1}+\epsilon\frac{\widehat{P}^{1}(k)ik}{2\widehat{\vartheta}(k)}
\sum_{j_{2}\in\{\pm1\}}\sum_{n=1}^{5}\widehat{\alpha}_{j_{1},j_{2}}^{n}(k,k-m,m)(\widehat{\phi}_{c}\ast \widehat{\vartheta}_{0}\widehat{\mathcal{R}}_{j_{2}}^{0})\nonumber\\
& \ \ +\epsilon\frac{\widehat{P}^{1}(k)ik}{2\widehat{\vartheta}(k)}\sum_{j_{2}\in\{\pm1\}}
\sum_{n=1}^{5}\widehat{\alpha}_{j_{1},j_{2}}^{n}(k,k-m,m)\big(\widehat{\phi}_{c}\ast \widehat{\vartheta}\widehat{R}_{j_{2}}^{1}\big)\nonumber\\
&+\epsilon^{2}j_{1}\frac{\widehat{P}^{1}(k)ik}{2\widehat{\vartheta}(k)}\widehat{\phi}_{3}\ast\big(\widehat{q}\widehat{\vartheta}
(\widehat{R}_{j_{1}}^{1}-\widehat{R}_{-j_{1}}^{1})\big)
+\epsilon^{2}\frac{\widehat{P}^{1}(k)ik}{2\widehat{\vartheta}(k)}\widehat{q}\widehat{\phi}_{4}\ast\big(\widehat{\vartheta}
(\widehat{R}_{j_{1}}^{1}+\widehat{R}_{-j_{1}}^{1})\big)\nonumber\\
&+\epsilon^{2}\frac{\widehat{P}^{1}(k)ik}{2\widehat{\vartheta}(k)\widehat{q}(k)}\widehat{q}\widehat{\phi}_{4}\ast\big(\widehat{q}\widehat{\vartheta}
(\widehat{R}_{j_{1}}^{1}-\widehat{R}_{-j_{1}}^{1})\big)
-\epsilon^{2}j_{1}\frac{\widehat{P}^{1}(k)ik}{2\widehat{\vartheta}(k)\widehat{q}(k)}\widehat{\phi}_{3}\ast\big(\widehat{\vartheta}
(\widehat{R}_{j_{1}}^{1}+\widehat{R}_{-j_{1}}^{1})\big)\nonumber\\
&-\epsilon^{2}\frac{\widehat{P}^{1}(k)ij_{1}k}{2\widehat{\vartheta}(k)\langle k\rangle^{2}\widehat{q}(k)}\big(\langle\widehat{\partial_{x}}\rangle^{-2}\widehat{\phi}_{3}\big)\ast \big(\langle\widehat{\partial_{x}}\rangle^{-2}\widehat{\vartheta}
(\widehat{R}_{j_{1}}^{1}+\widehat{R}_{-j_{1}}^{1})\big)\nonumber\\
&+\epsilon^{2}\frac{\widehat{P}^{1}(k)ij_{1}k}{2\widehat{\vartheta}(k)\widehat{q}(k)}\widehat{\mathcal{G}}
\widehat{\vartheta}\big(\widehat{R}_{j_{1}}^{1}+\widehat{R}_{-j_{1}}^{1}\big)
+\epsilon^{2}\frac{\widehat{P}^{1}(k)ij_{1}k}{2\widehat{\vartheta}(k)\widehat{q}(k)}\widehat{\mathcal{M}}
\widehat{\vartheta}\big(\widehat{R}_{j_{1}}^{1}+\widehat{R}_{-j_{1}}^{1}\big)
+\epsilon^{2}\mathcal{F}_{j_{1}}^{5}\nonumber\\
&+\epsilon^{-5/2}\widehat{Res}_{U_{j_{1}}^{1}}(\epsilon\Psi)\nonumber\\
=:&ij_{1}\omega(k)\widehat{R}_{j_{1}}^{1}+\epsilon\frac{\widehat{P}^{1}(k)ik}{2\widehat{\vartheta}(k)}
\sum_{j_{2}\in\{\pm1\}}\sum_{n=1}^{5}\widehat{\alpha}_{j_{1},j_{2}}^{n}(k,k-m,m)\big(\widehat{\phi}_{c}\ast \widehat{\vartheta}_{0}\widehat{\mathcal{R}}_{j_{2}}^{0}\big)\nonumber\\
& \ \ +\epsilon\frac{\widehat{P}^{1}(k)ik}{2\widehat{\vartheta}(k)}\sum_{j_{2}\in\{\pm1\}}
\sum_{n=1}^{5}\widehat{\alpha}_{j_{1},j_{2}}^{n}(k,k-m,m)\big(\widehat{\phi}_{c}\ast \widehat{\vartheta}\widehat{R}_{j_{2}}^{1}\big)\nonumber\\
&+(\sum_{m=1}^{7}\widehat{G}_{m})_{j_{1}}+\epsilon^{2}\mathcal{F}_{j_{1}}^{5}+\epsilon^{-5/2}\widehat{Res}_{U_{j_{1}}^{1}}(\epsilon\Psi),
\end{align}
where $\|\mathcal{F}^{5}\|_{H^{s'}}\leq C\|\mathcal{R}^{0},R^{1}\|_{H^{s}}$ for any $s, \ s'>0$.
Recall that $\widehat{\vartheta}(k)=1$ for $|k|>\delta$ and \eqref{phi3}, let $\phi_{1}:=\phi_{c}+\epsilon\phi_{3}$ and $\phi_{2}:=\phi_{c}+\epsilon\phi_{4}$ and insert \eqref{equation92} into \eqref{equ1,-1},  we have
\begin{equation*}
\begin{split}
\phi_{1}=&\phi_{c}+\epsilon\phi_{p_{1}}+\epsilon^{3/2}\Big[\vartheta\big(\mathcal{R}_{1}^{0}+\mathcal{R}_{-1}^{0}
-\epsilon F_{1}(R^{1})-\epsilon F_{-1}(R^{1})
+\frac{1}{2}(R_{1}^{1}+R_{-1}^{1})\Big],\\
\phi_{2}
=&\phi_{c}+\epsilon\phi_{p_{2}}+\epsilon^{3/2}\Big[\vartheta\big(\mathcal{R}_{1}^{0}-\mathcal{R}_{-1}^{0}
-\epsilon F_{1}(R^{1})+\epsilon F_{-1}(R^{1})
+\frac{1}{2}(R_{1}^{1}-R_{-1}^{1})\Big],
\end{split}
\end{equation*}
then we have
\begin{align}\label{R1}
\partial_{t}R_{j_{1}}^{1}=&j_{1}\Omega R_{j_{1}}^{1}
+\epsilon j_{1}\frac{P^{1}\partial_{x}}{2}\big(\phi_{1} q(R_{j_{1}}^{1}-R_{-j_{1}}^{1})\big)
+\epsilon\frac{P^{1}\partial_{x}}{2}\big(q\phi_{2} (R_{j_{1}}^{1}+R_{-j_{1}}^{1})\big)\nonumber\\
&+\epsilon \frac{P^{1}\partial_{x}}{2q}\big(q\phi_{2} q(R_{j_{1}}^{1}-R_{-j_{1}}^{1})\big)
-\epsilon j_{1}\frac{P^{1}\partial_{x}}{2q}\big(\phi_{1} (R_{j_{1}}^{1}+R_{-j_{1}}^{1})\big)\nonumber\\
&-\epsilon j_{1}\frac{P^{1}\partial_{x}\langle\partial_{x}\rangle^{-2}}{2q}\big(\langle\partial_{x}\rangle^{-2}\phi_{1} (\langle\partial_{x}\rangle^{-2}(R_{j_{1}}^{1}+R_{-j_{1}}^{1})\big)\nonumber\\
&+\epsilon^{2}j_{1}\frac{P^{1}\partial_{x}}{2q}\mathcal{G}(R_{j_{1}}^{1}+R_{-j_{1}}^{1})
+\epsilon^{2}j_{1}\frac{P^{1}\partial_{x}}{2q}\mathcal{M}(R_{j_{1}}^{1}+R_{-j_{1}}^{1})\nonumber\\
&+\epsilon^{2}\mathcal{F}_{j_{1}}^{6}
+\epsilon^{-5/2}Res_{U_{j_{1}}^{1}}(\epsilon\Psi)\nonumber\\
=:&j_{1}\Omega R_{j_{1}}^{1}+\epsilon(\sum_{m=1}^{7}D_{m})_{j_{1}}
+\epsilon^{2}\mathcal{F}_{j_{1}}^{6}
+\epsilon^{-5/2}Res_{U_{j_{1}}^{1}}(\epsilon\Psi),
\end{align}
where for arbitrary $s', \ s\geq6$, we have
\begin{equation*}
\begin{split}
\|\mathcal{G}\|_{H^{s}}&\lesssim\|\mathcal{R}^{0},R^{1}\|^{2}_{H^{s}},\\
\|\mathcal{M}\|_{H^{s+2}}&\lesssim\|\mathcal{R}^{0},R^{1}\|^{2}_{H^{s}},\\
\|\mathcal{F}^{6}\|_{H^{s'}}&\lesssim\|\mathcal{R}^{0},R^{1}\|_{H^{s}}.
\end{split}
\end{equation*}

\section{The error estimates}
From the above procedure, we have obtained the equation \eqref{equation91} and \eqref{equ11-1} (or \eqref{R1}) for $(\mathcal{R}^{0},R^{1})$. In order to control the error we define the energy
\begin{equation*}
\begin{split}
\mathcal{E}_{s}=\sum_{\ell=0}^{s}E_{\ell},
\end{split}
\end{equation*}
with
\begin{equation}
\begin{split}\label{equ14}
E_{\ell}=&\sum_{j_{1}\in\{\pm1\}}\bigg[\frac{1}{2}\Big(\int_{\mathbb{R}}(\partial_{x}^{\ell}\mathcal{R}_{j_{1}}^{0})^{2}dx+
\int_{\mathbb{R}}(\partial_{x}^{\ell}R_{j_{1}}^{1})^{2}dx\Big)\\
&+\epsilon\sum_{j_{2}\in\{\pm1\}}\sum_{n=1}^{5}\Big(
\int_{\mathbb{R}}\partial_{x}^{\ell}R_{j_{1}}^{1}\partial_{x}^{\ell}B_{j_{1},j_{2}}^{1,0,n}(\phi_{c},\mathcal{R}_{j_{2}}^{0})dx
+\int_{\mathbb{R}}\partial_{x}^{\ell}R_{j_{1}}^{1}\partial_{x}^{\ell}B_{j_{1},j_{2}}^{1,1,n}(\phi_{c},R_{j_{2}}^{1})dx\Big)\bigg],
\end{split}
\end{equation}
where
\begin{equation}
\begin{split}\label{B10}
\widehat{B}_{j_{1},j_{2}}^{1,0,n}(\phi_{c},\mathcal{R}_{j_{2}}^{0})
=\int_{\mathbb{R}}\widehat{b}_{j_{1},j_{2}}^{1,0,n}(k,k-m,m)\widehat{\phi}_{c}(k-m)\widehat{\mathcal{R}}_{j_{2}}^{0}(m)dm,
\end{split}
\end{equation}
and
\begin{equation}
\begin{split}\label{B1,1}
\widehat{B}_{j_{1},j_{2}}^{1,1,n}(\phi_{c},R_{j_{2}}^{1})
=\int_{\mathbb{R}}\widehat{b}_{j_{1},j_{2}}^{1,1,n}(k,k-m,m)\widehat{\phi}_{c}(k-m)\widehat{R}_{j_{2}}^{1}(m)dm,
\end{split}
\end{equation}
with
\begin{equation}
\begin{split}\label{b10}
b_{j_{1},j_{2}}^{1,0,n}=&\frac{-\frac{1}{2}k\widehat{P}^{1}(k)\widehat{\alpha}_{j_{1},j_{2}}^{n}(k,k-m,m)}{-j_{1}\omega(k)-\omega(k-m)+j_{2}\omega(m)}
\frac{\widehat{\vartheta}_{0}(m)}{\widehat{\vartheta}(k)}\\
=&\frac{-\frac{1}{2}k\widehat{\alpha}_{j_{1},j_{2}}^{n}(k,k-m,m)}{-j_{1}\omega(k)-\omega(k-m)+j_{2}\omega(m)}\widehat{\vartheta}_{0}(m),
\end{split}
\end{equation}
for $|k|>\delta$, $|m|<\delta$ and $|k-m\pm k_{0}|<\delta$,
and
\begin{equation}
\begin{split}\label{b1,1}
b_{j_{1},j_{2}}^{1,1,n}=&\frac{-\frac{1}{2}k\widehat{P}^{1}(k)\widehat{\alpha}_{j_{1},j_{2}}^{n}(k,k-m,m)}{-j_{1}\omega(k)-\omega(k-m)+j_{2}\omega(m)}
\frac{\widehat{\vartheta}(m)}{\widehat{\vartheta}(k)}\\
=&\frac{-\frac{1}{2}k\widehat{\alpha}_{j_{1},j_{2}}^{n}(k,k-m,m)}{-j_{1}\omega(k)-\omega(k-m)+j_{2}\omega(m)},
\end{split}
\end{equation}
for $|k|>\delta$, $|m|>\delta$ and $|k-m\pm k_{0}|<\delta$.
For the normal form transformation $B^{1,0,n}$, we have the following good properties.
\begin{lemma}\label{L10}
For the operators $B_{j_{1},j_{2}}^{1,0,n}(\phi_{c},\mathcal{R}_{j_{2}}^{0})$ for $1\leq n\leq5$, there exists a constant $C$ such that for any $s, s'\geq6$, we have\\
(a)
\begin{equation}
\begin{split}\label{154}
\epsilon\|B_{j_{1},j_{2}}^{1,0,n}(\phi_{c},\mathcal{R}_{j_{2}}^{0})\|_{H^{s'}}\lesssim\epsilon\|\mathcal{R}_{j_{2}}^{0}\|_{H^{s}},
\end{split}
\end{equation}\\
(b) and for all $f\in H^{1}(\mathbb{R},\mathbb{R})$, we have
\begin{equation}
\begin{split}\label{16'}
-j_{1}\Omega B_{j_{1},j_{2}}^{1,0,n}(\phi_{c},f)-B_{j_{1},j_{2}}^{1,0,n}(\Omega\phi_{c},f)
+j_{2} B_{j_{1},j_{2}}^{1,0,n}(\phi_{c},\Omega f)
=-\frac{P^{1}\partial_{x}}{2\vartheta}\alpha_{j_{1},j_{2}}^{n}(\phi_{c}f),
\end{split}
\end{equation}
where the operator $\Omega$ satisfies $\widehat{\Omega u}(k)=i\omega(k)\widehat{u}(k)$, $\omega$ is defined by the symbols \eqref{equation3}.
\end{lemma}
\begin{proof}
By Lemma \ref{L6} and recalling $\phi_{c}=\psi_{1}+\psi_{-1}$, we can rewrite
\begin{equation*}
\begin{split}
\widehat{B}_{j_{1},j_{2}}^{1,0,n}(\phi_{c},\mathcal{R}_{j_{2}}^{0})
=&\widehat{B}_{j_{1},j_{2}}^{1,0,n}(\psi_{1},\mathcal{R}_{j_{2}}^{0})+\widehat{B}_{j_{1},j_{2}}^{1,0,n}(\psi_{-1},\mathcal{R}_{j_{2}}^{0})\\
=&\int_{\mathbb{R}}\widehat{b}_{j_{1},j_{2}}^{1,0,n,+}(k)\widehat{\psi}_{1}(k-m)\widehat{\mathcal{R}}_{j_{2}}^{0}(m)dm\\
&+\int_{\mathbb{R}}\widehat{b}_{j_{1},j_{2}}^{1,0,n,-}(k)\widehat{\psi}_{-1}(k-m)\widehat{\mathcal{R}}_{j_{2}}^{0}(m)dm +\epsilon^{2}\mathcal{F}^{8},
\end{split}
\end{equation*}
where
\begin{equation}
\begin{split}\label{155}
\widehat{b}_{j_{1},j_{2}}^{1,0,n,+}(k)
=\frac{-\frac{1}{2}k\widehat{P}^{1}(k)\widehat{\alpha}_{j_{1},j_{2}}^{n}(k,k_{0},k-k_{0})}{-j_{1}\omega(k)-\omega(k_{0})+j_{2}\omega(k-k_{0})}
\frac{\widehat{\vartheta}_{0}(k-k_{0})}{\widehat{\vartheta}(k)},
\end{split}
\end{equation}
with a similar expression for $\widehat{b}_{j_{1},j_{2}}^{1,0,n,-}$ and the notation $\epsilon^{2}\mathcal{F}^{8}$ has the same property with $\epsilon^{2}\mathcal{F}^{1}$.
Due to the support properties of $\widehat{\phi}_{c}$ and the definitions of $\widehat{P}^{1}$ and $\widehat{\vartheta}(k)=1$, the expression \eqref{155} only has to be analysed for $|k-k_{0}|<\delta$ and $|k|>\delta$. So there is only a resonance at $k=k_{0}$ with $j_{1}=-1$ in the denominator of \eqref{155}, since $\widehat{P}^{1}(k)=0$ for $|k|<\delta$, while the resonance at $k=0$ does not play an essential role in the analysis of $B_{j_{1},j_{2}}^{1,0,n}$. However,  since the derivative of $\omega$ at $k_{0}$ is $\mathcal{O}(1)$, we have a bound on the denominator of the form
\begin{equation}
\begin{split}\label{156}
|-j_{1}\omega(k)-\omega(k_{0})+j_{2}\omega(k-k_{0})|\geq C|k-k_{0}|.
\end{split}
\end{equation}
This singularity is offset thanks to $|\widehat{\vartheta}_{0}(k-k_{0})|\leq C|k-k_{0}|$ and hence the kernel $\widehat{b}_{j_{1},j_{2}}^{1,0,n,+}$ can be extended continuously at $k=k_{0}$ with an $\mathcal{O}(1)$ bound on its size. Thus the kernel can be bounded by an $\mathcal{O}(1)$ bound for all values of $k$ and $m$.
Applying Young's inequality, due to the compact support of $\widehat{\mathcal{R}}^{0}$, the loss of regularity is not present in the estimate for $B_{j_{1},j_{2}}^{1,0,n}$ and we have the estimate \eqref{154}. (b) is a direct consequence of the construction of $B_{j_{1},j_{2}}^{1,0,n}$.
\end{proof}

Particularly, according to the form of $\widehat{\alpha}_{j_{1},j_{2}}^{5}(k,k-m,m)$ in the equation \eqref{equ1,-1}, we have
\begin{equation}
\begin{split}\label{b1,15}
b_{j_{1},j_{2}}^{1,1,5}=&\frac{j_{1}k\widehat{P}^{1}(k)}{-j_{1}\omega(k)-\omega(k-m)+j_{2}\omega(m)}\frac{1}{\langle k\rangle^{2}\langle k-m\rangle^{2}\langle m\rangle^{2}}
\frac{\widehat{\vartheta}(m)}{2\widehat{\vartheta}(k)}\\
=&\frac{j_{1}k}{-j_{1}\omega(k)-\omega(k-m)+j_{2}\omega(m)}\frac{1}{\langle k\rangle^{2}\langle k-m\rangle^{2}\langle m\rangle^{2}},
\end{split}
\end{equation}
for $|k|>\delta$, $|m|>\delta$ and $|k-m\pm k_{0}|<\delta$.
Thus we have the following estimates
\begin{equation}
\begin{split}\label{B1,15}
\epsilon\|B_{j_{1},j_{2}}^{1,1,5}(\phi_{c},R_{j_{2}}^{1})\|_{H^{s+3}}\lesssim\epsilon\|R_{j_{2}}^{1}\|_{H^{s}},
\end{split}
\end{equation}
provided with $R_{j_{2}}^{1}\in H^{s}$ for $s\geq6$ and $j_{1}, \ j_{2}\in\{\pm1\}$.

However, there is a loss of one derivative due to the growth of $k\widehat{\alpha}_{j_{1},j_{2}}^{n}\sim k$ as $|k|\rightarrow\infty$ for $1\leq n\leq4$ and hence, we can only bound $B_{j_{1},j_{2}}^{1,1,n}$ in $H^{s-1}$ rather than $H^{s}$ if $R_{j_{2}}^{1}\in H^{s}$ for $1\leq n\leq4$. In the following we analyse carefully the construction of the normal form transformation $B_{j_{1},j_{2}}^{1,1,n}$ for $1\leq n\leq4$.
\begin{lemma}\label{L8}
The operators $B_{j_{1},j_{2}}^{1,1,n}$ for $1\leq n\leq4$ have the following properties:\\
(a) Fix $h\in L^{2}(\mathbb{R},\mathbb{R})$, then the mapping $f\rightarrow B_{j,j}^{1,1,n}(h,f)$ defines a continuous linear map from $H^{1}(\mathbb{R},\mathbb{R})$ into $L^{2}(\mathbb{R},\mathbb{R})$ and $f\rightarrow B_{j,-j}^{1,1,n}(h,f)$ defines a continuous linear map from $L^{2}(\mathbb{R},\mathbb{R})$ into $L^{2}(\mathbb{R},\mathbb{R})$. In particular, for all $f\in H^{1}(\mathbb{R},\mathbb{R})$ we have
\begin{equation}
\begin{split}\label{15}
&B_{j,j}^{1,1,1}(h,f)=-j\partial_{x}(G_{j,j}h \ qf)+Q_{j,j}^{1}(h,f),\\
&B_{j,-j}^{1,1,1}(h,f)=-G_{j,-j}h \ qf+Q_{j,-j}^{1}(h,f),
\end{split}
\end{equation}
\begin{equation}
\begin{split}\label{151}
&B_{j,j}^{1,1,2}(h,f)=-\partial_{x}(G_{j,j}qh \ f)+Q_{j,j}^{2}(h,f),\\
&B_{j,-j}^{1,1,2}(h,f)=jG_{j,-j}qh \ f+Q_{j,-j}^{2}(h,f),
\end{split}
\end{equation}
\begin{equation}
\begin{split}\label{152}
&B_{j,j}^{1,1,3}(h,f)=-\partial_{x}(G_{j,j}qh \ qf)+Q_{j,j}^{3}(h,f),\\
&B_{j,-j}^{1,1,3}(h,f)=-jG_{j,-j}qh \ qf+Q_{j,-j}^{3}(h,f),
\end{split}
\end{equation}
and
\begin{equation}
\begin{split}\label{153}
&B_{j,j}^{1,1,4}(h,f)=j\partial_{x}(G_{j,j}h \ f)+Q_{j,j}^{4}(h,f),\\
&B_{j,-j}^{1,1,4}(h,f)=-G_{j,-j}h \ f+Q_{j,-j}^{4}(h,f),
\end{split}
\end{equation}
where
\begin{equation*}
\begin{split}
&\widehat{G_{j,j}h}(k)=\frac{\chi(k)}{-2i(jk+\omega(k))}\widehat{h}(k),\\
&\widehat{G_{j,-j}h}(k)=\frac{1}{2}\chi(k)\widehat{h}(k),\\
&\|Q_{j\pm j}^{n}(h,f)\|_{H^{1}}=\mathcal{O}(\|h\|_{L^{2}},\|f\|_{L^{2}}), \ n=1,2,3,4.
\end{split}
\end{equation*}
(b) For all $f\in H^{1}(\mathbb{R},\mathbb{R})$ we have
\begin{equation}
\begin{split}\label{16}
-j_{1}\Omega B_{j_{1},j_{2}}^{1,1,n}(\phi_{c},f)-B_{j_{1},j_{2}}^{1,1,n}(\Omega\phi_{c},f)
+j_{2} B_{j_{1},j_{2}}^{1,1,n}(\phi_{c},\Omega f)
=-\frac{P^{1}\partial_{x}}{2\vartheta}\alpha_{j_{1},j_{2}}^{n}(\phi_{c}f),
\end{split}
\end{equation}
where $n=1,2,3,4,5$, $\Omega$ is defined by $\widehat{\Omega u}(k)=i\omega(k)\widehat{u}(k)$ and the operator $\omega$ is defined by the symbols \eqref{equation3}.\\
(c) For all $f,g,h\in H^{1}(\mathbb{R},\mathbb{R})$, we have
\begin{equation}
\begin{split}\label{17}
\int_{\mathbb{R}}fB_{j_{1},j_{2}}^{1,1,\{1,4\}}(h,g)=-\frac{j_{1}}{j_{2}}\int_{\mathbb{R}}B_{j_{2},j_{1}}^{1,1,\{1,4\}}(h,f)gdx
+\int_{\mathbb{R}}S_{j_{2},j_{1}}^{\{1,4\}}(\partial_{x}h,f)gdx,
\end{split}
\end{equation}

\begin{equation}
\begin{split}\label{18}
\int_{\mathbb{R}}fB_{j_{1},j_{2}}^{1,1,\{2,3\}}(h,g)=-\int_{\mathbb{R}}B_{j_{2},j_{1}}^{1,1,\{2,3\}}(h,f)gdx
+\int_{\mathbb{R}}S_{j_{2},j_{1}}^{\{2,3\}}(\partial_{x}h,f)gdx,
\end{split}
\end{equation}
where
\begin{equation*}
\begin{split}
\widehat{S}_{j_{2},j_{1}}^{n}(\partial_{x}h,f)g(k)=\int_{\mathbb{R}}\widehat{s}_{j_{2},j_{1}}^{n}(k,k-m,m)\widehat{\partial_{x}h}(k-m)\widehat{f}(m)dm,
\end{split}
\end{equation*}
with
\begin{equation*}
\begin{split}
&\widehat{s}_{j_{2},j_{1}}^{1}(k,k-m,m)=\frac{\mp j_{1}(kq(m)-mq(k))}{2(k-m)i(-j_{2}\omega(k)-\omega(k-m)+j_{1}\omega(m))},\\
&\widehat{s}_{j_{2},j_{1}}^{2}(k,k-m,m)=\frac{-q(k-m)}{2i(-j_{2}\omega(k)-\omega(k-m)+j_{1}\omega(m))},\\
&\widehat{s}_{j_{2},j_{1}}^{3}(k,k-m,m)=\frac{\mp (kq(m)-mq(k))q(k-m)}{2(k-m)i(-j_{2}\omega(k)-\omega(k-m)+j_{1}\omega(m))},\\
&\widehat{s}_{j_{2},j_{1}}^{4}(k,k-m,m)=\frac{j_{1}}{2i(-j_{2}\omega(k)-\omega(k-m)+j_{1}\omega(m))},
\end{split}
\end{equation*}
where $\mp$ means that we take $-$ sign when $j_{1}$ and $j_{2}$ with the same sign and $+$ sign when $j_{1}$ and $j_{2}$ with the opposite sign.
In particular, we have
\begin{equation}
\begin{split}\label{19}
\widehat{S}_{j,j}^{1}(\partial_{x}h,f)=-jG_{j,j}\partial_{x}h \ qf+\widetilde{Q}_{j,j}^{1}(\partial_{x}h,f),
\end{split}
\end{equation}
\begin{equation}
\begin{split}\label{191}
\widehat{S}_{j,j}^{2}(\partial_{x}h,f)=-G_{j,j}\partial_{x}qh \ f+\widetilde{Q}_{j,j}^{2}(\partial_{x}h,f),
\end{split}
\end{equation}
\begin{equation}
\begin{split}\label{192}
\widehat{S}_{j,j}^{3}(\partial_{x}h,f)=-G_{j,j}\partial_{x}qh \ qf+\widetilde{Q}_{j,j}^{3}(\partial_{x}h,f),
\end{split}
\end{equation}
\begin{equation}
\begin{split}\label{193}
\widehat{S}_{j,j}^{4}(\partial_{x}h,f)=jG_{j,j}\partial_{x}h \ f+\widetilde{Q}_{j,j}^{4}(\partial_{x}h,f),
\end{split}
\end{equation}
with
\begin{equation*}
\begin{split}
\big\|\widetilde{Q}_{j,j}^{n}(\partial_{x}h,f)\big\|_{H^{2}}=\mathcal{O}\big(\|h\|_{L^{2}},\|f\|_{L^{2}}\big).
\end{split}
\end{equation*}
\end{lemma}
\begin{proof}
(a). Because of the support properties of $\widehat{\phi}_{c}$ and $\widehat{P}^{1}$, the kernel $b^{1,1,n}_{j_{1},j_{2}}$ \eqref{b1,1} of the normal form transformation $B^{1,1,n}_{j_{1},j_{2}}$ only has to be analysed for $|k-m\pm k_{0}|<\delta$, $|k|>\delta$ and $|m|>\delta$. We have
\begin{equation*}
\begin{split}
|-j_{1}\omega(k)-\omega(k-m)+j_{2}\omega(m)|\geq C>0,
\end{split}
\end{equation*}
with a constant $C$, which implies $|\widehat{b}_{j_{1},j_{2}}^{1,1,n}(k,k-m,m)|<\infty$ for $1\leq n\leq5$.
Next, we analyze the asymptotic behavior of the $|\widehat{b}_{j_{1},j_{2}}^{1,1,n}(k,k-m,m)|<\infty$ for $|k|\rightarrow\infty$. We have
\begin{equation}
\begin{split}\label{omega}
\omega(k)=kq(k)=k+\mathcal{O}(|k|^{-1}), \ \text{for} \  |k|\rightarrow\infty,
\end{split}
\end{equation}
and
\begin{equation}
\begin{split}\label{omega'}
\omega'(k)=1+\mathcal{O}(|k|^{-2}), \ \text{for} \  |k|\rightarrow\infty.
\end{split}
\end{equation}
By the mean value theorem we get
\begin{equation*}
\begin{split}
\widehat{b}_{j,j}^{1,1,1}(k,k-m,m)=&\frac{jkq(m)\chi(k-m)}{2(j(\omega(k)-\omega(m))+\omega(k-m))}\\
=&\frac{jkq(m)\chi(k-m)}{2(j(k-m)\omega'(k-\theta(k-m))+\omega(k-m))},
\end{split}
\end{equation*}
for some $\theta\in[0,1]$. Using again the fact that \text{supp}$\chi$ is compact, we conclude with the help of the expressions \eqref{omega'} and $\widehat{\alpha}^{1}_{j_{1},j_{2}}(k,k-m,m)$ that
\begin{equation*}
\begin{split}
\widehat{b}_{j,j}^{1,1,1}(k,k-m,m)
=&\frac{jkq(m)\chi(k-m)}{2\big(j(k-m)(1+\mathcal{O}(|k|^{-2}))+\omega(k-m)\big)}\\
=&\frac{jkq(m)\chi(k-m)}{2\big(j(k-m)+\omega(k-m)\big)}+\mathcal{O}(|k|^{-1}),  \ \ \text{for} \  |k|\rightarrow\infty.
\end{split}
\end{equation*}
Exploiting once more the compactness of \text{supp}$\chi$ as well as \eqref{omega}, \eqref{omega'} and $\widehat{\alpha}^{1}_{j_{1},j_{2}}(k,k-m,m)$ yields
\begin{equation*}
\begin{split}
\widehat{b}_{j,-j}^{1,1,1}(k,k-m,m)=&\frac{-jkq(m)\chi(k-m)}{2\big(j(\omega(k)+\omega(m))+\omega(k-m)\big)}\\
=&\frac{-jkq(m)\chi(k-m)}{2j\omega(k)\big(1+\mathcal{O}(|k|^{-1})\big)}\\
=&\frac{-q(m)}{2}+\mathcal{O}(|k|^{-1}),
\end{split}
\end{equation*}
for $|k|\rightarrow\infty$. Similarly, when $|k|\to\infty$ we have
\begin{equation*}
\begin{split}
&\widehat{b}_{j,j}^{1,1,2}(k,k-m,m)
=\frac{kq(k-m)\chi(k-m)}{2\big(j(k-m)+\omega(k-m)\big)}+\mathcal{O}(|k|^{-1}),\\
&\widehat{b}_{j,-j}^{1,1,2}(k,k-m,m)
=j\frac{q(k-m)}{2}+\mathcal{O}(|k|^{-1}),\\
&\widehat{b}_{j,j}^{1,1,3}(k,k-m,m)
=\frac{kq(k-m)q(m)\chi(k-m)}{2\big(j(k-m)+\omega(k-m)\big)}+\mathcal{O}(|k|^{-1}),\\
&\widehat{b}_{j,-j}^{1,1,3}(k,k-m,m)
=\frac{-jq(k-m)q(m)}{2}+\mathcal{O}(|k|^{-1}),\\
&\widehat{b}_{j,j}^{1,1,4}(k,k-m,m)
=\frac{-jk\chi(k-m)}{2\big(j(k-m)+\omega(k-m)\big)}+\mathcal{O}(|k|^{-1}),\\
&\widehat{b}_{j,-j}^{1,1,4}(k,k-m,m)
=-\frac{1}{2}+\mathcal{O}(|k|^{-1}).
\end{split}
\end{equation*}
These asymptotic expansions of the $\widehat{b}_{j_{1},j_{2}}^{1,1,n}(k,k-m,m)$ imply \eqref{15}-\eqref{153}.

Finally, since
\begin{equation*}
\begin{split}
\widehat{b}_{j_{1},j_{2}}^{1,1,n}(-k,-(k-m),-m)=\widehat{b}_{j_{1},j_{2}}^{1,1,n}(k,k-m,m)\in\mathbb{R},
\end{split}
\end{equation*}
and $\phi_{c}$ is real-valued, all assertions of (a) follow.

(b) is a direct consequence of the construction of the operators $B_{j_{1},j_{2}}^{1,1,n}$.

In order to prove (c), we compute for all $f, g, h\in H^{1}(\mathbb{R},\mathbb{R})$ that
\begin{equation*}
\begin{split}
(f,B_{j_{1},j_{2}}^{1,1,1}(h,g)) =&\int_{\mathbb{R}}\overline{\widehat{f}(k)}\widehat{B}_{j_{1},j_{2}}^{1,1,1}(h,g)(k)dk\\
=&\int_{\mathbb{R}}\int_{\mathbb{R}}\overline{\widehat{f}(k)}\frac{\mp\frac{j_{1}}{2}k\widehat{q}(m)}
{-j_{1}\omega(k)-\omega(k-m)+j_{2}\omega(m)}\widehat{h}(k-m)\widehat{g}(m)dmdk\\
=&\int_{\mathbb{R}}\int_{\mathbb{R}}\overline{\widehat{g}(-m)}\frac{\mp\frac{j_{1}}{2}k\widehat{q}(m)}
{-j_{1}\omega(k)-\omega(k-m)+j_{2}\omega(m)}\widehat{h}(k-m)\widehat{f}(-k)dkdm\\
=&\int_{\mathbb{R}}\int_{\mathbb{R}}\overline{\widehat{g}(k)}\frac{\pm\frac{j_{1}}{2}m\widehat{q}(k)}
{-j_{2}\omega(k)-\omega(k-m)+j_{1}\omega(m)}\widehat{h}(k-m)\widehat{f}(m)dmdk\\
=&\int_{\mathbb{R}}\int_{\mathbb{R}}\overline{\widehat{g}(k)}\frac{\pm\frac{j_{1}}{2}k\widehat{q}(m)}
{-j_{2}\omega(k)-\omega(k-m)+j_{1}\omega(m)}\widehat{h}(k-m)\widehat{f}(m)dmdk\\
&+\int_{\mathbb{R}}\int_{\mathbb{R}}\overline{\widehat{g}(k)}\frac{\mp\frac{j_{1}}{2}(k\widehat{q}(m)
-m\widehat{q}(k))}
{-j_{2}\omega(k)-\omega(k-m)+j_{1}\omega(m)}\widehat{h}(k-m)\widehat{f}(m)dmdk\\
=&-\frac{j_{1}}{j_{2}}\int_{\mathbb{R}}\overline{\widehat{g}(k)}\widehat{B}_{j_{2},j_{1}}^{1,1,1}(h,f)(k)dk
+\int_{\mathbb{R}}\overline{\widehat{g}(k)}\widehat{S}_{j_{2},j_{1}}^{1}(\partial_{x}h,f)(k)dk\\
=&-\frac{j_{1}}{j_{2}}\big(g,B_{j_{2},j_{1}}^{1,1,1}(h,f)\big)+\big(g,S_{j_{2},j_{1}}^{1}(h,f)\big).
\end{split}
\end{equation*}
Similarly, we have
\begin{equation*}
\begin{split}
\big(f,B_{j_{1},j_{2}}^{1,1,2}(h,g)\big)
=-\big(g,B_{j_{2},j_{1}}^{1,1,2}(h,f)\big)+\big(g,S_{j_{2},j_{1}}^{2}(h,f)\big),
\end{split}
\end{equation*}
\begin{equation*}
\begin{split}
\big(f,B_{j_{1},j_{2}}^{1,1,3}(h,g)\big)
=-\big(g,B_{j_{2},j_{1}}^{1,1,3}(h,f)\big)+\big(g,S_{j_{2},j_{1}}^{3}(h,f)\big),
\end{split}
\end{equation*}
and
\begin{equation*}
\begin{split}
\big(f,B_{j_{1},j_{2}}^{1,1,4}(h,g)\big)
=-\frac{j_{1}}{j_{2}}\big(g,B_{j_{2},j_{1}}^{1,1,4}(h,f)\big)+\big(g,S_{j_{2},j_{1}}^{4}(h,f)\big),
\end{split}
\end{equation*}
yielding \eqref{17} and \eqref{18}. Thanks to \eqref{15}-\eqref{153}, we obtain \eqref{19}-\eqref{193}.
\end{proof}

The transformation \eqref{equation92}, assertions of Lemma \ref{L10}, Lemma \ref{L8} (a), (c) and the Cauchy-Schwarz inequality imply
\begin{corollary}\label{C1}
$\sqrt{\mathcal{E}_{s}}$ is equivalent to $\|R^{0}\|_{H^{s}}+\|R^{1}\|_{H^{s}}$ for sufficiently small $\epsilon>0$.
\end{corollary}

Since the right-hand side of the error equation \eqref{R1} for $j_{1}\in\{\pm1\}$ loses one derivative, we will need the following identities to control the time evolution of $\mathcal{E}_{s}$. See also \cite{D1}.
\begin{lemma}\label{L9}
Let $j\in\{\pm1\}$, $a_{j}\in H^{2}(\mathbb{R},\mathbb{R})$, and $f_{j}\in H^{1}(\mathbb{R},\mathbb{R})$. Then we have
\begin{equation}
\begin{split}\label{part1}
&\int_{\mathbb{R}}a_{j}f_{j}\partial_{x}f_{j}dx=-\frac{1}{2}\int_{\mathbb{R}}\partial_{x}a_{j}f_{j}^{2}dx,
\end{split}
\end{equation}
\begin{equation}
\begin{split}\label{part2}
\sum _{j\in\{\pm1\}}&\int_{\mathbb{R}}a_{j}f_{j}\partial_{x}f_{-j}dx =\frac{1}{2}\int_{\mathbb{R}}(a_{-1}-a_{1})(f_{1}+f_{-1})\partial_{x}(f_{1}-f_{-1})dx\\
&+\mathcal{O}\Big(\big(\|a_{1}\|_{H^{2}(\mathbb{R},\mathbb{R})}+\|a_{-1}\|_{H^{2}(\mathbb{R},\mathbb{R})}\big)
\big(\|f_{1}\|_{L^{2}(\mathbb{R},\mathbb{R})}+\|f_{-1}\|_{L^{2}(\mathbb{R},\mathbb{R})})\Big),
\end{split}
\end{equation}
\begin{equation}
\begin{split}\label{part3}
\sum _{j\in\{\pm1\}}j&\int_{\mathbb{R}}a_{j}f_{j}\partial_{x}f_{-j}dx =\frac{1}{2}\int_{\mathbb{R}}(a_{1}+a_{-1})(f_{1}+f_{-1})\partial_{x}(f_{-1}-f_{1})dx\\
&+\mathcal{O}\Big(\big(\|a_{1}\|_{H^{2}(\mathbb{R},\mathbb{R})}+\|a_{-1}\|_{H^{2}(\mathbb{R},\mathbb{R})}\big)
\big(\|f_{1}\|_{L^{2}(\mathbb{R},\mathbb{R})}+\|f_{-1}\|_{L^{2}(\mathbb{R},\mathbb{R})}\big)\Big),
\end{split}
\end{equation}
\begin{equation}
\begin{split}\label{part4}
(\widehat{G}_{-1,-1}+\widehat{G}_{1,1})(k)=(\frac{1}{-ik}+ik)q(k),
\end{split}
\end{equation}
and
\begin{equation}
\begin{split}\label{part5}
(\widehat{G}_{-1,-1}-\widehat{G}_{1,1})(k)=(\frac{1}{-ik}+ik).
\end{split}
\end{equation}
\end{lemma}
\begin{proof}
By integration by parts, \eqref{part1} follows directly. Using again integration by parts, Cauchy-Schwarz inequality, and \eqref{part1}, we obtain
\begin{equation*}
\begin{split}
\sum_{j\in\{\pm1\}}&\int_{\mathbb{R}}a_{j}f_{j}\partial_{x}f_{-j}dx\\
=&\frac{1}{2}\sum_{j\in\{\pm1\}}\Big[\int_{\mathbb{R}}a_{j}f_{j}\partial_{x}f_{-j}dx-\int_{\mathbb{R}}a_{j}\partial_{x}f_{j}f_{-j}dx
-\int_{\mathbb{R}}\partial_{x}a_{j}f_{j}f_{-j}dx\Big]\\
=&\frac{1}{2}\Big[\int_{\mathbb{R}}(a_{-1}-a_{1})f_{-1}\partial_{x}f_{1}dx-\int_{\mathbb{R}}(a_{-1}-a_{1})f_{1}\partial_{x}f_{-1}dx\Big]\\
&+\mathcal{O}\big((\|a_{1}\|_{H^{2}}+\|a_{-1}\|_{H^{2}})(\|f_{1}\|_{H^{2}}+\|f_{-1}\|_{H^{2}})\big)\\
=&\frac{1}{2}\int_{\mathbb{R}}(a_{-1}-a_{1})(f_{1}+f_{-1})\partial_{x}(f_{1}-f_{-1})dx\\
&+\mathcal{O}\big((\|a_{1}\|_{H^{2}}+\|a_{-1}\|_{H^{2}})(\|f_{1}\|_{H^{2}}+\|f_{-1}\|_{H^{2}})\big),
\end{split}
\end{equation*}
and
\begin{equation*}
\begin{split}
\sum_{j\in\{\pm1\}}&j\int_{\mathbb{R}}a_{j}f_{j}\partial_{x}f_{-j}dx\\
=&\frac{1}{2}\sum_{j\in\{\pm1\}}\Big[j\int_{\mathbb{R}}a_{j}f_{j}\partial_{x}f_{-j}dx-j\int_{\mathbb{R}}a_{j}\partial_{x}f_{j}f_{-j}dx
-j\int_{\mathbb{R}}\partial_{x}a_{j}f_{j}f_{-j}dx\Big]\\
=&\frac{1}{2}\Big[\int_{\mathbb{R}}(a_{-1}+a_{1})f_{1}\partial_{x}f_{-1}dx-\int_{\mathbb{R}}(a_{-1}+a_{1})f_{-1}\partial_{x}f_{1}dx\Big]\\
&+\mathcal{O}\big((\|a_{1}\|_{H^{2}}+\|a_{-1}\|_{H^{2}})(\|f_{1}\|_{H^{2}}+\|f_{-1}\|_{H^{2}})\big)\\
=&-\frac{1}{2}\int_{\mathbb{R}}(a_{1}+a_{-1})(f_{1}+f_{-1})\partial_{x}(f_{1}-f_{-1})dx\\
&+\mathcal{O}\big((\|a_{1}\|_{H^{2}}+\|a_{-1}\|_{H^{2}})(\|f_{1}\|_{H^{2}}+\|f_{-1}\|_{H^{2}})\big).
\end{split}
\end{equation*}

Recalling $\omega(k)=k\widehat{q}(k)$ and $\widehat{q}(k)=\sqrt{\frac{2+k^{2}}{1+k^{2}}}$, we have
\begin{equation*}
\begin{split}
&\widehat{G}_{1,1}+\widehat{G}_{-1,-1}=\frac{1}{-2i}\Big(\frac{1}{k+\omega(k)}+\frac{1}{-k+\omega(k)}\Big)=\big(\frac{1}{-ik}+ik\big)\widehat{q}(k),\\
&\widehat{G}_{-1,-1}-\widehat{G}_{1,1}=\frac{1}{-2i}\Big(\frac{1}{-k+\omega(k)}-\frac{1}{k+\omega(k)}\Big)=\big(\frac{1}{-ik}+ik\big).
\end{split}
\end{equation*}
The proof is complete.
\end{proof}

Now, we are prepared to analyze $\partial_{t}E_{\ell}$. We compute
\begin{equation*}
\begin{split}
\partial_{t}E_{\ell}=&\sum_{j_{1}\in\{\pm1\}}\Big[\int_{\mathbb{R}}\partial_{x}^{\ell}\mathcal{R}_{j_{1}}^{0}
    \partial_{t}\partial_{x}^{\ell}\mathcal{R}_{j_{1}}^{0}dx
    +\int_{\mathbb{R}}\partial_{x}^{\ell}R_{j_{1}}^{1}
    \partial_{t}\partial_{x}^{\ell}R_{j_{1}}^{1}dx\\
    &+\epsilon\sum_{j_{2}\in\{\pm1\}}\sum_{n=1}^{5}\big(\int_{\mathbb{R}}\partial_{t}\partial_{x}^{\ell}R_{j_{1}}^{1}
    \partial_{x}^{\ell}B_{j_{1},j_{2}}^{1,0,n}(\phi_{c},\mathcal{R}_{j_{2}}^{0})dx
    +\int_{\mathbb{R}}\partial_{x}^{\ell}R_{j_{1}}^{1}
    \partial_{x}^{\ell}B_{j_{1},j_{2}}^{1,0,n}(\partial_{t}\phi_{c},\mathcal{R}_{j_{2}}^{0})dx\\
    &+\int_{\mathbb{R}}\partial_{x}^{\ell}R_{j_{1}}^{1}
    \partial_{x}^{\ell}B_{j_{1},j_{2}}^{1,0,n}(\phi_{c},\partial_{t}\mathcal{R}_{j_{2}}^{0})dx
    +\int_{\mathbb{R}}\partial_{t}\partial_{x}^{\ell}R_{j_{1}}^{1}
    \partial_{x}^{\ell}B_{j_{1},j_{2}}^{1,1,n}(\phi_{c},R_{j_{2}}^{1})dx\\
    &+\int_{\mathbb{R}}\partial_{x}^{\ell}R_{j_{1}}^{1}
    \partial_{x}^{\ell}B_{j_{1},j_{2}}^{1,1,n}(\partial_{t}\phi_{c},R_{j_{2}}^{1})dx
    +\int_{\mathbb{R}}\partial_{x}^{\ell}R_{j_{1}}^{1}
    \partial_{x}^{\ell}B_{j_{1},j_{2}}^{1,1,n}(\phi_{c},\partial_{t}R_{j_{2}}^{1})dx
\big)\Big].
\end{split}
\end{equation*}
Using the error equations \eqref{equation91},  \eqref{equ11-1} and \eqref{R1}, we get
\begin{equation*}
\begin{split}
&\partial_{t}E_{\ell}=\sum_{j_{1}\in\{\pm1\}}\Big[j_{1}\int_{\mathbb{R}}\partial_{x}^{\ell}\mathcal{R}_{j_{1}}^{0}
    \Omega\partial_{x}^{\ell}\mathcal{R}_{j_{1}}^{0}dx
    +\epsilon^{2}\int_{\mathbb{R}}\partial_{x}^{\ell}\mathcal{R}_{j_{1}}^{0}
    \mathcal{F}_{j_{1}}^{5}dx
    +j_{1}\int_{\mathbb{R}}\partial_{x}^{\ell}R_{j_{1}}^{1}
    \Omega\partial_{x}^{\ell}R_{j_{1}}^{1}dx\\
   & \ \ \ \ \ \ \ \ \ \ \ \ \ \ \ \ \ \ \ \ \ \ \ \ +\epsilon^{2}\int_{\mathbb{R}}\partial_{x}^{\ell}R_{j_{1}}^{1}\mathcal{F}_{j_{1}}^{6}dx
   +\int_{\mathbb{R}}\partial_{x}^{\ell}R_{j_{1}}^{1}(\epsilon^{-5/2}Res_{U^{1}_{j_{1}}}(\epsilon\Psi))\Big]\\
+\epsilon&\sum_{j_{1},j_{2}\in\{\pm1\}}\sum_{n=1}^{5}\Big[\frac{1}{2}\int_{\mathbb{R}} \partial_{x}^{\ell}R_{j_{1}}^{1}
 \partial_{x}^{\ell+1}P^{1}\vartheta^{-1}\alpha_{j_{1},j_{2}}^{n}(\phi_{c}\vartheta_{0}\mathcal{R}_{j_{2}}^{0})dx
 +j_{1}\int_{\mathbb{R}}\Omega\partial_{x}^{\ell}R_{j_{1}}^{1}
 \partial_{x}^{\ell}B_{j_{1},j_{2}}^{1,0,n}(\phi_{c},\mathcal{R}_{j_{2}}^{0})dx\\
 & \ \ \ \ \ \ \ \ \ \ \ \ \ \ \ \ \ \ \ -\int_{\mathbb{R}} \partial_{x}^{\ell}R_{j_{1}}^{1}
 \partial_{x}^{\ell}B_{j_{1},j_{2}}^{1,0,n}(\Omega\phi_{c},\mathcal{R}_{j_{2}}^{0})dx
 +j_{2}\int_{\mathbb{R}} \partial_{x}^{\ell}R_{j_{1}}^{1}
 \partial_{x}^{\ell}B_{j_{1},j_{2}}^{1,0,n}(\phi_{c},\Omega\mathcal{R}_{j_{2}}^{0})dx\Big]\\
+\epsilon&\sum_{j_{1},j_{2}\in\{\pm1\}}\sum_{n=1}^{5}\Big[\frac{1}{2}\int_{\mathbb{R}}\partial_{x}^{\ell}R_{j_{1}}^{1}
 \partial_{x}^{\ell+1}P^{1}\vartheta^{-1}\alpha_{j_{1},j_{2}}^{n}(\phi_{c}R_{j_{2}}^{1})dx
+j_{1}\int_{\mathbb{R}} \Omega\partial_{x}^{\ell}R_{j_{1}}^{1}
 \partial_{x}^{\ell}B_{j_{1},j_{2}}^{1,1,n}(\phi_{c},R_{j_{2}}^{1})dx\\
 & \ \ \ \ \ \ \ \ \ \ \ \ \ \ \ \ \ \ \ \ -\int_{\mathbb{R}} \partial_{x}^{\ell}R_{j_{1}}^{1}
 \partial_{x}^{\ell}B_{j_{1},j_{2}}^{1,1,n}(\Omega\phi_{c},R_{j_{2}}^{1})dx
 +j_{2}\int_{\mathbb{R}} \partial_{x}^{\ell}R_{j_{1}}^{1}
 \partial_{x}^{\ell}B_{j_{1},j_{2}}^{1,1,n}(\phi_{c},\Omega R_{j_{2}}^{1})dx\Big]\\
+\epsilon&\sum_{j_{1},j_{2}\in\{\pm1\}}\sum_{n=1}^{5}\Big[\int_{\mathbb{R}} \partial_{x}^{\ell}R_{j_{1}}^{1}
 \partial_{x}^{\ell}B_{j_{1},j_{2}}^{1,0,n}(\partial_{t}\phi_{c}+\Omega\phi_{c},\mathcal{R}_{j_{2}}^{0})dx
 +\int_{\mathbb{R}} \partial_{x}^{\ell}R_{j_{1}}^{1}
 \partial_{x}^{\ell}B_{j_{1},j_{2}}^{1,1,n}(\partial_{t}\phi_{c}+\Omega\phi_{c},R_{j_{2}}^{1})dx\Big]\\
 +\epsilon^{2}&\sum _{j_{1},j_{2}\in\{\pm1\}}\sum_{n=1}^{5}\Big[
\int_{\mathbb{R}}\partial_{x}^{\ell}((\sum_{m=1}^{7}D_{m})_{j_{1}}+\epsilon\mathcal{F}_{j_{1}}^{6}+\epsilon^{-7/2}Res_{U_{j_{1}}^{1}}) \partial_{x}^{\ell}B_{j_{1},j_{2}}^{1,0,n}(\phi_{c}, \mathcal{R}_{j_{2}}^{0})dx
+\int_{\mathbb{R}}\partial_{x}^{\ell}R_{j_{1}}^{1}\partial_{x}^{\ell}\mathcal{F}_{j_{2}}^{4}\Big]\\
+\epsilon&\sum _{j_{1},j_{2}\in\{\pm1\}}\sum_{n=1}^{5}\Big[
\int_{\mathbb{R}}\partial_{x}^{\ell}(\epsilon^{2}\mathcal{F}_{j_{1}}^{6}+\epsilon^{-5/2}Res_{U_{j_{1}}^{1}}) \partial_{x}^{\ell}B_{j_{1},j_{2}}^{1,1,n}(\phi_{c}, R_{j_{2}}^{1})dx\\
& \ \ \ \ \ \ \ \ \ \ \ \ \ \ \ \ \ \ \ +\int_{\mathbb{R}}\partial_{x}^{\ell}R_{j_{1}}^{1} \partial_{x}^{\ell}B_{j_{1},j_{2}}^{1,1,n}(\phi_{c}, \epsilon^{2}\mathcal{F}_{j_{2}}^{6}+\epsilon^{-5/2}Res_{U_{j_{2}}^{1}})dx\Big]\\
+\epsilon^{2}&\sum _{j_{1},j_{2}\in\{\pm1\}}\sum_{n=1}^{5}\Big[
\int_{\mathbb{R}}\partial_{x}^{\ell}(\sum_{m=1}^{7}D_{m})_{j_{1}} \partial_{x}^{\ell}B_{j_{1},j_{2}}^{1,1,n}(\phi_{c}, R_{j_{2}}^{1})dx
+\int_{\mathbb{R}}\partial_{x}^{\ell}R_{j_{1}}^{1} \partial_{x}^{\ell}B_{j_{1},j_{2}}^{1,1,n}(\phi_{c}, (\sum_{m=1}^{7}D_{m})_{j_{2}})dx\Big]\\
+\epsilon^{3}&\sum _{j_{1}\in\{\pm1\}}
\int_{\mathbb{R}}\partial_{x}^{\ell}R_{j_{1}}^{1}\partial_{x}^{\ell}(\sum_{m=1}^{7}G_{m})_{j_{1}}dx.\\
\end{split}
\end{equation*}

Due to the skew symmetry of $\Omega$, the first and third integrals equal to zero. Since the operators $B_{j_{1},j_{2}}^{1,0,n}$ and $B_{j_{1},j_{2}}^{1,1,n}$ satisfy \eqref{16'} and \eqref{16}, the sixth integral cancels with the sum of the seventh, the eighth and the ninth integral. Similarly, the tenth integral cancels with the sum of the eleventh, the twelfth and the thirteenth integral. Moreover, because of the estimates \eqref{equation12-1} and \eqref{equation12-3} for the residual, the form of the terms $\mathcal{F}^{4}$ and $\mathcal{F}^{5}$ in equations \eqref{equation91} and \eqref{equ11-1}, the Lemma \ref{L5} for the bound $\partial_{t}\widehat{\psi}_{\pm1}+i\omega\widehat{\psi}_{\pm1}$, the regularity properties of the operators $B_{j_{1},j_{2}}^{1,0,n}$ and $B_{j_{1},j_{2}}^{1,1,n}$ from Lemma \ref{L10}(a) and Lemma \ref{L8}(a), identity \eqref{part1} and the Corollary \ref{C1}, the second, the fourth, the fifth, the fourteenth to the nineteenth integrals can be bounded by $C\epsilon^{2}(\mathcal{E}_{s}+\epsilon^{3/2}\mathcal{E}_{s}^{3/2}+\epsilon^{3}\mathcal{E}_{s}^{2}+1)$ for some constant $C$. Hence, we have
\begin{align*}
\partial_{t}E_{\ell}=&\epsilon^{2}\sum _{j_{1}\in\{\pm1\}}\sum_{n=1}^{5}\Big[
\int_{\mathbb{R}}\partial_{x}^{\ell}(\sum_{m=1}^{7}D_{m})_{j_{1}} \partial_{x}^{\ell}B_{j_{1},j_{1}}^{1,1,n}(\phi_{c}, R_{j_{1}}^{1})dx\\
&+\int_{\mathbb{R}}\partial_{x}^{\ell}R_{j_{1}}^{1} \partial_{x}^{\ell}B_{j_{1},j_{1}}^{1,1,n}(\phi_{c}, (\sum_{m=1}^{7}D_{m})_{j_{1}})dx\Big]\\
&+\epsilon^{2}\sum _{j_{1}\in\{\pm1\}}\sum_{n=1}^{5}\Big[\int_{\mathbb{R}}\partial_{x}^{\ell}(\sum_{m=1}^{7}D_{m})_{j_{1}} \partial_{x}^{\ell}B_{j_{1},-j_{1}}^{1,1,n}(\phi_{c}, R_{-j_{1}}^{1})dx\\
&+\int_{\mathbb{R}}\partial_{x}^{\ell}R_{j_{1}}^{1} \partial_{x}^{\ell}B_{j_{1},-j_{1}}^{1,1,n}(\phi_{c}, (\sum_{m=1}^{7}D_{m})_{-j_{1}})dx\Big]\\
&+\epsilon^{3}\sum _{j_{1}\in\{\pm1\}}
\int_{\mathbb{R}}\partial_{x}^{\ell}R_{j_{1}}^{1}\partial_{x}^{\ell}(\sum_{m=1}^{7}G_{m})dx\\
&+\epsilon^{2}(\mathcal{E}_{s}+\epsilon^{3/2}\mathcal{E}_{s}^{3/2}+\epsilon^{3}\mathcal{E}_{s}^{2}+1)\\
=&:\sum_{n=1}^{5}I_{1}^{n}+\sum_{n=1}^{5}I_{2}^{n}+I_{3}
+\epsilon^{2}(\mathcal{E}_{s}+\epsilon^{3/2}\mathcal{E}_{s}^{3/2}+\epsilon^{3}\mathcal{E}_{s}^{2}+1).
\end{align*}
Firstly, we analyse $I_{1}^{n}$. To extract all terms with more than $\ell$ spatial derivatives falling on $R_{1}^{1}$ or $R_{-1}^{1}$ we apply Leibniz's rule.
For $I_{1}^{1}$, recalling the equation \eqref{R1}, \  \eqref{17} and the asymptotic expansion \eqref{15} and \eqref{19}, we have
\begin{align*}
I_{1}^{1}=&\epsilon^{2}\sum _{j_{1}\in\{\pm1\}}\Big[
\int\partial_{x}^{\ell}(\sum_{m=1}^{7}D_{m})_{j_{1}} B_{j_{1},j_{1}}^{1,1,1}(\phi_{c}, \partial_{x}^{\ell}R_{j_{1}}^{1})dx\\
&+\ell\int\partial_{x}^{\ell}(\sum_{m=1}^{7}D_{m})_{j_{1}} B_{j_{1},j_{1}}^{1,1,1}(\partial_{x}\phi_{c}, \partial_{x}^{\ell-1}R_{j_{1}}^{1})dx
+\int\partial_{x}^{\ell}R_{j_{1}}^{1} B_{j_{1},j_{1}}^{1,1,1}(\phi_{c}, \partial_{x}^{\ell}(\sum_{m=1}^{7}D_{m})_{j_{1}})dx\\
&+\ell\int\partial_{x}^{\ell}R_{j_{1}}^{1} B_{j_{1},j_{1}}^{1,1,1}(\partial_{x}\phi_{c}, \partial_{x}^{\ell-1}(\sum_{m=1}^{7}D_{m})_{j_{1}})dx\Big]
+\epsilon^{2}(\mathcal{E}_{s}+\epsilon^{3/2}\mathcal{E}_{s}^{3/2}+\epsilon^{3}\mathcal{E}_{s}^{2}+1)\\
=&\epsilon^{2}\sum _{j_{1}\in\{\pm1\}}\Big[
\int\partial_{x}^{\ell}(\sum_{m=1}^{7}D_{m})_{j_{1}} S_{j_{1},j_{1}}^{1}(\partial_{x}\phi_{c}, \partial_{x}^{\ell}R_{j_{1}}^{1})dx\\
& \ \ \ \ \ \ \ \ \ +2\ell\int\partial_{x}^{\ell}(\sum_{m=1}^{7}D_{m})_{j_{1}} B_{j_{1},j_{1}}^{1,1,1}(\partial_{x}\phi_{c}, \partial_{x}^{\ell-1}R_{j_{1}}^{1})dx\Big]+\epsilon^{2}(\mathcal{E}_{s}+\epsilon^{3/2}\mathcal{E}_{s}^{3/2}+\epsilon^{3}\mathcal{E}_{s}^{2}+1)\\
=&-\epsilon^{2}(2\ell+1)\sum _{j_{1}\in\{\pm1\}}j_{1}\int\partial_{x}^{\ell}(\sum_{m=1}^{7}D_{m})_{j_{1}}(G_{j_{1},j_{1}}\partial_{x}\phi_{c}) (q\partial_{x}^{\ell}R_{j_{1}}^{1})dx\\
& \ \ \ \ \ \ \ \ +\epsilon^{2}(\mathcal{E}_{s}+\epsilon^{3/2}\mathcal{E}_{s}^{3/2}+\epsilon^{3}\mathcal{E}_{s}^{2}+1)\\
=:&\sum_{i=1}^{7}I_{1i}^{1}+\epsilon^{2}(\mathcal{E}_{s}+\epsilon^{3/2}\mathcal{E}_{s}^{3/2}+\epsilon^{3}\mathcal{E}_{s}^{2}+1).
\end{align*}
For $I_{11}^{1}$, because of \eqref{part1} and \eqref{part3}, we have
\begin{equation*}
\begin{split}
I_{11}^{1}=&-\epsilon^{2}(\ell+1/2)\sum _{j_{1}\in\{\pm1\}}\int\partial_{x}^{\ell+1}\big(\phi_{1}q(R_{j_{1}}^{1}-R_{-j_{1}}^{1})\big)G_{j_{1},j_{1}}\partial_{x}\phi_{c} q\partial_{x}^{\ell}R_{j_{1}}^{1}dx\\
=&\epsilon^{2}(\ell+1/2)\sum _{j_{1}\in\{\pm1\}}\int (G_{j_{1},j_{1}}\partial_{x}\phi_{c})\phi_{1} (q\partial_{x}^{\ell}R_{j_{1}}^{1})\partial_{x}^{\ell+1}(qR_{-j_{1}}^{1}) dx
+\epsilon^{2}\mathcal{O}(\mathcal{E}_{s}+\epsilon^{3/2}\mathcal{E}_{s}^{3/2})\\
=&\frac{\epsilon^{2}(\ell+1/2)}{2}\int (G_{-1,-1}-G_{1,1})\partial_{x}\phi_{c}\phi_{1} \partial_{x}^{\ell}q(R_{1}^{1}+R_{-1}^{1})\partial_{x}^{\ell+1}q(R_{1}^{1}-R_{-1}^{1}) dx\\
&+\epsilon^{2}\mathcal{O}(\mathcal{E}_{s}+\epsilon^{3/2}\mathcal{E}_{s}^{3/2}).
\end{split}
\end{equation*}
Recalling \eqref{equation3} and \eqref{part4}, we have
\begin{equation*}
(\widehat{G}_{1,1}-\widehat{G}_{-1,-1})(k)=(\frac{1}{-ik}+ik)\chi(k),
\end{equation*}
and
\begin{equation*}
\widehat{q}(k)-1=\mathcal{O}(k^{-2}).
\end{equation*}
Then we have
\begin{equation*}
\begin{split}
I_{11}^{1}=&\frac{\epsilon^{2}(\ell+1/2)}{2}\int (-\phi_{c}+\partial_{x}^{2}\phi_{c})\phi_{1} \partial_{x}^{\ell}(R_{1}^{1}+R_{-1}^{1})\partial_{x}^{\ell+1}(R_{1}^{1}-R_{-1}^{1}) dx\\
&+\epsilon^{2}\mathcal{O}(\mathcal{E}_{s}+\epsilon^{3/2}\mathcal{E}_{s}^{3/2}).
\end{split}
\end{equation*}
By the equation \eqref{equ1,-1}, we have
\begin{equation}
\begin{split}\label{equ}
\partial_{t}(R_{1}^{1}+R_{-1}^{1})=&\partial_{x}q(R_{1}^{1}-R_{-1}^{1})+\frac{\epsilon}{2} \partial_{x}\big(\phi_{1}q(R_{1}^{1}-R_{-1}^{1})\big)
+\frac{\epsilon}{2} \partial_{x}\big(q\phi_{2}(R_{1}^{1}+R_{-1}^{1})\big)\\
&+\epsilon^{-5/2}\big(Res_{U_{1}^{1}}(\epsilon\Psi)+Res_{U_{-1}^{1}}(\epsilon\Psi)\big).
\end{split}
\end{equation}
Taking $\partial_{x}^{\ell}$ on the equation \eqref{equ}, we have
\begin{equation}
\begin{split}\label{equu}
\partial_{x}^{\ell+1}&(R_{1}^{1}-R_{-1}^{1})\\
=&\frac{1}{1+\frac{\epsilon\phi_{1}}{2}}\Big[
\partial_{x}^{\ell}\partial_{t}(R_{1}^{1}+R_{-1}^{1})-\frac{\epsilon q\phi_{2}}{2}\partial_{x}^{\ell+1}(R_{1}^{1}+R_{-1}^{1})
+(1+\frac{\epsilon\phi_{1}}{2})(1-q)\partial_{x}^{\ell+1}(R_{1}^{1}-R_{-1}^{1})\\
&-\frac{\epsilon}{2}\sum_{i=1}^{\ell+1}C_{\ell+1}^{i}\partial_{x}^{i}\phi_{1}\partial_{x}^{\ell-i+1}q(R_{1}^{1}-R_{-1}^{1})
-\frac{\epsilon}{2}\sum_{i=1}^{\ell+1}C_{\ell+1}^{i}\partial_{x}^{i}q\phi_{2}\partial_{x}^{\ell-i+1}q(R_{1}^{1}+R_{-1}^{1})\\
&-\epsilon^{-5/2}\partial_{x}^{\ell}(Res_{U_{1}^{1}}(\epsilon\Psi)+Res_{U_{-1}^{1}}(\epsilon\Psi))\Big].
\end{split}
\end{equation}
Then by \eqref{equu} and integration by parts, we have
\begin{align*}
I_{11}^{1}=&\frac{\epsilon^{2}(\ell+1/2)}{2}\int \frac{(-\phi_{c}+\partial_{x}^{2}\phi_{c})\phi_{1}}{1+\frac{\epsilon\phi_{1}}{2}} \partial_{x}^{\ell}(R_{1}^{1}+R_{-1}^{1})\partial_{x}^{\ell}\partial_{t}(R_{1}^{1}+R_{-1}^{1}) dx\\
&-\frac{\epsilon^{3}(\ell+1/2)}{4}\int \frac{(-\phi_{c}+\partial_{x}^{2}\phi_{c})\phi_{1}q\phi_{2}}{1+\frac{\epsilon\phi_{1}}{2}} \partial_{x}^{\ell}(R_{1}^{1}+R_{-1}^{1})\partial_{x}^{\ell+1}(R_{1}^{1}+R_{-1}^{1}) dx\\
&+\epsilon^{2}\mathcal{O}(\mathcal{E}_{s}+\epsilon^{3/2}\mathcal{E}_{s}^{3/2})\\
=&\frac{\epsilon^{2}(\ell+1/2)}{2}\frac{d}{dt}\int \frac{(-\phi_{c}+\partial_{x}^{2}\phi_{c})\phi_{1}}{2+\epsilon\phi_{1}}
 \big(\partial_{x}^{\ell}(R_{1}^{1}+R_{-1}^{1})\big)^{2}dx +\epsilon^{2}\mathcal{O}(\mathcal{E}_{s}+\epsilon^{3/2}\mathcal{E}_{s}^{3/2}).
\end{align*}

Similarly, we have
\begin{equation*}
\begin{split}
I_{12}^{1}=&-\epsilon^{2}(\ell+1/2)\sum _{j_{1}\in\{\pm1\}}j_{1}\int\partial_{x}^{\ell+1}\big(q\phi_{2}(R_{j_{1}}^{1}+R_{-j_{1}}^{1})\big)G_{j_{1},j_{1}}\partial_{x}\phi_{c} (q\partial_{x}^{\ell}R_{j_{1}}^{1})dx\\
=&-\epsilon^{2}(\ell+1/2)\sum _{j_{1}\in\{\pm1\}}j_{1}\int G_{j_{1},j_{1}}\partial_{x}\phi_{c}q\phi_{2} \partial_{x}^{\ell}R_{j_{1}}^{1}\partial_{x}^{\ell+1}R_{-j_{1}}^{1} dx
+\epsilon^{2}\mathcal{O}(\mathcal{E}_{s}+\epsilon^{3/2}\mathcal{E}_{s}^{3/2})\\
=&\frac{\epsilon^{2}(\ell+1/2)}{2}\int (G_{1,1}+G_{-1,-1})\partial_{x}\phi_{c}q\phi_{2} \partial_{x}^{\ell}(R_{1}^{1}+R_{-1}^{1})\partial_{x}^{\ell+1}(R_{1}^{1}-R_{-1}^{1}) dx\\
&+\epsilon^{2}\mathcal{O}(\mathcal{E}_{s}+\epsilon^{3/2}\mathcal{E}_{s}^{3/2})\\
=&\frac{\epsilon^{2}(\ell+1/2)}{2}\frac{d}{dt}\int \frac{q(-\phi_{c}+\partial_{x}^{2}\phi_{c})q\phi_{2}}{2+\epsilon\phi_{1}}
 \big(\partial_{x}^{\ell}(R_{1}^{1}+R_{-1}^{1})\big)^{2}dx +\epsilon^{2}\mathcal{O}(\mathcal{E}_{s}+\epsilon^{3/2}\mathcal{E}_{s}^{3/2}),
\end{split}
\end{equation*}
\begin{equation*}
\begin{split}
I_{13}^{1}=-\frac{\epsilon^{2}(\ell+1/2)}{2}\frac{d}{dt}\int \frac{q(-\phi_{c}+\partial_{x}^{2}\phi_{c})q\phi_{2}}{2+\epsilon\phi_{1}}
 \big(\partial_{x}^{\ell}(R_{1}^{1}+R_{-1}^{1})\big)^{2}dx
 +\epsilon^{2}\mathcal{O}(\mathcal{E}_{s}+\epsilon^{3/2}\mathcal{E}_{s}^{3/2}),
\end{split}
\end{equation*}
\begin{equation*}
\begin{split}
I_{14}^{1}
=\frac{\epsilon^{2}(\ell+1/2)}{2}\frac{d}{dt}\int \frac{(-\phi_{c}+\partial_{x}^{2}\phi_{c})\phi_{1}}{2+\epsilon\phi_{1}}
 \big(\partial_{x}^{\ell}(R_{1}^{1}+R_{-1}^{1})\big)^{2}dx
 +\epsilon^{2}\mathcal{O}(\mathcal{E}_{s}+\epsilon^{3/2}\mathcal{E}_{s}^{3/2}),
\end{split}
\end{equation*}
and
\begin{equation*}
\begin{split}
I_{16}^{1}
=-\frac{\epsilon^{3}(\ell+1/2)}{2}\frac{d}{dt}\int \frac{\mathcal{G}(-\phi_{c}+\partial_{x}^{2}\phi_{c})}{2+\epsilon\phi_{1}}
 \big(\partial_{x}^{\ell}(R_{1}^{1}+R_{-1}^{1})\big)^{2}dx
 +\epsilon^{2}\mathcal{O}(\mathcal{E}_{s}+\epsilon^{3/2}\mathcal{E}_{s}^{3/2}).
\end{split}
\end{equation*}
Because of the well properties for $I_{15}^{1}$ and $I_{17}^{1}$, we obtain $I_{15}^{1}+I_{17}^{1}
=\epsilon^{2}\mathcal{O}(\mathcal{E}_{s}+\epsilon^{3/2}\mathcal{E}_{s}^{3/2})$.

Then we have
\begin{equation*}
\begin{split}
I_{1}^{1}=&\frac{\epsilon^{2}(\ell+1/2)}{2}\frac{d}{dt}\int \frac{(-\phi_{c}+\partial_{x}^{2}\phi_{c})(2\phi_{1}-\epsilon\mathcal{G})}{2+\epsilon\phi_{1}}
 \big(\partial_{x}^{\ell}(R_{1}^{1}+R_{-1}^{1})\big)^{2}dx\\
&+\epsilon^{2}\mathcal{O}(\mathcal{E}_{s}+\epsilon^{3/2}\mathcal{E}_{s}^{3/2}).
\end{split}
\end{equation*}
Similarly, by the equations \eqref{151}, \eqref{18}, \eqref{191}, \eqref{part1}, \eqref{part2}, \eqref{part5} and \eqref{equu}, we have
\begin{equation*}
\begin{split}
I_{1}^{2}=&\epsilon^{2}\sum _{j_{1}\in\{\pm1\}}\Big[
\int\partial_{x}^{\ell}(\sum_{m=1}^{7}D_{m})_{j_{1}} B_{j_{1},j_{1}}^{1,1,2}(\phi_{c}, \partial_{x}^{\ell}R_{j_{1}}^{1})dx
+\ell\int\partial_{x}^{\ell}(\sum_{m=1}^{7}D_{m})_{j_{1}} B_{j_{1},j_{1}}^{1,1,2}(\partial_{x}\phi_{c}, \partial_{x}^{\ell-1}R_{j_{1}}^{1})dx\\
&+\int\partial_{x}^{\ell}R_{j_{1}}^{1} B_{j_{1},j_{1}}^{1,1,2}\big(\phi_{c}, \partial_{x}^{\ell}(\sum_{m=1}^{7}D_{m})_{j_{1}}\big)dx
+\ell\int\partial_{x}^{\ell}R_{j_{1}}^{1} B_{j_{1},j_{1}}^{1,1,2}\big(\partial_{x}\phi_{c}, \partial_{x}^{\ell-1}(\sum_{m=1}^{7}D_{m})_{j_{1}}\big)dx\Big]\\
&+\epsilon^{2}\mathcal{O}(\mathcal{E}_{s}+\epsilon^{3/2}\mathcal{E}_{s}^{3/2})\\
=&\epsilon^{2}\sum _{j_{1}\in\{\pm1\}}\Big[
\int\partial_{x}^{\ell}(\sum_{m=1}^{7}D_{m})_{j_{1}} S_{j_{1},j_{1}}^{2}(\partial_{x}\phi_{c}, \partial_{x}^{\ell}R_{j_{1}}^{1})dx
+2\ell\int\partial_{x}^{\ell}(\sum_{m=1}^{7}D_{m})_{j_{1}} B_{j_{1},j_{1}}^{1,1,2}(\partial_{x}\phi_{c}, \partial_{x}^{\ell-1}R_{j_{1}}^{1})dx\Big]\\
&+\epsilon^{2}\mathcal{O}(\mathcal{E}_{s}+\epsilon^{3/2}\mathcal{E}_{s}^{3/2})\\
=&-\epsilon^{2}(2\ell+1)\sum _{j_{1}\in\{\pm1\}}\int\partial_{x}^{\ell}(\sum_{m=1}^{7}D_{m})_{j_{1}}G_{j_{1},j_{1}}\partial_{x}q\phi_{c} \partial_{x}^{\ell}R_{j_{1}}^{1}dx
+\epsilon^{2}\mathcal{O}(\mathcal{E}_{s}+\epsilon^{3/2}\mathcal{E}_{s}^{3/2})\\
=&\frac{\epsilon^{2}(\ell+1/2)}{2}\frac{d}{dt}\int \frac{-q^{2}(-\phi_{c}+\partial_{x}^{2}\phi_{c})(2\phi_{1}-\epsilon\mathcal{G})}{2+\epsilon\phi_{1}}
 (\partial_{x}^{\ell}(R_{1}^{1}+R_{-1}^{1}))^{2}dx\\
&+\epsilon^{2}\mathcal{O}(\mathcal{E}_{s}+\epsilon^{3/2}\mathcal{E}_{s}^{3/2}).
\end{split}
\end{equation*}
By the equations \eqref{152}, \eqref{18}, \eqref{192}, \eqref{part1}, \eqref{part2}, \eqref{part3}, \eqref{part5} and \eqref{equu}, we have
\begin{equation*}
\begin{split}
I_{1}^{3}=&\epsilon^{2}\sum _{j_{1}\in\{\pm1\}}\Big[
\int\partial_{x}^{\ell}(\sum_{m=1}^{7}D_{m})_{j_{1}} B_{j_{1},j_{1}}^{1,1,3}(\phi_{c}, \partial_{x}^{\ell}R_{j_{1}}^{1})dx
+\ell\int\partial_{x}^{\ell}(\sum_{m=1}^{7}D_{m})_{j_{1}} B_{j_{1},j_{1}}^{1,1,3}(\partial_{x}\phi_{c}, \partial_{x}^{\ell-1}R_{j_{1}}^{1})dx\\
&+\int\partial_{x}^{\ell}R_{j_{1}}^{1} B_{j_{1},j_{1}}^{1,1,3}\big(\phi_{c}, \partial_{x}^{\ell}(\sum_{m=1}^{7}D_{m})_{j_{1}}\big)dx
+\ell\int\partial_{x}^{\ell}R_{j_{1}}^{1} B_{j_{1},j_{1}}^{1,1,3}\big(\partial_{x}\phi_{c}, \partial_{x}^{\ell-1}(\sum_{m=1}^{7}D_{m})_{j_{1}}\big)dx\Big]\\
&+\epsilon^{2}\mathcal{O}(\mathcal{E}_{s}+\epsilon^{3/2}\mathcal{E}_{s}^{3/2})\\
=&\epsilon^{2}\sum _{j_{1}\in\{\pm1\}}\Big[
\int\partial_{x}^{\ell}(\sum_{m=1}^{7}D_{m})_{j_{1}} S_{j_{1},j_{1}}^{3}(\partial_{x}\phi_{c}, \partial_{x}^{\ell}R_{j_{1}}^{1})dx
+2\ell\int\partial_{x}^{\ell}(\sum_{m=1}^{7}D_{m})_{j_{1}} B_{j_{1},j_{1}}^{1,1,3}(\partial_{x}\phi_{c}, \partial_{x}^{\ell-1}R_{j_{1}}^{1})dx\Big]\\
&+\epsilon^{2}\mathcal{O}(\mathcal{E}_{s}+\epsilon^{3/2}\mathcal{E}_{s}^{3/2})\\
=&-\epsilon^{2}(2\ell+1)\sum _{j_{1}\in\{\pm1\}}\int\partial_{x}^{\ell}(\sum_{m=1}^{7}D_{m})_{j_{1}}G_{j_{1},j_{1}}\partial_{x}q\phi_{c} q\partial_{x}^{\ell}R_{j_{1}}^{1}dx
+\epsilon^{2}\mathcal{O}(\mathcal{E}_{s}+\epsilon^{3/2}\mathcal{E}_{s}^{3/2})\\
=&\frac{\epsilon^{2}(\ell+1/2)}{2}\frac{d}{dt}\int \frac{-q^{2}(-\phi_{c}+\partial_{x}^{2}\phi_{c})(2\phi_{1}-\epsilon\mathcal{G})}{2+\epsilon\phi_{1}}
 (\partial_{x}^{\ell}(R_{1}^{1}+R_{-1}^{1}))^{2}dx\\
&+\epsilon^{2}\mathcal{O}(\mathcal{E}_{s}+\epsilon^{3/2}\mathcal{E}_{s}^{3/2}).
\end{split}
\end{equation*}
By the equations \eqref{153}, \eqref{17}, \eqref{193}, \eqref{part1}, \eqref{part2}, \eqref{part3}, \eqref{part5} and \eqref{equu}, we have
\begin{equation*}
\begin{split}
I_{1}^{4}=&\epsilon^{2}\sum _{j_{1}\in\{\pm1\}}\Big[
\int\partial_{x}^{\ell}(\sum_{m=1}^{7}D_{m})_{j_{1}} B_{j_{1},j_{1}}^{1,1,4}(\phi_{c}, \partial_{x}^{\ell}R_{j_{1}}^{1})dx
+\ell\int\partial_{x}^{\ell}(\sum_{m=1}^{7}D_{m})_{j_{1}} B_{j_{1},j_{1}}^{1,1,4}(\partial_{x}\phi_{c}, \partial_{x}^{\ell-1}R_{j_{1}}^{1})dx\\
&+\int\partial_{x}^{\ell}R_{j_{1}}^{1} B_{j_{1},j_{1}}^{1,1,4}\big(\phi_{c}, \partial_{x}^{\ell}(\sum_{m=1}^{7}D_{m})_{j_{1}}\big)dx
+\ell\int\partial_{x}^{\ell}R_{j_{1}}^{1} B_{j_{1},j_{1}}^{1,1,4}\big(\partial_{x}\phi_{c}, \partial_{x}^{\ell-1}(\sum_{m=1}^{7}D_{m})_{j_{1}}\big)dx\Big]
+\epsilon^{2}\mathcal{O}(\mathcal{E}_{s})\\
=&\epsilon^{2}\sum _{j_{1}\in\{\pm1\}}\Big[
\int\partial_{x}^{\ell}(\sum_{m=1}^{7}D_{m})_{j_{1}} S_{j_{1},j_{1}}^{4}(\partial_{x}\phi_{c}, \partial_{x}^{\ell}R_{j_{1}}^{1})dx
+2\ell\int\partial_{x}^{\ell}(\sum_{m=1}^{7}D_{m})_{j_{1}} B_{j_{1},j_{1}}^{1,1,4}(\partial_{x}\phi_{c}, \partial_{x}^{\ell-1}R_{j_{1}}^{1})dx\Big]\\
=&\epsilon^{2}(2\ell+1)\sum _{j_{1}\in\{\pm1\}}j_{1}\int\partial_{x}^{\ell}(\sum_{m=1}^{7}D_{m})_{j_{1}}G_{j_{1},j_{1}}\partial_{x}\phi_{c} \partial_{x}^{\ell}R_{j_{1}}^{1}dx
+\epsilon^{2}\mathcal{O}(\mathcal{E}_{s})\\
=&\frac{\epsilon^{2}(\ell+1/2)}{2}\frac{d}{dt}\int \frac{-(-\phi_{c}+\partial_{x}^{2}\phi_{c})(2\phi_{1}-\epsilon\mathcal{G})}{2+\epsilon\phi_{1}}
 \big(\partial_{x}^{\ell}(R_{1}^{1}+R_{-1}^{1})\big)^{2}dx\\
&+\epsilon^{2}\mathcal{O}(\mathcal{E}_{s}+\epsilon^{3/2}\mathcal{E}_{s}^{3/2}).
\end{split}
\end{equation*}

For $I_{1}^{5}$, recall
\begin{equation*}
\begin{split}
b_{j_{1},j_{2}}^{1,1,5}=&\frac{-\frac{1}{2}k\langle k\rangle^{-2}\langle k-m\rangle^{-2}\langle m\rangle^{-2}}{-j_{1}\omega(k)-\omega(k-m)+j_{2}\omega(m)}
 \ \ for\ |k|>\delta, \ |m|>\delta.
\end{split}
\end{equation*}
By integration by parts, we have
\begin{equation*}
\begin{split}
I_{1}^{5}=&\epsilon^{2}\sum _{j_{1}\in\{\pm1\}}\Big[
\int\partial_{x}^{\ell}(\sum_{m=1}^{7}D_{m})_{j_{1}} \partial_{x}^{\ell}B_{j_{1},j_{1}}^{1,1,5}(\phi_{c}, R_{j_{1}}^{1})dx
+\int\partial_{x}^{\ell}R_{j_{1}}^{1} \partial_{x}^{\ell}B_{j_{1},j_{1}}^{1,1,5}\big(\phi_{c}, (\sum_{m=1}^{7}D_{m})_{j_{1}}\big)dx\Big]\\
=&\epsilon^{2}\mathcal{O}(\mathcal{E}_{s}+\epsilon^{3/2}\mathcal{E}_{s}^{3/2}).
\end{split}
\end{equation*}
Then we have the estimate
\begin{equation*}
\begin{split}
\sum_{n=1}^{5}I_{1}^{n}=&-\frac{\epsilon^{2}(\ell+1/2)}{2}\frac{d}{dt}\int \frac{2q^{2}(-\phi_{c}+\partial_{x}^{2}\phi_{c})(2\phi_{1}-\epsilon\mathcal{G})}{2+\epsilon\phi_{1}}
\big(\partial_{x}^{\ell}(R_{1}^{1}+R_{-1}^{1})\big)^{2}dx\\
&+\epsilon^{2}\mathcal{O}(\mathcal{E}_{s}+\epsilon^{3/2}\mathcal{E}_{s}^{3/2}).
\end{split}
\end{equation*}
Next, we estimate $I_{2}^{n}$.

For $I_{2}^{1}$, by \eqref{15}, \eqref{17} and Lemma \ref{L9}, we have
\begin{equation*}
\begin{split}
I_{2}^{1}=&\epsilon^{2}\sum _{j_{1}\in\{\pm1\}}\Big[
\int\partial_{x}^{\ell}(\sum_{m=1}^{7}D_{m})_{j_{1}} B_{j_{1},-j_{1}}^{1,1,1}(\phi_{c}, \partial_{x}^{\ell}R_{-j_{1}}^{1})dx
+\int\partial_{x}^{\ell}R_{j_{1}}^{1} B_{j_{1},-j_{1}}^{1,1,1}\big(\phi_{c}, \partial_{x}^{\ell}(\sum_{m=1}^{7}D_{m})_{-j_{1}}\big)dx\Big]\\
&+\epsilon^{2}\mathcal{O}(\mathcal{E}_{s}+\epsilon^{3/2}\mathcal{E}_{s}^{3/2})\\
=&\epsilon^{2}\sum _{j_{1}\in\{\pm1\}}\Big[
\int\partial_{x}^{\ell}(\sum_{m=1}^{7}D_{m})_{j_{1}} B_{j_{1},-j_{1}}^{1,1,1}(\phi_{c}, \partial_{x}^{\ell}R_{-j_{1}}^{1})dx
+\int\partial_{x}^{\ell}(\sum_{m=1}^{7}D_{m})_{-j_{1}} B_{-j_{1},j_{1}}^{1,1,1}(\phi_{c}, \partial_{x}^{\ell}R_{j_{1}}^{1})dx\Big]\\
&+\int\partial_{x}^{\ell}(\sum_{m=1}^{7}D_{m})_{-j_{1}} S_{-j_{1},j_{1}}^{1}(\phi_{c}, \partial_{x}^{\ell}R_{j_{1}}^{1})dx
+\epsilon^{2}\mathcal{O}(\mathcal{E}_{s}+\epsilon^{3/2}\mathcal{E}_{s}^{3/2})\\
=:&\sum_{i=1}^{7}I_{2i}^{1}+\epsilon^{2}\mathcal{O}(\mathcal{E}_{s}+\epsilon^{3/2}\mathcal{E}_{s}^{3/2}).
\end{split}
\end{equation*}

For $I_{21}^{1}$, we have
\begin{equation*}
\begin{split}
I_{21}^{1}=&\frac{\epsilon^{2}}{2}\sum _{j_{1}\in\{\pm1\}}j_{1}\int\partial_{x}^{\ell+1}\big(\phi_{1}q(R_{j_{1}}^{1}-R_{-j_{1}}^{1})\big)B^{1,1,1}_{j_{1},-j_{1}}(\phi_{c} , \partial_{x}^{\ell}R_{-j_{1}}^{1})dx\\
&-j_{1}\int\partial_{x}^{\ell+1}\big(\phi_{1}q(R_{j_{1}}^{1}-R_{-j_{1}}^{1})\big)B^{1,1,1}_{-j_{1},j_{1}}(\phi_{c} , \partial_{x}^{\ell}R_{j_{1}}^{1})dx\\
=&0.
\end{split}
\end{equation*}

For $I_{22}^{1}$, we have
\begin{equation*}
\begin{split}
I_{22}^{1}=&\frac{\epsilon^{2}}{2}\sum _{j_{1}\in\{\pm1\}}\int\partial_{x}^{\ell+1}\big(q\phi_{2}(R_{j_{1}}^{1}+R_{-j_{1}}^{1})\big)B^{1,1,1}_{j_{1},-j_{1}}(\phi_{c} , \partial_{x}^{\ell}R_{-j_{1}}^{1})dx\\
&+\int\partial_{x}^{\ell+1}\big(q\phi_{2}\big(R_{j_{1}}^{1}+R_{-j_{1}}^{1})\big)B^{1,1,1}_{-j_{1},j_{1}}(\phi_{c} , \partial_{x}^{\ell}R_{j_{1}}^{1})dx\\
=&\epsilon^{2}\sum _{j_{1}\in\{\pm1\}}\int\partial_{x}^{\ell+1}\big(q\phi_{2}(R_{j_{1}}^{1}+R_{-j_{1}}^{1})\big)B^{1,1,1}_{-j_{1},j_{1}}(\phi_{c} , \partial_{x}^{\ell}R_{j_{1}}^{1})dx\\
=&-\frac{\epsilon^{2}}{2}\sum _{j_{1}\in\{\pm1\}}\int\phi_{c}q\phi_{2}\partial_{x}^{\ell}R_{j_{1}}^{1}\partial_{x}^{\ell+1}R_{-j_{1}}^{1}dx
+\epsilon^{2}\mathcal{O}(\mathcal{E}_{s}+\epsilon^{3/2}\mathcal{E}_{s}^{3/2})\\
=&\epsilon^{2}\mathcal{O}(\mathcal{E}_{s}+\epsilon^{3/2}\mathcal{E}_{s}^{3/2}).
\end{split}
\end{equation*}

For $I_{23}^{1}$, we have
\begin{equation*}
\begin{split}
I_{23}^{1}=&\frac{\epsilon^{2}}{2}\sum _{j_{1}\in\{\pm1\}}\int\frac{1}{q}\partial_{x}^{\ell+1}(q\phi_{2}q\big(R_{j_{1}}^{1}-R_{-j_{1}}^{1})\big)B^{1,1,1}_{j_{1},-j_{1}}(\phi_{c} , \partial_{x}^{\ell}R_{-j_{1}}^{1})dx\\
&+\int\frac{1}{q}\partial_{x}^{\ell+1}\big(q\phi_{2}q(R_{j_{1}}^{1}-R_{-j_{1}}^{1})\big)B^{1,1,1}_{-j_{1},j_{1}}(\phi_{c} , \partial_{x}^{\ell}R_{j_{1}}^{1})dx\\
=&0.
\end{split}
\end{equation*}

For $I_{24}^{1}$, we have
\begin{equation*}
\begin{split}
I_{24}^{1}=&\frac{\epsilon^{2}}{2}\sum _{j_{1}\in\{\pm1\}}\Big[-j_{1}\int\frac{1}{q}\partial_{x}^{\ell+1}\big(\phi_{1}(R_{j_{1}}^{1}+R_{-j_{1}}^{1})\big)B^{1,1,1}_{j_{1},-j_{1}}(\phi_{c} , \partial_{x}^{\ell}R_{-j_{1}}^{1})dx\\
&+j_{1}\int\frac{1}{q}\partial_{x}^{\ell+1}\big(\phi_{1}(R_{j_{1}}^{1}+R_{-j_{1}}^{1})\big)B^{1,1,1}_{-j_{1},j_{1}}(\phi_{c} , \partial_{x}^{\ell}R_{j_{1}}^{1})dx\Big]\\
=&\epsilon^{2}\sum _{j_{1}\in\{\pm1\}}j_{1}\int\frac{1}{q}\partial_{x}^{\ell+1}\big(\phi_{1}(R_{j_{1}}^{1}+R_{-j_{1}}^{1})\big)B^{1,1,1}_{-j_{1},j_{1}}(\phi_{c} , \partial_{x}^{\ell}R_{j_{1}}^{1})dx\\
=&-\frac{\epsilon^{2}}{2}\sum _{j_{1}\in\{\pm1\}}j_{1}\int\phi_{c}\phi_{1}\partial_{x}^{\ell}R_{j_{1}}^{1}\partial_{x}^{\ell+1}R_{-j_{1}}^{1}dx
+\epsilon^{2}\mathcal{O}(\mathcal{E}_{s}+\epsilon^{3/2}\mathcal{E}_{s}^{3/2})\\
=&-\frac{\epsilon^{2}}{2}\int\phi_{c}\phi_{1}\partial_{x}^{\ell}(R_{1}^{1}+R_{-1}^{1})\partial_{x}^{\ell+1}(R_{1}^{1}-R_{-1}^{1})dx
+\epsilon^{2}\mathcal{O}(\mathcal{E}_{s}+\epsilon^{3/2}\mathcal{E}_{s}^{3/2})\\
=&-\frac{\epsilon^{2}}{2}\frac{d}{dt}\int\frac{\phi_{c}\phi_{1}}{2+\epsilon\phi_{1}}(\partial_{x}^{\ell}(R_{1}^{1}+R_{-1}^{1}))^{2}
+\epsilon^{2}\mathcal{O}(\mathcal{E}_{s}+\epsilon^{3/2}\mathcal{E}_{s}^{3/2}).
\end{split}
\end{equation*}

For $I_{25}^{1}$, we have
\begin{equation*}
\begin{split}
I_{25}^{1}=&\frac{\epsilon^{2}}{2}\sum _{j_{1}\in\{\pm1\}}-j_{1}\int\frac{1}{q}\partial_{x}^{\ell+1}\langle|\partial_{x}|\rangle^{-2}
\big(\langle|\partial_{x}|\rangle^{-2}\phi_{1}\langle|\partial_{x}|\rangle^{-2}(R_{j_{1}}^{1}+R_{-j_{1}}^{1})\big)B^{1,1,1}_{j_{1},-j_{1}}(\phi_{c} , \partial_{x}^{\ell}R_{-j_{1}}^{1})dx\\
&+j_{1}\int\frac{1}{q}\partial_{x}^{\ell+1}\langle|\partial_{x}|\rangle^{-2}
\big(\langle|\partial_{x}|\rangle^{-2}\phi_{1}\langle|\partial_{x}|\rangle^{-2}(R_{j_{1}}^{1}+R_{-j_{1}}^{1})\big)B^{1,1,1}_{-j_{1},j_{1}}(\phi_{c} , \partial_{x}^{\ell}R_{j_{1}}^{1})dx\\
=&\epsilon^{2}j_{1}\int\frac{1}{q}\partial_{x}^{\ell+1}\langle|\partial_{x}|\rangle^{-2}
\big(\langle|\partial_{x}|\rangle^{-2}\phi_{1}\langle|\partial_{x}|\rangle^{-2}(R_{j_{1}}^{1}+R_{-j_{1}}^{1})\big)B^{1,1,1}_{-j_{1},j_{1}}(\phi_{c} , \partial_{x}^{\ell}R_{j_{1}}^{1})dx\\
=&\epsilon^{2}\mathcal{O}(\mathcal{E}_{s}+\epsilon^{3/2}\mathcal{E}_{s}^{3/2}).
\end{split}
\end{equation*}

For $I_{26}^{1}$, we have
\begin{equation*}
\begin{split}
I_{26}^{1}=&\frac{\epsilon^{3}}{2}\sum _{j_{1}\in\{\pm1\}}j_{1}\int\frac{1}{q}\partial_{x}^{\ell+1}\big(\mathcal{G}(R_{j_{1}}^{1}+R_{-j_{1}}^{1})\big)B^{1,1,1}_{j_{1},-j_{1}}(\phi_{c} , \partial_{x}^{\ell}R_{-j_{1}}^{1})dx\\
&-j_{1}\int\frac{1}{q}\partial_{x}^{\ell+1}\big(\mathcal{G}(R_{j_{1}}^{1}+R_{-j_{1}}^{1})\big)B^{1,1,1}_{-j_{1},j_{1}}(\phi_{c} , \partial_{x}^{\ell}R_{j_{1}}^{1})dx\\
=&-\epsilon^{3}\sum _{j_{1}\in\{\pm1\}}j_{1}\int\frac{1}{q}\partial_{x}^{\ell+1}\big(\mathcal{G}(R_{j_{1}}^{1}+R_{-j_{1}}^{1})\big)B^{1,1,1}_{-j_{1},j_{1}}(\phi_{c} , \partial_{x}^{\ell}R_{j_{1}}^{1})dx\\
=&-\frac{\epsilon^{3}}{2}\int\phi_{c}\mathcal{G}\partial_{x}^{\ell}(R_{1}^{1}+R_{-1}^{1})\partial_{x}^{\ell+1}(R_{1}^{1}-R_{-1}^{1})dx
+\epsilon^{2}\mathcal{O}(\mathcal{E}_{s}+\epsilon^{3/2}\mathcal{E}_{s}^{3/2})\\
=&-\frac{\epsilon^{3}}{2}\frac{d}{dt}\int\frac{\phi_{c}\mathcal{G}}{2+\epsilon\phi_{1}}(\partial_{x}^{\ell}(R_{1}^{1}+R_{-1}^{1}))^{2}dx
+\epsilon^{2}\mathcal{O}(\mathcal{E}_{s}+\epsilon^{3/2}\mathcal{E}_{s}^{3/2}).
\end{split}
\end{equation*}
Similar to $I_{25}^{1}$, we have $I_{27}^{1}=\epsilon^{2}\mathcal{O}(\mathcal{E}_{s}+\epsilon^{3/2}\mathcal{E}_{s}^{3/2})$.

Then, we have
\begin{equation*}
\begin{split}
I_{2}^{1}
=-\frac{\epsilon^{2}}{2}\frac{d}{dt}\int\frac{\phi_{c}(\phi_{1}+\epsilon\mathcal{G})}{2+\epsilon\phi_{1}}\big(\partial_{x}^{\ell}(R_{1}^{1}+R_{-1}^{1})\big)^{2}dx
+\epsilon^{2}\mathcal{O}(\mathcal{E}_{s}+\epsilon^{3/2}\mathcal{E}_{s}^{3/2}).
\end{split}
\end{equation*}
Similarly,  by \eqref{151}, \eqref{18} and Lemma \ref{L9}, we have
\begin{equation*}
\begin{split}
I_{2}^{2}
=\frac{\epsilon^{2}}{2}\frac{d}{dt}\int\frac{\phi_{c}q\phi_{2}}{2+\epsilon\phi_{1}}\big(\partial_{x}^{\ell}(R_{1}^{1}+R_{-1}^{1})\big)^{2}dx
+\epsilon^{2}\mathcal{O}(\mathcal{E}_{s}+\epsilon^{3/2}\mathcal{E}_{s}^{3/2}).
\end{split}
\end{equation*}
By \eqref{152}, \eqref{18} and Lemma \ref{L9}, we have
\begin{equation*}
\begin{split}
I_{2}^{3}
=\frac{\epsilon^{2}}{2}\frac{d}{dt}\int\frac{\phi_{c}q\phi_{2}}{2+\epsilon\phi_{1}}\big(\partial_{x}^{\ell}(R_{1}^{1}+R_{-1}^{1})\big)^{2}dx
+\epsilon^{2}\mathcal{O}(\mathcal{E}_{s}+\epsilon^{3/2}\mathcal{E}_{s}^{3/2}).
\end{split}
\end{equation*}
By \eqref{153}, \eqref{17} and Lemma \ref{L9}, we have
\begin{equation*}
\begin{split}
I_{2}^{4}&=-\frac{\epsilon^{2}}{2}\frac{d}{dt}\int\frac{\phi_{c}(\phi_{2} -\epsilon\mathcal{G})}{2+\epsilon\phi_{1}}\big(\partial_{x}^{\ell}(R_{1}^{1}+R_{-1}^{1})\big)^{2}dx+\epsilon^{2}\mathcal{O}(\mathcal{E}_{s}+\epsilon^{3/2}\mathcal{E}_{s}^{3/2}),\\
I_{2}^{5}&=\epsilon^{2}\mathcal{O}(\mathcal{E}_{s}+\epsilon^{3/2}\mathcal{E}_{s}^{3/2}).
\end{split}
\end{equation*}
Then we have
\begin{equation*}
\begin{split}
\sum_{n=1}^{5}I_{2}^{n}
=-\frac{\epsilon^{2}}{2}\frac{d}{dt}\int\frac{\phi_{c}(\phi_{1}+\phi_{2} +2q\phi_{2})}{2+\epsilon\phi_{1}}\big(\partial_{x}^{\ell}(R_{1}^{1}+R_{-1}^{1})\big)^{2}dx
+\epsilon^{2}\mathcal{O}(\mathcal{E}_{s}+\epsilon^{3/2}\mathcal{E}_{s}^{3/2}).
\end{split}
\end{equation*}
Recall \eqref{equ11-1}, similarly, by the Lemma \ref{L9}, we have
\begin{equation*}
\begin{split}
I_{3}
=\frac{\epsilon^{2}}{2}\frac{d}{dt}\int\frac{\mathcal{G}}{2+\epsilon\phi_{1}}\big(\partial_{x}^{\ell}(R_{1}^{1}+R_{-1}^{1})\big)^{2}dx
+\epsilon^{2}\mathcal{O}(\mathcal{E}_{s}+\epsilon^{3/2}\mathcal{E}_{s}^{3/2}).
\end{split}
\end{equation*}
Hence, we define the modified energy
\begin{equation*}
\begin{split}
\widetilde{\mathcal{E}_{s}}=\mathcal{E}_{s}+\frac{\epsilon^{2}}{4}\sum_{\ell=1}^{s}h_{\ell},
\end{split}
\end{equation*}
with
\begin{equation*}
\begin{split}
h_{\ell}=\int_{\mathbb{R}}&\big((2\ell+1)q^{2}(-\phi_{c}+\partial_{x}^{2}\phi_{c})(2\phi_{1}-\epsilon\mathcal{G})
+\phi_{c}(\phi_{1}+\phi_{2}-2q\phi_{2})-\mathcal{G}\big)\big(\partial_{x}^{\ell}(R_{1}^{1}+R_{-1}^{1})\big)^{2}dx,
\end{split}
\end{equation*}
to obtain
\begin{equation*}
\begin{split}
\partial_{t}\widetilde{\mathcal{E}_{s}}\lesssim\epsilon^{2}(\widetilde{\mathcal{E}_{s}}
+\epsilon^{1/2}\widetilde{\mathcal{E}_{s}}^{3/2}+\epsilon\widetilde{\mathcal{E}_{s}}^{2}+1).
\end{split}
\end{equation*}
Consequently, Gronwall's inequality yields the $\mathcal{O}(1)$ boundedness of $\widetilde{\mathcal{E}_{s}}$ for all $t\in[0,T_{0}/\epsilon^{2}]$, for sufficiently small $\epsilon>0$. Theorem \ref{Thm1} then follows thanks to the fact that $\|\mathcal{R}^{0}+R^{1}\|_{H^{s}}\lesssim\sqrt{\widetilde{\mathcal{E}_{s}}}$ for sufficiently small $\epsilon>0$ and estimate \eqref{equation12-2} .

\bigskip
\paragraph
{\bf Aknowledgement.} The second author thanks Professor Yan Guo for his valuable discussions and suggestions on this paper.



\end{document}